\documentclass[12pt]{amsart}
\usepackage{amscd,amsmath,amssymb,amsfonts} 
\usepackage[cmtip, all]{xy}

\newtheorem{thm}{Theorem}[section]
\newtheorem{prop}[thm]{Proposition}
\newtheorem{lem}[thm]{Lemma}
\newtheorem{cor}[thm]{Corollary}
\newtheorem{conj}[thm]{Conjecture}
\newtheorem{ques}[thm]{Question}
\newtheorem*{claim}{Claim}
 
\theoremstyle{definition}
\newtheorem{exm}[thm]{Example}
\newtheorem{defn}[thm]{Definition}

\theoremstyle{remark}
\newtheorem{remk}[thm]{Remark}
\newtheorem{remks}[thm]{Remarks}
\newtheorem{exms}[thm]{Examples}
\newtheorem{notat}[thm]{Notation}
\newtheorem*{ack}{Acknowledgments}
\numberwithin{equation}{section}

{\hfill$\square$\end{defn}}

{\hfill$\square$\end{remk}}

{\hfill$\square$\end{remks}}

{\hfill$\square$\end{exm}}

{\hfill$\square$\end{exms}}

{\hfill$\square$\end{notat}}

\newcommand{\thmref}{Theorem~\ref}
\newcommand{\propref}{Proposition~\ref}

\newcommand{\defref}{Definition~\ref}
\newcommand{\lemref}{Lemma~\ref}

\newcommand{\conjref}{Conjecture~\ref}

\newcommand{\sB}{{\mathcal B}}

\newcommand{\sI}{{\mathcal I}}

\newcommand{\sO}{{\mathcal O}}

\newcommand{\sW}{{\mathcal W}}

\newcommand{\sZ}{{\mathcal Z}}
\newcommand{\wt}{\widetilde}
\newcommand{\wh}{\widehat}

\newcommand{\A}{{\mathbb A}}

\newcommand{\G}{{\mathbb G}}

\renewcommand{\P}{{\mathbb P}}

\newcommand{\W}{{\mathbb W}}

\newcommand{\Z}{{\mathbb Z}}

\newcommand{\CH}{{\rm CH}}

\newcommand{\inj}{\hookrightarrow}

\newcommand{\codim}{{\rm codim}}

\newcommand{\Spec}{{\rm Spec \,}}

\newcommand{\id}{{\operatorname{id}}}

\newcommand{\Sch}{{\operatorname{\mathbf{Sch}}}}

\renewcommand{\max}{{\operatorname{\rm max}}}

\newcommand{\Sm}{{\mathbf{Sm}}}
\newcommand{\SmProj}{{\mathbf{SmProj}}}

\newcommand{\Ab}{{\mathbf{Ab}}}

\newcommand{\ds}{{/\kern-3pt/}}

\renewcommand{\log}{{\operatorname{log}}}

\newcommand{\Proj}{{\operatorname{Proj}}}

\newcommand{\sgn}{{\operatorname{\rm sgn}}}

\newcommand{\TZ}{{\operatorname{Tz}}}
\renewcommand{\TH}{{\operatorname{TCH}}}
\renewcommand{\CH}{{\operatorname{CCH}}}
\newcommand{\un}{\underline}
\newcommand{\ov}{\overline}

\newcommand{\dgn}{{\operatorname{degn}}}
\renewcommand{\dim}{\text{dim}}
\newcommand{\tuborg}{\left\{\begin{array}{ll}}
\newcommand{\sluttuborg}{\end{array}\right.}

\begin{document}
\title{Moving lemma for additive Chow groups and applications}
\author{Amalendu Krishna, Jinhyun Park}
\address{School of Mathematics, Tata Institute of Fundamental Research,  
Homi Bhabha Road, Colaba, Mumbai, India}
\email{amal@math.tifr.res.in}
\address{Department of Mathematical Sciences, KAIST,  
Daejeon, 305-701, Republic of Korea (South)}
\email{jinhyun@mathsci.kaist.ac.kr; jinhyun@kaist.edu}

\baselineskip=10pt 
  

\begin{abstract}
We study additive higher Chow groups with several modulus conditions.
Apart from exhibiting the validity of all known results for the
additive Chow groups with these
modulus conditions, we prove the moving lemma for them:
for a smooth projective variety $X$ and a finite collection $\mathcal{W}$ of 
its locally closed algebraic subsets, every additive higher Chow cycle is 
congruent to an admissible cycle intersecting properly all members of 
$\mathcal{W}$ times faces. This is the additive analogue of the moving lemma 
for the higher Chow groups studied by S. Bloch and M. Levine. 

As applications, we show that any map from a quasi-projective variety
to a smooth projective variety induces a pull-back map of additive higher
Chow groups. Using the moving lemma, we also establish the structure of
graded-commutative differential graded algebra (CDGA) on the additive
higher Chow groups.
\end{abstract}

\maketitle

\tableofcontents

\section{Introduction}
Working with algebraic cycles, formal finite sums of closed subvarieties of a 
variety, often requires some forms of moving results, as differential geometry
often requires Sard's lemma. A classical example is Chow's moving lemma in 
\cite{Chow} that moves algebraic cycles under rational equivalence. A modern 
one for higher Chow groups shows in \cite{Bl2,Le} that, for a smooth 
quasi-projective variety $X$ and a finite set of locally closed subvarieties 
of $X$, one can move (modulo boundaries) admissible cycles to other 
admissible cycles that intersect a given finite set of subvarieties 
in the right codimensions. Any such results on moving of cycles is generally 
referred to as moving lemmas. Such moving results have played a very crucial
role in the development and applications of the theory of higher Chow groups.
The primary goal of this paper is to prove this 
latter kind of moving lemma for the additive higher Chow groups of a smooth 
and projective variety and to study some very important applications on the 
structural properties of additive higher Chow groups.   

The additive Chow groups of zero cycles on a field were first introduced by 
Bloch and Esnault in \cite{BE1} in an attempt to describe the $K$-theory 
and motivic cohomology of the ring of dual numbers via algebraic cycles.
Bloch and Esnault \cite{BE2} later defined these groups by putting a 
{\sl modulus} condition on the additive Chow cycles in the hope of describing 
the $K$-groups of any given truncated polynomial ring over a field. The 
additive higher Chow groups of any given variety were defined in the most 
general form by Park in \cite{P1} and were later studied in more detail in 
\cite{KL}, where many of the expected properties of these groups were also 
established. 

The most crucial part of the existing definition(s) of the additive higher 
Chow groups which makes them distinct from the higher Chow groups, is the
{\sl modulus} condition on the admissible additive cycles. This condition
also brings the extra subtlety which does not persist with the higher Chow
groups. As conjectured in \cite{KL,P1}, the additive higher Chow groups are
expected to complement higher Chow groups for non-reduced schemes so as to 
obtain the right motivic cohomology groups. In particular, for a smooth
projective variety $X$, one expects a Atiyah-Hirzebruch spectral sequence
\begin{equation}\label{eqn:SS}
{\TH}^{-q}(X, -p-q; m) \Rightarrow K^{\rm nil}_{-p-q}(X),
\end{equation}
where $K^{\rm nil}$ is the homotopy fiber of the restriction map
$K(X \times {\rm Spec}(k[t])) \to K(X \times {\rm Spec}(k[t]/{t^{m+1}}))$.

Since these beliefs are still conjectural, it is not clear if the
modulus conditions used to study the additive higher Chow groups of varieties
in the literature are the right ones which would give the correct motivic 
cohomology, {\sl e.g.}, the one which would satisfy ~\eqref{eqn:SS}. 
One aspect of this paper is to exhibit that the modulus condition (which we
call $M_{sup}$ in this paper) used in \cite{KL} may not be the best 
possible one, if one expects the theory of additive higher Chow groups to 
yield the correct motivic cohomology and the motives of non-reduced schemes.  

We study the theory of additive Chow groups based on two other modulus 
conditions in this paper: $M = M_{sum}$ is based on the modulus condition
used by Bloch-Esnault-R{\"u}lling in \cite{BE2,R}, and $M = M_{ssup}$ is
a new modulus condition introduced in this paper. Although this new
modulus condition $M_{ssup}$ may appear mildly stronger than
that used in \cite{KL, P1}, it turns out that the resulting 
additive Chow groups have all the properties known for the additive Chow groups of \cite{BE2}, \cite{KL}, and \cite{P1}. In addition, we prove many other crucial structural 
properties of the additive higher Chow groups based on the modulus conditions
$M_{sum}$ and $M_{ssup}$. Although it may seem surprising, the techniques
used in proving the results of this paper make one believe that such results
may not be possible for the additive higher Chow groups based on the modulus 
condition $M_{sup}$ of \cite{KL, P1}, if Conjecture~\ref{conj:QI} turns out to be false.  

We now outline the structure of this paper and elaborate on our main results.
We define our basic objects, the additive higher Chow groups with various
modulus conditions, in Section~\ref{section:AHCG}. We also prove some
preliminary results which are used repeatedly in the paper.
In Section~\ref{section:BHCG}, we prove the basic properties of these
additive Chow groups. In particular, we demonstrate all those results for
the additive higher Chow groups based on the modulus condition $M_{ssup}$,
which are known for the additive higher Chow groups of \cite{BE2}, \cite{KL} 
and \cite{P1}. Section~\ref{section:ML} gives the proofs of further 
preliminary results needed to prove our moving lemma for the additive higher
Chow groups. 

The subsequent Sections~\ref{section:MLP} and 
~\ref{section:moving for varieties} are devoted to our first main result,
the moving lemma for additive higher Chow groups.   
As in the case of higher Chow groups, any theory of additive motivic 
cohomology which would compute the $K$-theory as in ~\eqref{eqn:SS} is expected 
to have a form of moving lemma to make them more amenable to deeper study.
This was one of the primary motivation for working on this paper.
We show in Theorem~\ref{thm:ML} that for a smooth projective variety $X$ and 
a finite collection $\mathcal{W}$ of its locally closed algebraic subsets, 
every additive higher Chow cycle is congruent to an admissible cycle 
intersecting properly all members of $\mathcal{W}$ times faces. This is the 
additive analogue of the moving lemma for the higher Chow groups studied by 
S. Bloch and M. Levine. 

While lack of this result for general quasi-projective varieties may seem 
disappointing, one would rather not expect this to be the case. For instance, 
${\A}^1$-homotopy invariance and localization sequences fail for 
additive higher Chow groups, but these are indirectly implied if the moving 
lemma is assumed for all quasi-projective varieties such as 
$X \times {\A}^1$. A concrete quasi-projective example, where the standard 
arguments fail, is given in Example~\ref{exm:first example}.

Our proof of the above result is broadly speaking based on the techniques of
\cite{Bl1} and \cite{Le} that prove the analogous result for the higher Chow
groups. The main difficulty with the techniques of higher Chow groups which
does not immediately allow them to be adapted into the additive world is that these
arguments are mostly intersection theoretic and are not equipped to
handle the most delicate modulus condition of additive Chow cycles. We show 
one such phenomenon in the last section of this paper. However, in the case of 
projective varieties, we carefully modify the arguments at every step so that we
can keep track of the modulus condition whenever we encounter new cycles
in the process, especially in the construction of the chain homotopy 
variety. This is achieved using our new {\sl containment} sort of
argument and some results of \cite{KL}. This essentially proves the above
theorem. On the log-additive higher Chow groups, one can prove the moving 
lemma for any general smooth quasi-projective varieties using our main theorem.

In Section~\ref{section:CF}, we give our first application of the moving 
lemma. We establish the contravariant functoriality property 
of the additive higher Chow groups in the most general form 
by showing in Theorem~\ref{thm:funct} that for a morphism
$f: X \to Y$ of quasi-projective varieties over a field $k$, where
$Y$ is smooth and projective, there is a pull-back map
$f^*: {\TH}^q(Y, n; m) \to {\TH}^q(X, n; m)$, and this satisfies the
expected composition law.

If $X$ is also smooth and projective, the pull-back map on the additive Chow
groups was constructed in \cite{KL} using the action of higher Chow groups on
the additive ones. However, the contravariance functoriality in this 
general form as above (even if $Y$ is smooth) is new and is based on a 
crucial use of Theorem~\ref{thm:ML} as is the case of general pull-back maps
of higher Chow groups ({\sl cf.} \cite[Theorem~4.1]{Bl1}), and another
use of our containment argument. Even in the special case of $X$ being smooth
and projective, our proof is different and more direct than the one in
\cite{KL}.

Our final set of main results of this paper are motivated by the question of 
what are the important and necessary properties one would expect the additive 
higher Chow groups to satisfy, if they are the right motivic cohomology to 
compute the nil $K$-theory of the infinitesimal deformations of smooth schemes.
In particular, one could ask what are the necessary implications on the
structural properties of the additive higher Chow groups if there is indeed
a spectral sequence as in ~\eqref{eqn:SS}. The reader would recall in this
context that the $K$-theory of the infinitesimal deformations of the base
field $k$ is expressed in terms of the modules of absolute K{\"a}hler
differentials and the absolute de Rham-Witt complex of Hesselholt-Madsen,
as shown for example in \cite{Hesselholt}. For general
smooth projective varieties over $k$, one expects these $K$-groups to be
given by the cohomology of the absolute de Rham-Witt complex. 
It is well known that these de Rham-Witt complexes have the structure
of a graded-commutative differential graded algebra (CDGA) and they
are initial objects in the category of so called Witt complexes over a
base scheme ({\sl cf.} \cite{HeMa, R}). This makes it imperative that 
the additive higher Chow groups posses such structures. Our next set of
results together implies that this is indeed the case. 

In Section~\ref{section:WD}, we show in Theorem~\ref{thm:wedge-product*}
that the additive higher Chow groups indeed have a wedge product
which makes the direct sum of all additive higher Chow groups a graded-commutative
algebra. This product structure has all the functoriality properties and
satisfies the projection formula. Furthermore, this is compatible with the
module structure on the additive higher Chow groups over the Chow ring
of the variety. We show in Theorem~\ref{thm:preDGA} that these are 
equipped with a differential operator, too. 

For the modulus condition $M_{sum}$, we further show in 
Theorem~\ref{thm:DGAsum} that these differential 
operators and the wedge product turn the resulting additive higher Chow groups
into a CDGA. As a very important
ingredient needed to achieve this, we introduce the normalized version
of additive cycle complex and additive higher Chow groups in 
Section~\ref{section:NACom}. We show that this normalized additive
cycle complex is in fact quasi-isomorphic to the additive cycle complex
if one uses the modulus condition $M_{sum}$. This is an additive analog
of a similar result of S. Bloch \cite{Bl3} for higher Chow groups. 
This is the last result we need to complete the program of CDGA structure on 
additive higher Chow groups.

Although we do not
go into this in order not to further increase the length of this seemingly long 
paper, the reader can easily see using our techniques that
the push-forward and pull-back maps given by the power map $a \mapsto a^r$ on $\G_m$ (which is finite and flat) induce on the additive higher Chow groups of smooth projective varieties two operators, so called the {\sl Frobenius} and the {\sl Verschiebung} 
operators. Together with the above CDGA structure, this turns them into a 
Witt complex.   

One hopes that this very general abstract structure of a Witt complex
will help us in making a significant progress towards the eventual goal 
of showing that the additive higher Chow groups are the right motivic 
cohomology of the infinitesimal deformations of smooth schemes. This goal was 
in some sense the starting point of the theory of additive higher 
groups. 

In the last section, we append some calculations of the additive higher Chow 
groups the authors found in the process of working on the problem. This 
suggests some kind of ``pseudo"-${\A}^1$-homotopy properties of additive 
higher Chow groups. 

We finally remark that the only reason for not including the modulus condition
$M_{ssup}$ in our Theorem~\ref{thm:DGAsum} is the lack of an
affirmative answer of Question~\ref{ques:QI} in this case.  
We strongly believe the answer to be indeed positive and hope that a proof 
will be available soon.    
 
Throughout this paper, a {\sl $k$-scheme}, or a {\sl scheme over $k$}, is 
always a separated scheme of finite type over a perfect field $k$. A \emph{$k$-variety} is an integral $k$-scheme. The ground field $k$ will be fixed throughout this paper.

\section{Additive higher Chow groups}\label{section:AHCG} 
 In this section, we define additive higher Chow groups from a unifying 
perspective than those in the literature by Bloch-Esnault, 
R\"ulling, Krishna-Levine, and Park, treating the modulus conditions as 
``variables''. We also prove some elementary results that are needed to
study and compare the additive Chow groups based on various modulus 
conditions. 

We begin by fixing some notations which will be used throughout this
paper. We write $\Sch/k$, $\Sm/k$ and $\SmProj/k$ for the
categories of $k$-schemes, smooth quasi-projective varieties, and smooth
projective varieties, respectively. $D^{-}(\Ab)$ is the derived category
of bounded above complexes of abelian groups.
Recall from \cite{KL,P1} that for a normal variety $X$ over $k$, and
a finite set of Weil divisors $\{Y_1, \cdots , Y_s\}$ on $X$, the 
supremum of these divisors, denoted by $\sup_{1 \le i \le s} Y_i$,
is the Weil divisor defined to be 
\begin{equation}\label{eqn:sup}
\sup_{1 \le i \le s} Y_i = \sum_{Y \in Pdiv(X)}
(\underset{1 \le i \le s}{\max} \  {\rm ord}_Y (Y_i))[Y],
\end{equation}
where $Pdiv(X)$ is the set of all prime Weil divisors of $X$.  
One observes that the set of all Cartier divisors on a normal scheme
$X$ is contained in the set of all Weil divisors, and the supremum of
a collection of Cartier divisors may not remain a Cartier divisor in general,
unless $X$ is factorial. We shall need some elementary results about
Cartier and Weil divisors on normal varieties.

Here are some basic facts about divisors on normal varieties:

\begin{lem}\label{lem:rest}
Let $X$ be a normal variety and let $D_1$ and $D_2$ be effective Cartier
divisors on $X$ such that $D_1 \ge D_2$ as Weil divisors. Let $Y \subset X$ 
be a closed subset which intersects $D_1$ and $D_2$ properly. Let 
$f: Y^N \to X$ be the composite of the inclusion and the normalization of 
$Y_{\rm red}$. Then $f^*(D_1) \ge f^*(D_2)$.
\end{lem}
\begin{proof} For any effective Cartier divisor $D$ on $X$, let ${\sI}_{D}$
denote the sheaf of ideals defining $D$ as a locally principal closed
subscheme of $X$. We first claim that $D_1 \ge D_2$ if and only if 
${\sI}_{D_1} \subset {\sI}_{D_2}$. We only need to show the
only if part as the other implication is obvious. Now, $D_1 \ge D_2$ 
implies that $D = D_1 - D_2$ is effective as a Cartier divisor since the 
group of Cartier divisors forms a subgroup of Weil divisors on a normal scheme.
Since ${\sI}_{D_1} \subset {\sI}_{D_2}$ is a local question, we can assume
that $X = {\rm Spec}(A)$ is local normal integral scheme and 
${\sI}_{D_i} = (a_i)$. Put $a = {a_1}/{a_2}$ as an element of the function
field of $X$. We need to show that $a \in A$. Since $A$ is normal, it
suffices to show that $a \in A_{\mathfrak p}$ for every height one prime ideal
$\mathfrak p$ of $A$. But this is precisely the meaning of $D_1 \ge D_2$.
This proves the claim.

Since $D_i$ intersect $Y$ properly, we see that $f^*(D_i)$ is a locally
principal closed subscheme of $Y^N$ for $i =1, 2$. The lemma now follows
directly from the above claim.
\end{proof}  
The following is a refinement of \cite[Lemma~3.2]{KL}:
\begin{lem}\label{lem:surjm}
Let $f : Y \to X$ be a surjective map of normal integral $k$-schemes. Let 
$D$ be a Cartier divisor on $X$ such that $f^*(D) \ge 0$ on $Y$. Then 
$D \ge 0$ on $X$.
\end{lem}
\begin{proof} As is implicit in the proof of the \lemref{lem:rest}, we can localize
at the generic points of ${\rm Supp}(D)$ and assume that $X = {\rm Spec}(A)$,
where $A$ is a dvr which is essentially of finite type over $k$. The
divisor $D$ is then given by a rational function $a \in K$, where $K$ is the
field of fractions of $A$. Choosing a uniformizing parameter $\pi$ of $A$,
we can write $a$ uniquely as $a = u {\pi}^{n}$, where $u \in A^{\times}$
and $n \in \Z$. 

Since $f$ is surjective, there is a closed point $y \in Y$ such that $f(y)$
is the closed point of $X$. Since $Y$ is integral, the surjectivity of $f$
also implies that the generic point of $Y$ (which is also the generic point of
${\rm Spec}({\sO}_{Y,y})$) must go to the generic point of $X$ under $f$.
Hence the map ${\rm Spec}({\sO}_{Y,y}) \to X$ is surjective. 
This implies in particular that the image of $\pi$ in 
${\sO}_{Y,y}$ is a non-zero element of the maximal ideal $\mathfrak m$ of the 
local ring ${\sO}_{Y,y}$. On the other hand, $f^*(D) \ge 0$ implies that
as a rational function on $Y$, $a$ actually lies in ${\sO}_{Y,y}$. Since 
$u \in {\sO}^{\times}_{Y,y}$ and $\pi \in \mathfrak m$, this can happen
only when $n \ge 0$. That is, $D$ is effective.
\end{proof}    

We assume a $k$-scheme $X$ is equi-dimensional in this paper to define the
additive Chow groups although one can 
easily remove this condition by writing the additive Chow cycles in terms of
their dimensions rather than their codimensions. Throughout this paper, for any such scheme $X$, 
we shall denote the normalization of $X_{\rm red}$ by $X^N$. 
Thus $X^N$ is the disjoint union of the normalizations of all the 
irreducible components of $X_{\rm red}$.

Set $\A^1:=\Spec k[t]$,  
$\G_m:=\Spec k[t,t^{-1}]$, $\P^1:=\Proj\, k[Y_0,Y_1]$ 
and let $y:=Y_1/Y_0$ be the standard
coordinate function on $\P^1$. We set $\square^n:=(\P^1\setminus\{1\})^n$. 
For $n \ge 1$, let $B_n = {\G}_m \times \square^{n-1}$,
${\wt B}_n = \A^1 \times \square^{n-1}$, 
$\ov{B}_n = \A^1 \times ({\P}^{1})^{n-1}\supset {\wt B}_n$ and
$\widehat{B}_n = {\P}^{1} \times ({\P}^{1})^{n-1} \supset \ov{B}_n$. 
We use the coordinate system 
$(t, y_1, \cdots , y_{n-1})$ on $\widehat{B}_n$, with $y_i:=y\circ q_i$, where $
q_i: {\widehat B}_n\to \P^1$ is the projection onto the $i$-th $\P^1$.

Let $F^1_{n,i}$, for $i=1,\ldots, n-1$, be the Cartier divisor
on $\widehat{B}_n$ defined by $\{y_i=1\}$ and $F_{n,0}\subset \widehat{B}_n$
the Cartier divisor defined by $\{t=0\}$. Notice that the divisor $F_{n,0}$ is
in fact contained in ${\ov B}_n \subset {\widehat B}_n$. Let $F^1_n$ denote
the Cartier divisor $\sum_{i=1} ^{n-1} F^1_{n,i}$ on
$\widehat{B}_n$. 
 
 A {\em face} of $B_n$ is a subscheme $F$ defined by equations of the form
 \[
 y_{i_1}=\epsilon_1, \ldots,  y_{i_s}=\epsilon_s;\ \epsilon_j\in\{0,\infty\}.
 \]

For $\epsilon=0, \infty$,  and $i = 1, \cdots, n-1$, let 
\[
\iota_{n,i,\epsilon}: B_{n-1}\to B_n
\]
be the inclusion
\begin{equation}\label{eqn:face}
\iota_{n,i,\epsilon}(t,y_1,\ldots, y_{n-2})=(t,y_1,\ldots, y_{i-1}, 
\epsilon,y_i,\ldots, y_{n-2}).
\end{equation}
We now define the modulus conditions that we shall consider for defining our
additive higher Chow groups.
\subsection{Modulus conditions}\label{subsection:MD}
\begin{defn}\label{defn:M5}
Let $X$ be a $k$-scheme as above and let $V$ be an integral closed subscheme 
of $X \times B_n$. Let $\ov V$ denote the closure of $V$ in $X \times
{\wh B}_n$ and let $\nu : {\ov V}^N \to  X \times {\widehat B}_n$ denote
the induced map from the normalization of $\ov V$. We fix an integer $m \ge 1$. 
\begin{enumerate}
\item
We say that $V$ satisfies the modulus $m$ condition $M_{sum}$ (or the 
\emph{sum}-modulus condition) on $X \times B_n$ if as Weil divisors on 
${\ov V}^N$,
\[
(m+1)[{\nu}^*(F_{n,0})] \le [\nu^*(F^1_n)].
\] 
This condition was used by Bloch-Esnault and R\"ulling in \cite{Bl1,R} to
study additive Chow groups of zero cycles on fields.
\item 
We say that $V$ satisfies the modulus $m$ condition $M_{sup}$ (or the 
\emph{sup}-modulus condition) on $X \times B_n$ if as Weil divisors on 
${\ov V}^N$,
\[
(m+1)[{\nu}^*(F_{n,0})] \le \sup_{1 \leq i \leq n-1} [\nu^*(F^1_{n,i})].
\]
This condition was used by Park and Krishna-Levine in \cite{KL,P1} to define
their additive higher Chow groups. 
\item  
We say that $V$ satisfies the modulus $m$ condition $M_{ssup}$ (or the 
\emph{strong sup}-modulus condition) on $X \times B_n$ if there exists an 
integer $1 \le i \le n-1$ such that 
\[
(m+1)[{\nu}^*(F_{n,0})] \le [{\nu}^*(F^1_{n,i})]
\]
as Weil divisors on ${\ov V}^N$.
\end{enumerate}
\end{defn}
Since the modulus conditions are defined for a given fixed integer $m$, we shall
often simply say that $V$ satisfies a modulus condition $M$ without
mentioning the integer $m$. Notice that since $V$ is contained in 
$X \times B_n$, its closure $\ov V$ intersects all the Cartier divisors 
$F_{n,0}$ and $F^1_{n,i}$ ($1 \le i \le n-1$) properly in $X \times {\wh B}_n$.
In particular, their pull-backs of $F_{n,0}$ and $F^1 _{n,i}$ are all effective Cartier divisors on
${\ov V}^N$. Notice also that 
\begin{equation}\label{eqn:order}
M_{ssup} \Rightarrow  M_{sup} \Rightarrow M_{sum}.
\end{equation} 

The following restriction property of the modulus conditions $M_{sum}$ and 
$M_{ssup}$ will be used repeatedly in this paper. 

\begin{prop}[Containment lemma]\label{prop:restp}
Let $X$ be a $k$-scheme and let $Y \subset X
\times B_n$ be a closed subvariety such that its closure ${\ov Y} \subset
X \times {\widehat B}_n$ intersects the Cartier divisors $X \times F_{n,0}$ 
and $X \times F^1_n$ properly. 
Let $V$ be an irreducible closed subvariety of $X \times B_n$ such that $V$ 
satisfies the modulus condition $M_{sum}$ or $M_{ssup}$ on $X \times B_n$. Let $\ov V$ be its closure in $X \times \widehat{B}_n$.
Let $V_Y$ be an irreducible component of
$V \cap Y$ and let ${\widehat V_Y}^N$ denote the normalization of the closure 
of $V_Y$ in $\ov Y$. Let ${\nu}_Y: {\widehat V_Y}^N \to 
X \times {\widehat B}_n$ denote the 
natural map. \\
$(1)$ If $V$ satisfies the modulus condition $M_{ssup}$, then there is an 
$1 \le i \le n-1$ such that
$(m+1)[{\nu}^*_Y(F_{n,0})] \le [{\nu}^*_Y(F^1_{n,i})]$, that is, $V_Y$ also satisfies $M_{ssup}$. \\
$(2)$ If $V$ satisfies the modulus condition $M_{sum}$, then
$(m+1)[{\nu}^*_Y(F_{n,0})] \le [{\nu}^*_Y(F^1_{n})]$, that is, $V_Y$ also satisfies $M_{sum}$.
\end{prop}
\begin{proof} If $V \cap Y = \emptyset$, then there is nothing to prove. Hence, we assume
that $V \cap Y$ is nonempty, so there is at least one nonempty irreducible component $V_Y$.
We consider the following commutative diagram:
\begin{equation}\label{eqn:restp0}
\xymatrix{
{\wt V}^N_Y \ar@{^{(}->}[r] \ar[d]_{f_Y} & {Z}^N _1 \ar[d]_{f^N} 
\ar[r]^{\ov g} & 
{Z_1} \ar[d]_{f} \ar[r]^{\ov p} & 
{\ov V}^N \ar[d] \ar@/^1pc/[dd]^{\nu} \\
{\widehat V}^N_Y \ar@{^{(}->}[r] \ar[drr]_{{\nu}_Y} & Z^N \ar[r]^{g} \ar[dr]^{h} & 
Z \ar[r]^{p} \ar[d] & {\ov V} \ar@{^{(}->}[d] \\
& & {\ov Y}  \ar@{^{(}->}[r]_{j} & X \times {\widehat B}_n.}
\end{equation}
Here $Z$ and $Z_1$ are defined so that  both the upper and the lower
squares on the right are Cartesian. It is then easy to check that
$Z \cap Y = V \cap Y$ and hence $\ov {V \cap Y}$ is a union of irreducible 
components of $Z$. In particular, ${\widehat V_Y}^N$ is one of the disjoint
components of $Z^N$. Since $f^N$ is
finite and surjective, there is a component ${\wt V}^N_Y$ of ${Z}^N _1$ 
lying over ${\widehat V}^N_Y$, and the restriction $f_Y$ of $f^N$ also is a finite and surjective map.
Since $V \cap F_{n, 0}  = \emptyset$ and $V_Y \neq \emptyset$, we see
that $F_{n, 0}$ and $F^1_{n, i}$ all intersect $Z$ properly.
Now if we use $M_{ssup}$, then the modulus condition for $V$ and Lemma~\ref{lem:rest} imply that
there is an integer $1 \le i \le n-1$ such that
${\ov g}^* \circ {\ov p}^* [ {\nu}^* (F^1_{n,i} - (m+1)F_{n,0})]
\ge 0$ on ${Z}^N _1$ and hence on ${\widetilde V_Y}^N$. 
In particular, by commutativity, we get ${f_Y}^* [{{\nu}_Y}^*
(F^1_{n,i} - (m+1)F_{n,0}) ]\ge 0$ on ${\widetilde V_Y}^N$.
Since $f_Y$ is finite and surjective map of normal varieties, 
the proposition now follows from \lemref{lem:surjm} for the modulus condition
$M_{ssup}$. The case of $M_{sum}$ follows exactly the same way using $F^1 _n$ instead of $F_{n, i} ^1$, and the fact that $F^1_n$ is also an effective
Cartier divisor.
\end{proof}
As one can see from the above proposition, although the modulus condition $M_{sup}$ lies between the other two 
modulus conditions $M_{sum}$ and $M_{ssup}$, it turns out that the additive higher Chow groups
based on the latter modulus conditions have better structural properties.

In this paper, we study the additive higher Chow groups based on
the modulus conditions $M_{sum}$ and $M_{ssup}$. We shall show in the next 
section that the additive Chow groups based on our new modulus condition 
$M_{ssup}$ satisfy all the properties known to be satisfied by the 
additive higher Chow groups of Krishna-Levine, Park, Bloch-Esnault and 
R{\"u}lling. 

The following lemma is not used in the paper, but we decided to keep it for it might be useful for some follow-up works.
\begin{lem}\label{lem:rest*}
Let $X$ be a normal variety and let $D_1$ and $D_2$ be effective Cartier
divisors on $X$ such that $D_1 \ge D_2$ as Weil divisors. Let $X'$ denote
the normalization of the blow-up $Bl_X(Z)$ of $X$ along a closed subscheme
$Z \subset X$ of codimension at least two. Let $f : X' \to X$ be the natural
map. Then $f^*(D_1) \ge f^*(D_2)$.
\end{lem}
\begin{proof} Let $D_i = {\sum} n_{ij}V_{ij}$ for $i = 1, 2$, where
$V_{ij}$ are prime divisors on $X$. Since $Z$ is of codimension at least
two, we see that for each $i$, $f^*(D_i) = {\sum} n_{ij} V'_{ij} + 
{\sum} n_{l} E_l$, where $E_l$ are the components of the exceptional 
divisor, $ V'_{ij}$ is the proper transform of $ V_{ij}$ and
$n_l \ge 0$. The lemma now immediately follows.
\end{proof}

\subsection{Additive cycle complex}\label{subsection:ACC}
We define the additive cycle complex based on the above
modulus conditions. 
\begin{defn}\label{defn:AdditiveComplex}
Let $M$ be the modulus condition $M_{sum}$ or $M_{ssup}$.
Let $X$ be a $k$-scheme, and let $r,m$ be integers with $m\ge1$. \\
\\
(0) $\un{\TZ}_r(X, 1; m)_M$  is the free abelian group 
on integral closed subschemes $Z$ of $X \times {\G}_m$ of dimension $r$.\\
\\
For $n>1$,  $\un{\TZ}_r(X, n; m)_M$ is the free abelian group 
on integral closed subschemes $Z$ of $X \times B_n$ of dimension $r+n-1$ 
such that: \\ \\
$(1)$ (Good position) For each face $F$ of $B_n$, $Z$ intersects  $X \times F$
properly:
\[
\dim (Z\cap (X\times F)) \le r+\dim (F)-1, {\rm and} 
\] 
$(2)$ (Modulus condition) $Z$ satisfies the modulus $m$ condition $M$
on $X \times B_n$.   
\end{defn}
As our scheme $X$ is equi-dimensional of dimension $d$ over $k$, we write for $q \ge 0$
\[
\un{\TZ}^q(X, n; m)_M = \un{\TZ}_{d+1-q}(X, n; m)_M.
\]
We now observe that the good position condition on $Z$ implies that the cycle 
$(\id_X\times\iota_{n,i,\epsilon})^*(Z)$
is well-defined and each component satisfies the good position condition.
Moreover, letting $Y = X \times F$ for $F = {\iota}_{n,i,\epsilon}(B_{n-1})$
in \propref{prop:restp}, we first of all see that $\ov Y$ intersects $X \times F_{n,0}$
and $X \times F^1_n$ properly in $X \times {\wh B}_n$, and each component of $(\id_X\times\iota_{n,i,\epsilon})^*(Z)$ satisfies the
modulus condition $M$ on $X \times B_{n-1}$. We thus conclude that if 
$Z \subset X \times B_n$ satisfies the above conditions $(1)$ and $(2)$, then 
every component of ${\iota_{n,i,\epsilon}}^*(Z)$ also satisfies these 
conditions on $X \times B_{n-1}$. In particular, we have the cubical abelian 
group $\un{n}\mapsto \un{\TZ}^{q}(X, n; m)_M$.
\begin{defn}
The {\em additive cycle complex} $\TZ^q(X, \bullet ; m)_M$ of $X$ in codimension $q$ and with modulus $m$ condition $M$ is the non-degenerate complex associated to the cubical abelian group 
$n\mapsto \un{\TZ}^{q}(X, n; m)_M$, 
i.e.,
\[
\TZ^q(X, n; m)_M: = \frac{\un{\TZ}^q(X, n; m)_M}
{\un{\TZ}^q(X, n; m)_{M, \dgn}}.
\]
The boundary map of this complex at level $n$ is given by
$\partial = \sum_{i=1} ^{n-1} (-1)^i (\partial^{\infty}_i
- \partial^0_i)$, which satisfies $\partial ^2 = 0$. 
The homology
\[
\TH^q(X, n; m)_M: = H_n (\TZ^q(X, \bullet ; m)_M); \ n \ge 1
\]
is the 
{\it additive higher Chow group}  of $X$ with modulus $m$ condition $M$. 
\end{defn}
From now on, we shall drop the subscript $M$ from the notations
and it will be understood that the additive cycle complex or the additive
higher Chow group in question is based on the modulus condition $M$,
where $M$ could be either $M_{sum}$ or $M_{ssup}$. The reader should
however always bear in mind that these two are different objects.

There are a few comments in order. We could also have defined our additive
cycle complex by taking $\un{\TZ}_r(X, n; m)$ to be the free abelian group
generated by integral closed subschemes of $X \times {\wt B}_n$ which
have good intersection property with respect to the faces of ${\wt B}_n$,
and which satisfy the modulus condition on $X \times {\ov B}_n$
({\sl cf.} \cite{KL, P1}). However, the following easy consequence
of the modulus condition shows that this does not change the cycle complex.
\begin{lem}\label{lem:M5*}
Let $M$ be a modulus condition in \defref{defn:M5}. Then, there is a canonical
bijective correspondence between the set of irreducible closed subvarieties 
$V \subset X \times B_n$ satisfying the modulus $m$ condition $M$ and the set 
of irreducible closed subvarieties $W \subset X \times {\wt B}_n$, whose 
Zariski closure in $X \times {\ov B}_n$ satisfies the modulus $m$ condition 
$M$. Here, the correspondence is actually given by the identity map.
\end{lem}
\begin{proof} First of all, since for any integral closed subscheme $V$
of $X \times {\wh B}_n$, the pull-back ${\nu}^*(F_{n, 0})$ on $V^N$
is contained in the open subset ${\nu}^{-1}(X \times {\ov B}_n)$, we can
replace $\wh B_n$ by $\ov B_n$ in the definition of the modulus conditions.

Now, if $\Sigma$ and $\wt {\Sigma}$ are the two sets in the statement, then
the modulus condition forces that if $V \in \Sigma$, then $V$ is same as
its closure in $X \times {\wt B}_n$. Conversely, if $V \in \wt {\Sigma}$,
then the modulus condition again forces $V$ to be contained in
$X \times B_n$.
\end{proof}   
Let ${\TZ^q(X, \bullet ; m)}_{sup}$ be the additive cycle complex as defined
in \cite{KL, P1}. This complex is based on the modulus condition $M_{sup}$ 
above. It follows from ~\eqref{eqn:order} that there are natural inclusions of 
cycle complexes
\begin{equation}\label{eqn:order1}
\TZ^q(X, \bullet ; m)_{ssup} \inj {\TZ^q(X, \bullet ; m)}_{sup} \inj
{\TZ^q(X, \bullet ; m)}_{sum}
\end{equation}
and hence there are natural maps
\begin{equation}\label{eqn:order2}
{\TH}^q(X, \bullet ; m)_{ssup} \to {{\TH}^q(X, \bullet ; m)}_{sup} \to
{{\TH}^q(X, \bullet ; m)}_{sum}.
\end{equation}

One drawback of the cycle complex based on $M_{sup}$ is that the underlying
modulus condition for a cycle is not necessarily preserved when it is 
restricted to a face of $B_n$. This forces one to put an extra 
{\sl induction} condition in the definition of 
${\TZ^q(X, \bullet ; m)}_{sup}$ that requires for cycles to be admissible, not 
only the cycles themselves to satisfy $M_{sup}$ on $X \times B_n$, but also 
all their intersections with various faces to satisfy $M_{sup}$. In particular,
as $n$ gets large, this condition gets more serious, and  it might be
a very tedious job to find admissible cycles. On the other hand, the 
definition of our cycle complexes shows that this extra {\sl induction} 
condition is superfluous for the cycle complexes based on $M_{sum}$ or 
$M_{ssup}$. Based on this discussion and all the results of this 
paper, one is led to believe that even though the modulus condition 
$M_{ssup}$ may appear mildly stronger and $M_{sum}$ weaker than the 
modulus condition $M_{sup}$, the following conjecture should be true. 
\begin{conj}\label{conj:QI}
For a smooth projective variety $X$ over $k$, the natural inclusions of cycle 
complexes $\TZ^q(X, \bullet ; m)_{ssup} \inj {\TZ^q(X, \bullet ; m)}_{sup}
\inj \TZ^q(X, \bullet ; m)_{sum}$ are quasi-isomorphisms.
\end{conj}   

\section{Basic properties of ${\TH}^q(X, \bullet ; m)$}\label{section:BHCG}  
In this section, our aim is to demonstrate that the additive higher Chow
groups defined above for $M_{sum}$ and $M_{ssup}$ have all the properties (except 
Theorem~\ref{thm:1-cycle} which we do not know for $M_{sum}$)
which are known to be true for the additive Chow groups for $M_{sup}$ of \cite{KL, P1}.
Since most of the arguments in the proofs can be given either by quoting 
these references verbatim or by straight-forward modifications of the same, we
only give the sketches of the proofs with minimal explanations whenever
deemed necessary. We begin with the following structural properties of
our additive Chow groups.
\begin{thm}\label{thm:basic}
Let $f : Y \to X$ be a morphism of $k$-schemes.
\begin{enumerate}
\item
If $f$ is projective, there is a natural map of cycle complexes
$f_*: \TZ_r(Y, \bullet ; m) \to \TZ_r(X, \bullet ; m)$ which induces
the analogous push-forward map on the homology.
\item
If $f$ is flat, there is a natural map of cycle complexes
$f^*: \TZ_r(X, \bullet ; m) \to \TZ_r(Y, \bullet ; m)$ which induces
the analogous pull-back map on the homology. These pull-back and push-forward
maps satisfy the obvious functorial properties.
\item
If $X$ is smooth and projective, there is a product  
\[
\cap_X: \CH^r(X, p) \otimes \TH_s(X,q ; m)\to \TH_{s-r}(X, p+q ; m),
\] 
that is natural with respect to flat pull-back, and that satisfies the projection formula 
\[
f_*( f^*(a)\cap_Xb)=a\cap_Y f_*(b)
\]
for $f:X\to Y$ a morphism of smooth projective varieties. If $f$ is flat in addition, we have an additional projection formula
\[
f_*(a\cap_X f^*(b))=f_*(a)\cap_Y b.
\]
\item
If $X$ is smooth and quasi-projective, there is a product
\[
\cap_X:\CH^r(X) \otimes\TH_s(X,q ; m)  \to \TH_{s-r}(X, q ; m),
\] 
that is natural with respect to flat pull-back, and that satisfies the projection formula 
\[
f_*( f^*(a)\cap_Xb)=a\cap_Yf_*(b)
\]
for $f:X\to Y$ a projective morphism of smooth quasi-projective varieties. If $f$ is flat in addition, we have an additional projection formula
\[
f_*(a\cap_Xf^*(b))= f_*(a)\cap_Yb.
\]
\\
Furthermore, all products are associative.
\end{enumerate}
\end{thm}
\begin{proof}({\sl cf.} \cite{KL}) Granting the flat pull-back and the 
projective push-forward, the theorem is a direct consequence
of \cite[Lemmas~4.7, 4.9]{KL} whose proofs are independent of the choice
of the modulus conditions of \defref{defn:M5}, as the interested reader may 
verify. The proofs of the flat pull-back and projective push-forward maps
on the level of cycle complexes also follow in the same way as in 
{\sl loc. cit.} using our \lemref{lem:surjm}.
\end{proof}
\begin{thm}[Projective Bundle and Blow-up formulae]\label{thm:PBF}
Let $X$ be a smooth quasi-projective variety and 
let $E$ be a vector bundle on $X$ of rank $r+1$. Let 
$p:  {\P}(E) \to X$ be the associated projective bundle over 
$X$. Let ${\eta} \in CH^1(\P(E))$ be the class of the tautological 
line bundle $\sO(1)$.
Then for any $q , n \ge 1$ and $m \ge 2$, the map 
\[
\theta: \bigoplus_{i= 0}^r \TH^{q-i}(X, n; m)\to \TH^q(\P(E), n; m)
\]
 given by
\[
(a_0, \cdots , a_r) \mapsto \stackrel{r}{\underset {i =0} \sum} 
\eta^i \cap_{\P(E)} p^*(a_i)
\]
 is an isomorphism. 

If $i:Z\to X$ is a closed  immersion of smooth projective varieties and 
$\mu:X_Z\to X$ is the blow-up of $X$ along $Z$ with $i_E:E\to X_Z$ the 
exceptional divisor with morphism $q:E\to Z$. Then the sequence
\begin{multline*}
0\to \TH^s(X,n;m)\xrightarrow{(i^*,\mu^*)}  
\TH^s(Z,n;m)\oplus\TH^s(X_Z,n;m)\\\xrightarrow{q^*-i_E^*}
\TH^s(E,n;m)\to 0
\end{multline*}
is split exact.
\end{thm}
\begin{proof} ({\sl cf.} \cite{KL}) The proof of both the formulae is a
consequence of the corresponding decomposition of motives in the 
triangulated category of Chow motives ${\rm Mot}_k$ together with the fact that the 
additive Chow groups can be defined as a functor of graded abelian groups
on ${\rm Mot}_k$. But this functoriality is a direct consequence of
Theorem~\ref{thm:basic}.
\end{proof}  
\begin{thm}\label{thm:Ext} Assume that $k$ admits resolution of 
singularities. Then the functor $\TZ_r(-;m):\SmProj/k\to D^-(\Ab)$ 
extends to a functor
\[
\TZ_r^\log(-;m):\Sch/k \to D^-(\Ab)
\]
together with a natural transformation of functors $\TZ_r^\log(-;m) \to
\TZ_r(-;m)$ satisfying:\\
\\
(1) Let $\mu:Y\to X$ be a proper morphism in $\Sch/k$, $i:Z\to X$ a closed 
immersion. Suppose that $\mu:\mu^{-1}(X\setminus Z)\to X\setminus Z$ is an 
isomorphism. Set $E:=\mu^{-1}(X\setminus Z)$ with maps $i_E:E\to Y$, 
$q:E\to Z$. There is a canonical extension of
the sequence in $D^-(\Ab)$
\[
\TZ_r^\log(E;m)\xrightarrow{(i_{E*},-q_*)}\TZ^\log_r(Y;m)\oplus 
\TZ^\log_r(Z;m)\xrightarrow{\mu_*+i_*} \TZ^\log_r(X;m)
\]
to a distinguished triangle in $D^-(\Ab)$.\\
\\
(2) Let $i:Z\to X$ be a closed immersion in $\Sch/k$, $j:U\to X$ the open 
complement. Then there is a canonical distinguished triangle in $D^-(\Ab)$:
\[
\TZ^\log_r(Z;m)\xrightarrow{i_*}\TZ^\log_r(X;m)
\xrightarrow{j^*}\TZ^\log_r(U;m)\to \TZ^\log_r(Z;m)[1],
\]
which is natural with respect to proper morphisms of pairs $(X,U)\to (X',U')$.
\\
(3) For any $X \in \Sch/k$, the natural map $\TH_r^\log(X, n ;m) \to
\TH_{r+p}^\log(X \times {\A}^p, n ;m)$ is an isomorphism.
\end{thm}
\begin{proof}({\sl cf.} \cite{KL}) This follows directly from 
\thmref{thm:basic} and \thmref{thm:PBF} above together with the main results of
Guill{\'e}n and Navarro Aznar \cite{GN}. We refer \cite[Section~6]{KL}
for details. The natural transformation of functors is an immediate
consequence of the constructions of Guill{\'e}n and Navarro Aznar using the
proper hyper cubical resolutions, and the proper push-forward property
of additive cycle complex.
\end{proof}  
Next we study the question of the existence of the regulator maps from our 
additive higher Chow groups to the modules of absolute K{\"a}hler 
differentials. First we prove the following result of Bloch-Esnault \cite{BE2}
and R{\"u}lling \cite{R} on $0$-cycles for the modulus condition $M_{ssup}$. 
\begin{thm}\label{thm:0-cycle}
Assume that char$(k) \neq 2$ and let ${\W}_m{\Omega}^{\bullet}_k$ denote
the generalized de Rham-Witt complex of Hesselholt-Madsen ({\sl cf.}
\cite{R}). Then there is a natural isomorphism
\[
R^n_{0, m} : \TH^n(k, n ;m) \to {\W}_m{\Omega}^{n-1}_k.
\]
\end{thm}
\begin{proof} This is already known for $M_{sum}$. For the modulus condition $M_{ssup}$, we first note that the map $R^n_{0, m}$ is the composite 
map $$\TH^n(k, n ;m)_{ssup} \to {\TH^n(k, n ;m)}_{sum} \xrightarrow{\theta}
{\W}_m{\Omega}^{n-1}_k,$$ where ${\theta}$ is constructed in \cite{R}
and this coincides with the regulator map of Bloch-Esnault for $m =1$.
Furthermore for $m = 1$, Bloch-Esnault define the inverse map ${\Omega}^{n-1}_k \to
{\TH^n(k, n; 1)}_{sum}$ using a presentation of ${\Omega}^{n-1}_k$. The 
reader can easily check from the proof of \cite[Proposition~6.3]{BE2} that
the inverse map is actually defined from ${\Omega}^{n-1}_k$
to $\TH^n(k, n; 1)_{ssup}$. This completes the proof when $m = 1$.

For $m \ge 2$, the proof of  K. R\"ulling for ${\TH^n(k, n; 1)}_{sum}$ has three 
main steps, namely: \\
$(1)$ The existence of map $R^n_{0, m}$, \\
$(2)$ The isomorphism of $R^1_{0, m}$, and \\
$(3)$ The existence of transfer maps on the additive higher Chow groups 
for finite extensions of fields. \\ 
$(4)$ Showing that pro-group ${\{\TH^n(k, n; m)\}}_{n, m \ge 1}$  
is an example of a restricted Witt complex (see {\sl loc. cit.}, Remark~4.22).

We have already shown $(1)$ for our $\TH^n (k,n;m)_{ssup}$. The proof of 
$(3)$ is a simple consequence of Theorem~\ref{thm:basic}. 
The surjectivity part of $(2)$ follows from the result of R\"ulling 
and the isomorphism $\un\TZ^n(k, n; m)_{ssup} = {\un\TZ^n(k, n; m)}_{sum}$.
To prove injectivity, we follow the proof of Corollary~4.6.1 of {\sl loc.
cit.} and observe that if there is a cycle $\zeta \in  \un\TZ^1(k, 1; m)$
such that $R^1_{0, m}(\zeta) = 0$, then $\zeta$ is the boundary of a curve
$C$ which is an admissible cycle with the modulus condition $M_{sum}$.
But then $C$ is admissible cycle also with the modulus condition $M_{ssup}$
since one has $M_{ssup} = M_{sup} = M_{sum}$ when $n = 2$ by definition.
This proves $(2)$. Note that this does not need any assumption on the
characteristic of the ground field.

We now sketch the proof of $(4)$ to complete the proof of the theorem.
We have seen in Remark~\ref{remk:delaltS} that 
$\left({\{\TH^n(k, n; m)\}}_{n, m \ge 1}, \wedge, \delta_{alt}\right)$ forms 
a graded-commutative differential graded algebra. It is also easy to see
that the push-forward and pull-back maps for the finite and flat
map $a \mapsto a^r$ on $\G_m$ induces the Frobenius and 
Verschiebung operators $F_r$ and $V_r$ on these additive higher Chow groups, 
and they satisfy $\delta_{alt} F_r = r F_r \delta_{alt}$ and
$r \delta_{alt} V_r = V_r \delta_{alt}$.
Moreover, the same proof as in \cite[Lemma~4.17]{R} shows that if
char$(k) \neq 2$, then $F_r \delta_{alt} V_r = \delta_{alt}$. This proves
$(4)$. 

As shown in {\sl loc. cit.}, the above four ingredients and the universality
of the de Rham-Witt complex imply that there is a map
${\W}_m{\Omega}^{n-1}_k \xrightarrow{S^n_{0,m}}\TH^n(k, n ;m)$ which is
surjective. On the other hand, one checks from the 
construction of the map $R^n_{0.m}$ in {\sl loc. cit.} that
$R^n_{0.m} \circ S^n_{0.m}$ is identity. 
\end{proof}
\begin{remk}\label{remk:ch2}
One would like to have the assumption char$(k) \neq 2$ removed from the
statement of Theorem~\ref{thm:0-cycle}. In this context, we remark
that the only place we used this assumption was to show 
the identity $F_r \delta_{alt} V_r = \delta_{alt}$. This is an imrovement
over the result of R{\"u}lling who needs this assumption even to get a
CDGA structure on the additive Chow groups. It is possible that
the identity $F_r \delta_{alt} V_r = \delta_{alt}$ holds in the additive
higher Chow groups for our choice of derivation even if char$(k) = 2$.
But we have not checked this.
\end{remk}

The following is an immediate consequence of the results of R\"ulling
and Theorem~\ref{thm:0-cycle}. This gives an evidence of \conjref{conj:QI}.
\begin{cor}\label{cor:0-cycle*}
For every $n, m \ge 1$, the natural maps
\[
\TH^{n}(k, n; m)_{ssup} \to {\TH^{n}(k, n; m)}_{sup} \to 
{\TH^{n}(k, n; m)}_{sum}
\]
are isomorphisms.
\end{cor}
$\hfill \square$

\enlargethispage*{50pt}
We finally turn to the regulator maps for 1-cycles as considered in \cite{P1}.
\begin{thm}\label{thm:1-cycle}
Suppose that $k$ is of characteristic zero and assume the modulus condition
to be $M_{ssup}$. Then there is a natural non-trivial regulator map
\begin{equation}\label{eqn:MCC}
R^n_{1, m} : \TH^{n-1}(k, n; m)_{ssup} \to {\Omega}^{n-3}_k.
\end{equation}
This map is surjective if $k$ is moreover algebraically closed.
\end{thm} 
\begin{proof} The map $R^n_{1, m}$ is the composite map
$$\TH^n(k, n ;m)_{ssup} \to {\TH^n(k, n ;m)}_{sup} \xrightarrow{\theta}
{\Omega}^{n-3}_k,$$ where ${\theta}$ is constructed in \cite{P1}.
For the non-triviality of $R^n_{1, m}$, J. Park constructs a 1-cycle
$\Gamma$ (see \cite[Proposition~1.9]{P2}, \cite[7.11]{KL}) and shows 
(see \cite[Lemmas~1.7, 1.9]{P2}) that each component of $\Gamma$
in fact satisfies the modulus condition $M_{ssup}$. Hence $R^n_{1, m}$
is non-trivial. If $k = {\ov k}$, then the proof of the surjectivity
in \cite[Section~7]{KL} follows from the following: \\
$(1)$ An action of $k^{\times}$ on $\TH^n(k, n ;m)$, \\ 
$(2)$ Suitable $k^{\times}$-equivariance of $R^3_{1, m}$ up to a scalar, \\
$(3)$ The surjectivity of $R^3_{1, m}$, \\
$(4)$ The cap product $CH^n(k,n) {\otimes}_{\Z} \TH^2(k, 3 ;m) \to
\TH^{n+2}(k, n+3 ;m)$. \\
The action of $k^{\times}$ on our additive higher Chow groups is given
exactly as in \cite{KL} by
\begin{equation}\label{eqn:MCC**}
a * (x, t_1, \cdots, t_{n-1}) = (x/a,  t_1, \cdots, t_{n-1}).
\end{equation}
This action extends to an action of $k^{\times}$ on ${\wh B}_n$.
The proof of $(2)$ now follows from the $k^{\times}$-equivariance of the
natural map $\TZ_r(k, n;m)_{ssup} \to {\TZ_r(k, n;m)}_{sup}$ and the results of
\cite{KL}. The proof of $(3)$ is a direct consequence of $(1)$, $(2)$
and the fact that $k$ is algebraically closed field of characteristic zero.
Finally, $(4)$ is already shown in \thmref{thm:basic}.
\end{proof}
\section{Preliminaries for Moving lemma}\label{section:ML}
The underlying additive cycle complexes and additive higher
Chow groups in all the results in the rest of this paper will be based on the 
modulus condition $M_{sum}$ or $M_{ssup}$, unless one of these is 
specifically mentioned.   
Our next three sections will be devoted to proving our first main result
of this paper:
\begin{thm}\label{thm:ML}
Let $X$ be a smooth projective variety over a perfect field $k$. Let 
$\sW$ be a finite collection of locally closed subsets of $X$. Then, 
the inclusion of additive higher Chow cycle complexes 
(see below for definitions)
\[
\TZ^q_{\sW} (X, \bullet; m) \inj \TZ^q (X, \bullet; m)
\]
is a quasi-isomorphism. In other words, every admissible additive higher Chow 
cycle is congruent to another admissible cycle intersecting properly all 
given finitely many locally closed subsets of $X$ times faces.
\end{thm}   
In this section, we set up our notations and machinery that are needed to
prove this theorem, and prove some preliminary steps. 
Let $X$ be a smooth projective variety over $k$ and we fix an
integer $m \ge 1$. Let $\sW$ be a finite collection of locally closed 
algebraic subsets of $X$. If a member of $\sW$ is not irreducible, we always 
replace it by all of its irreducible components so that we assume all members 
of $\sW$ are irreducible. For a locally closed subset $Y \subset X$, recall 
that the codimension ${\rm codim}_X Y$ is defined to be the minimum of 
${\rm codim}_X Z$ for all irreducible components $Z$ of $Y$.  
\begin{defn}\label{defn:moving0}
We define $\underline{\TZ}^{q}_{\sW}(X, n;m)$ to be the subgroup of 
$\underline{\TZ}^{q}(X, n;m)$ generated by integral closed subschemes
$Z \subset X \times B_n$ such that \\
$(1) \ Z$ is in $\underline{\TZ}^{q}(X, n;m)$ and \\
$(2) \ \codim _{W \times F} (Z \cap (W \times F)) \geq q$ for all 
$W \in \mathcal{W}$ and all faces $F$ of $B_n$.
\end{defn}
It is easy to see that $\underline{\TZ}^{q}_{\sW}(X, \bullet ;m)$ forms a
cubical subgroup of $\underline{\TZ}^{q}(X, \bullet ;m)$, giving us the
subcomplex 
\[
{\TZ}^{q}_{\sW}(X, \bullet ;m) =
{\frac{\underline{\TZ}^{q}_{\sW}(X, \bullet ;m)}
{{\underline{\TZ}^{q}_{\sW}(X, \bullet ;m)}_{\rm degn}}} \subset
{\TZ}^{q}(X, \bullet ;m).
\]
Let ${\TH}^q_{\sW}(X, \bullet ;m)$ denote the homology of the complex 
${\TZ}^{q}_{\sW}(X, \bullet ;m)$. Then the above inclusion induces a
natural map of homology 
\begin{equation}\label{eqn:incl0}
{\TH}^q_{\sW}(X, \bullet ;m) \to {\TH}^q(X, \bullet ;m).
\end{equation} 
More generally, if $e: \mathcal{W} \to {\Z}_{\geq 0}$ is a set-theoretic 
function,
then one can define subcomplexes $\underline{\TZ}^{q}_{\sW, e}(X, \bullet;m)$
replacing the condition $(2)$ above by \\
$(2e) \ \codim _{W \times F} (Z \cap (W \times F)) \geq q - e(W)$.
\\
In this generality, the subcomplex $\underline{\TZ}^{q}_{\sW}(X, \bullet;m)$
is same as $\underline{\TZ}^{q}_{\sW, 0}(X, \bullet;m)$.
\begin{remk}\label{remk:general}
Let $\Phi$ be the set of all set-theoretic functions 
$e: \mathcal{W} \to \mathbb{Z}_{\geq 0}$. Give a partial ordering on 
$\Phi$ by declaring $e' \geq e$ if $e' (W) \geq e (W)$ for all 
$W \in \mathcal{W}$. If two functions $e, e' \in \Phi$ satisfy $e' \geq e$, 
then for any irreducible admissible variety 
$Z \in Tz_{\mathcal{W}, e} ^q (X, n;m)$, we have
\begin{equation}
\codim_{W \times F} (Z \cap (W \times F)) \geq q - e (W) \geq q - e'(W)
\end{equation} 
for all $W \in \mathcal{W}$ and all faces $F \subset B_n$. Thus, we have 
\begin{equation}
{\TZ}_{\mathcal{W}, e} ^q (X, n; m) \subset 
{\TZ}_{\mathcal{W}, e'} ^q (X, n; m) \ \ \ \mbox{ for } e \leq e'. 
\end{equation}
Note that if $e \in \Phi$ satisfies $e \geq q$ where $q$ is considered as a 
constant function in $\Phi$, then automatically 
\begin{equation} 
{\TZ}^q _{\mathcal{W}, q}(X, n;m) = {\TZ}^q _{\mathcal{W}, e} (X, n; m) = 
{\TZ}^q (X, n;m).
\end{equation} 
Since $0 \leq e$ for all $e \in \Phi$, for each triple $e, e', e''$ such that 
$e \leq e' \leq q \leq e''$, we have
\begin{eqnarray*} 
{\TZ}_{\mathcal{W}} ^q (X, n; m) \subset {\TZ}_{\mathcal{W}, e} ^q (X, n;m) 
\subset {\TZ}_{\mathcal{W}, e'} ^q (X, n;m)  \\
\subset {\TZ}_{\mathcal{W}, q}^q  (X, n;m) = 
{\TZ}_{\mathcal{W}, e''} ^q (X, n;m) = {\TZ}^q (X, n;m).
\end{eqnarray*}
All these (in)equalities are equivariant with respect to the boundary maps.
\end{remk}
\begin{remk}\label{remk:chain}
The main theorem is equivalent to that the inclusion induces an isomorphism 
$\TH^q _{\mathcal{W}} (X, n;m) \simeq \TH ^q (X, n;m)$ for the given 
modulus conditions $M$. This is equivalent to that for each pair 
$e, e' \in \Phi$ with $e \leq e'$ the inclusion induces an isomorphism
\begin{equation}
\TH^q _{\mathcal{W}, e} (X, n;m) \simeq \TH^q _{\mathcal{W}, e'} (X, n;m).
\end{equation}
\end{remk}
Our remaining objective in this section is to prove the following additive 
analogue
of the spreading argument of Bloch-Levine. We begin with the following
results.
\begin{lem}\label{lem:extn}
Let $f : X \to Y$ be a projective and dominant map of integral 
normal varieties and let $\eta$ denote the generic point of $Y$. Consider the
fiber diagram
\begin{equation}\label{eqn:ext0}
\xymatrix{
X_{\eta} \ar[r]^{j_{\eta}} \ar[d] & X \ar[d]^f \\
{\{{\eta}\}} \ar[r] & Y.}
\end{equation}
Let $D$ be a Weil divisor on $X$ such that $j^*_{\eta} (D)$ is effective.
Then there is a non-empty open subset $U \subset Y$ such that if 
$j : f^{-1}(U) \to X$ is the open inclusion, then $j^*(D)$ is also effective.
\end{lem}
\begin{proof} Let $D = \sum n_iD_i$. Then $j^*_{\eta} (D)$ is effective
if and only if for every $i$ with $n_i < 0$, one has $D_i \cap X_{\eta}
= \emptyset$. Since $D$ is a finite sum, it suffices to show that if $D$ is
a prime divisor on $X$ such that $D \cap X_{\eta} = \emptyset$, then there is
a nonempty open subset $U \subset Y$ such that $D \cap f^{-1}(U) = \emptyset$.

Since $f$ is projective map, it is in particular closed. Hence $f(D)$ is
closed in $Y$. Moreover, our hypothesis implies that $f(D)$ is a proper
closed subset of $Y$. Thus $U = Y \backslash f(D)$ is the desired open subset 
of $Y$.
\end{proof} 
\begin{lem}\label{lem:pro} 
Let $X$ be a quasi-projective $k$-variety and let 
$\sW$ be a finite collection of locally closed subsets of $X$. Let $K$ be a 
finite field extension of $k$. Let $X_K$ be the base extension 
$X _K = X \times_{{\rm Spec} (k)} {{\rm Spec} (K)}$, and let $\mathcal{W}_K$ 
be the set of the base extensions of sets in $\mathcal{W}$. Then there are 
natural maps 
\[
p^*: \frac{{\TZ}^{q}(X, \bullet ;m)}
{{\TZ}^{q}_{\sW}(X, \bullet ;m)} \to
\frac{{\TZ}^{q}(X_K, \bullet ;m)}
{{\TZ}^{q}_{{\sW}_K}(X_K, \bullet ;m)}
\]
\[
p_*: \frac{{\TZ}^{q}(X_K, \bullet ;m)}
{{\TZ}^{q}_{{\sW}_K }(X_K, \bullet ;m)} \to
\frac{{\TZ}^{q}(X, \bullet ;m)}
{{\TZ}^{q}_{\sW}(X, \bullet ;m)}
\]
such that $p_* \circ p^* = [K:k] \cdot id$. 
\end{lem}
\begin{proof} By \thmref{thm:basic}, one as well has the flat pull-back and 
finite push-forward maps
${\TZ}^{q}_{\sW'}(X, \bullet ;m) \to 
{\TZ}^{q}_{{\sW'}_K}(X_K, \bullet ;m)$ and
${\TZ}^{q}_{{\sW'}_K}(X_K, \bullet ;m) \to 
{\TZ}^{q}_{\sW'}(X, \bullet ;m)$ for any $\sW'$. Taking for $\sW'$, the 
collection $ \{X\}$
and also $\sW$, and then taking the quotients of the two, we get the desired 
maps. The last property of the composite map is obvious from the construction 
of the pull-back and the push-forward maps on the additive cycle complexes 
({\sl cf.} \cite{KL}).
\end{proof} 
\begin{prop}[Spreading lemma]\label{prop:SL}
Let $k \subset K$ be a purely transcendental extension. For a smooth 
projective variety $X$ over $k$ and any finite collection $\sW$ of 
locally closed algebraic subsets of $X$, let $X_K$ and $\mathcal{W}_K$ be the 
base extensions as before. Let 
$p_K: X_K \to X_k$ be the natural map. Then, the pull-back map
\[
p_{K} ^* : \frac{{\TZ} ^q (X, \bullet ;m)}{{\TZ}^q _{\sW} (X, \bullet;m)} \to 
\frac{{\TZ}^q (X_K, \bullet;m)}{ {\TZ}^q _{{\sW}_K} (X_K, \bullet;m)}
\]
is injective on homology.
\end{prop}
\begin{proof} If $k$ is a finite field, then for each prime $l$ different
from the characteristic of $k$, there are infinite pro-$l$ algebraic
extensions of $k$. Combining this with Lemma~\ref{lem:pro}, we can assume 
that $k$ is infinite. Furthermore, since the additive Chow groups of
$X_K$ is a projective limit of the additive Chow groups of $X_L$, where
$L (\subset K)$ is a purely transcendental extension of $k$ of finite 
transcendence degree over $k$, we can assume that the transcendence degree of 
$K$ over $k$ is finite.

Now let $Z \in {\TZ} ^q (X, n; m)$ be a cycle such that 
$\partial Z \in {\TZ}^q _{\sW} (X, n; m)$ where there are admissible cycles 
$B_K \in {\TZ}^q (X_K, n+1; m)$ and $V_K \in {\TZ}^q _{{\sW}_K} (X_K, n; m)$
satisfying $Z_K = \partial (B_K) + V_K$. 

We first consider the natural inclusion of complexes 
${\TZ}^q (X, \bullet; m) \inj z^q (X \times {\A }^1 _k, \bullet -1)$. Since 
$K$ is the function field of some affine space ${\A}^r _k$, we can use the 
specialization argument for Bloch's cycle complexes
({\sl cf.} \cite[Lemma~2.3]{Bl1}) 
to find an open subset $Y \subset {\A}^r _k$ and cycles
\[
B_Y \in z^q (X \times Y \times {\A}^1 _k, n), \ V_Y \in 
z^q _{\sW \times Y \times \mathbb{A}^1 _k} (X \times Y \times {\A}^1 _k, n-1)
\]
such that $B_K$ and $V_K$ are the restrictions of $B_Y$ and $V_Y$ 
respectively to the generic point of $Y$ and $Z \times Y = 
\partial (B_Y) + V_Y$. In particular,
all components of $B_Y$ and $V_Y$ intersect all faces of $X \times Y \times
B_{n+1}$ and $X \times Y \times B_n$ properly. To make $B_Y$ and $V_Y$ 
admissible additive cycles, we modify them using our Lemma~\ref{lem:extn}.
 
To check the modulus condition for our cycles, let $\eta$ denote the generic 
point ${\rm Spec}(K)$ of $Y$.  
Let ${\widehat B_Y}^N$ and ${\widehat V_Y}^N$ denote the normalizations of 
the closures of $B_Y$ and $V_Y$ in $X \times Y \times {\widehat B}_{n+1}$
and $X \times Y \times {\widehat B}_{n}$ respectively. 

We first prove the admissibility under the modulus condition $M_{ssup}$ which 
is a priori more difficult than $M_{sum}$.  
The admissibility of $B_K$ and $V_K$ implies that
there are integers $1 \le i \le n$ and $1 \le i' \le n-1$ such that 
in the Diagram~\eqref{eqn:ext0}, the Weil divisors $j^*_{\eta}(F^1_{n+1, i} -
{(m+1)}F_{n+1, 0})$ and $j^*_{\eta}(F^1_{n, i'} -
{(m+1)}F_{n, 0})$ are effective on ${\widehat B}_{Y, \eta}^N$ and 
${\widehat V}_{Y, \eta}^N$ respectively. Since $X$ and ${\widehat B}_n$ are
projective, the maps ${\widehat B_Y}^N, {\widehat V_Y}^N \to Y$ are
projective. These maps are dominant since $B_K$ and $V_K$ are non-zero
cycles. Thus we can apply Lemma~\ref{lem:extn} to
find an open subset $U \subset Y$ such that $j^*_{U}(F^1_{n+1, i} -
{(m+1)}F_{n+1, 0})$ and $j^*_{U}(F^1_{n, i'} -
{(m+1)}F_{n, 0})$ are also effective. The same argument applies for
the modulus condition $M_{sum}$ as well. We just have to replace the
Cartier divisors $F^1_{n+1,i}$ and $F^1_{n, i'}$ by $F^1_{n+1}$ and
$F^1_n$ respectively. Lemma~\ref{lem:extn} applies in this case, too. 

Replacing $Y$ by $U$, we see that
\begin{equation}\label{eqn:SL0}
B_U \in {\TZ}^q (X \times U, n+1; m), \ V_U \in 
{\TZ}^q _{\sW \times U} (X \times U , n; m), \ \
Z \times U = \partial (B_U) + V_U.
\end{equation}
Next, ~\eqref{eqn:SL0} implies that for a $k$-rational point $u \in U (k)$
such that the restrictions of $B_U$ and $V_U$ to $X \times \{u\}$
give well-defined cycles in $z^q (X \times {\A}^1, n)$ and 
$z^q _{\sW \times \mathbb{A}^1 _k} (X \times {\A}^1 _k, n-1)$, one has $Z = 
\partial \left(i^*_u(B_U)\right) + i^*_u( V_U)$, where 
$i_u: X \times \{ u \} \to X \times U$ is the closed immersion. We can assume 
that
$i^*_u(B_U)$ and $i^*_u(V_U)$ are not zero. 
We now only need to show that $i^*_u(B_U)$ and $i^*_u(V_U)$ satisfy
the modulus condition on $X \times \{u\}$. But this follows directly
from ~\eqref{eqn:SL0} and the containment lemma, 
Proposition~\ref{prop:restp}. This completes the proof of the proposition.
\end{proof}
\section{Moving lemma for projective spaces}\label{section:MLP}
We follow the strategy of S. Bloch to prove the moving lemma for the additive
higher Chow groups. This involves proving the moving lemma first for the
projective spaces and then deducing the same for general smooth projective 
varieties
using the techniques of linear projections. This section is devoted to the
proof of the moving lemma for the projective spaces.
We use the following technique a few times to prove the proper-intersection 
properties of moved cycles with the prescribed algebraic sets.
\begin{lem}[{{\sl cf.} \cite[Lemma~1.1]{Bl1}}]\label{lem:action} 
Let $X$ be an algebraic $k$-scheme, and $G$ a connected algebraic $k$-group 
acting on $X$. Let $A, B \subset X$ be closed subsets, and assume that the 
fibers of the map 
\[
G \times A \to X \ \ \ (g,a) \mapsto g\cdot a
\]
all have the same dimension, and that this map is dominant. Then, there 
exists a non-empty open subset $U \subset G$ such that for all $g \in U$, the 
intersection $g(A) \cap B$ is proper in $X$.
\end{lem}
\begin{proof}
Consider the fiber square 
\[
\xymatrix{
C \ar[r] \ar[d] & G \times A \ar[d] \\
B \ar[r] & X,}
\]
and take 
\[
U = \{ g \in G | \mbox{the fiber of } C \to G \times A \to 
G \mbox{ over } g \mbox{ has the smallest dimension.} \}.
\]
For such $g \in U$, we have the desired property.  
\end{proof}
\begin{prop}[Admissibility of projective image]\label{prop:P0}
Let $f : X \to Y$ be a projective morphism of quasi-projective varieties
over a field $k$. Let $Z \in \un{\TZ}^r(X, n; m)$ be an irreducible
admissible cycle and let $V = f(Z)$. Then $V \in \un{\TZ}^s(Y, n; m)$,
where $s$ is the codimension of $V$ in $Y \times B_n$.
\end{prop}
\begin{proof}
We prove it in several steps. \\
{\bf Claim (1):} \emph{$V$ intersects all codimension one faces $F$ of $B_n$ 
properly in $B_n$. }\\

Consider $F = F_{n, i} ^{\epsilon} = {\iota}_{n,i, \epsilon}(B_{n-1})$ for 
some $i \in \{ 1, 2, \cdots, n-1 \}$, $\epsilon \in \{ 0 , \infty \}$, and 
consider the diagram 
$$\begin{CD}
X \times B_{n-1} @>{{\iota}_{n,i, \epsilon}}>> X \times B_n \\
@VV{f_{n-1}}V @VV{f_n}V \\
Y \times B_{n-1} @>{{\iota}_{n,i, \epsilon}}>> Y \times B_n.
\end{CD}$$ 
Since $F$ is a divisor in $B_n$, that $V$ intersects $Y \times F$ properly is 
equivalent to that $Y \times F\not \supset V$. Towards contradiction, suppose 
that $V \subset Y \times F$. Then, $$ Z \subset f_n ^{-1} (f_n (Z)) = 
f_n ^{-1} (V) \subset f_{n} ^{-1} (Y \times F) = {\iota}_{n,i, \epsilon} 
(f_{n-1} ^{-1} (Y \times B_{n-1}) = X \times F.$$ By assumption, $Z$ 
intersects $X \times F$ properly so that we must have 
$Z \not \subset X \times F$. 
This is a contradiction. This proves Claim (1). \\ 
{\bf Claim (2):} \emph{$V$ intersects all lower dimensional faces of $B_n$ 
properly. }\\

By the admissibility assumption, all cycles 
$\partial_i ^{\epsilon} (Z)= Z \cap (X \times F_{n, i}^{\epsilon})$ are 
admissible. Moreover, it is easy to see that $\partial_i ^{\epsilon} (V) = 
f_{n-1}(\partial_i ^{\epsilon} (Z))$. Thus we can replace $Z$ by 
$\partial_i ^{\epsilon} (Z)$ and
apply the same argument as above; inductively we see that $V$ has good 
intersection property. \\
{\bf Claim (3):} \emph{For each face $F$ of $B_n$, including the case 
$F= B_n$, the 
cycle $V \cap (Y \times F)$ has the modulus condition. }\\

For any face $F = \iota (B_i)\subset B_n$, where $\iota: B_i 
\hookrightarrow B_n$ is a face map, and for the projections 
$f_i : X \times B_i \to Y \times B_i$, note that $V \cap (Y \times F) = 
f_n (Z \cap (X \times F)) = f_i (Z |_{X \times F})$. But the admissibility
of $Z$ implies that $Z |_{X \times F}$ is also admissible ({\sl cf.}
\propref{prop:restp}). 
Hence, replacing $Z |_{X \times F}$ by $Z$, we only need to prove it for 
$F = B_n$, that is,  we just need to show that $V$ satisfies the modulus 
condition.
Consider the diagram
$$\begin{CD}
X\times B_n @>>> X \times \wh{B}_{n}\\
@VV{f_n = f}V @VV{\bar{f}_n = \bar{f}}V\\
Y \times B_n @>>> Y \times \wh{B}_n\end{CD}$$
\noindent
{\bf Subclaim:} Let $\ov{V}$ be the closure of $V$ in $Y \times \wh{B}_n$ 
and let $\ov{Z}$ be the closure of $Z$ in $X \times \wh{B}_n$. Then
$\ov{V} = \bar{f} (\ov{Z})$. \\
Since $Z \subset f^{-1} (V) \subset \bar{f}^{-1} (\ov{V})$ and $V$ is closed,
we have $\ov{Z} \subset \ov{f} ^{-1} (\ov{V})$. Hence, $\bar{f} (\ov{Z}) 
\subset \ov{V}$. For the other inclusion, note that $W = f(Z) \subset \bar{f}
(\ov{Z})$ and $\bar{f} (\ov{Z})$ is closed because $\bar{f}$ is projective. 
Hence $\ov{W} \subset \bar{f} (\ov{Z})$. This proves this subclaim.

To prove the modulus condition for $V$, we take the normalizations 
$\nu_{\ov{Z}} : {\ov Z}^N \to \ov{Z}$ and $\nu_{\ov{V}} : 
{\ov V}^N \to \ov{V}$ of $\ov{Z}$ and $\ov V$, and consider the following 
diagram
\[
\xymatrix{ 
{\ov Z}^N \ar[d]^{f^N_Z} \ar[r] ^{\nu_{\ov{Z}}} & 
\ov{Z} \ar[r]^{\iota_1} \ar[d] ^{f_Z =\bar{f} |_{\ov{Z}}} & 
X \times \wh{B}_n \ar[d] ^{\bar{f}} \\
{\ov V}^N \ar[r] ^{\nu_{\ov{V}}} & \ov{V} \ar[r]^{\iota_2} & 
Y \times \wh{B}_n,}
\]
where $\iota_1, \iota_2$ are the inclusions, and $f^N_Z$ is given by the 
universal property of the normalization $\nu_{\ov{V}}$ for dominant 
morphisms. Note that $f^N_Z$ is automatically projective and surjective 
because $f_Z$ is so. Let $q_{\ov{Z}}:= \iota_1 \circ \nu_{\ov{Z}}$ and 
$q_{\ov{V}} = \iota_2 \circ \nu_{\ov{V}}$. 

Suppose $Z$ satisfies the
modulus condition $M_{ssup}$ and consider on 
$\wh{B}_n$ the Cartier divisors $D_i :=  F^1_{n,i} - (m+1) F_{n,0}$ for 
$1 \leq i \leq n-1.$ That the cycle $Z$ has the modulus condition means 
that $[q_{\ov{Z}}^*\circ \bar{f} ^* (D_i)] \geq 0$ for an index $i$. By the 
commutativity of the above diagram, this means that the Cartier divisor
${f^N_Z}^* [ q_{\ov{V}} ^* (D_i) ]\geq 0$. By  Lemma~\ref{lem:surjm}, 
this implies that $[q_{\bar{V}}^* (D_i)] \geq 0$, 
which is the modulus condition for $V$. If $Z$ satisfies the modulus 
condition $M_{sum}$, we use the same argument by replacing $F^1_{n,i}$ with 
$F^1_n$. This finishes the proof of the proposition.
\end{proof}

\begin{remk}\label{remk:n=1}
In Proposition~\ref{prop:P0}, if $X$ is projective, $Y = {\rm Spec}(k)$ and 
$n=1$, then $V$ is always a single point. 
To see this, let $Z \subset X \times B_1= X \times {\G}_m$ be an admissible 
irreducible closed subvariety. Let $V = p(Z)$, where 
$p: X \times {\G}_m \to {\G}_m$ is the projection. 

Since $X$ is complete, $p$ is a closed map. Hence, $V=p(Z)$ is an irreducible 
closed subvariety of $\G_m$. But the only closed subvarieties of $\G_m$ are 
finite subsets or all of $\G_m$. On the other hand, if $\ov Z$ is the
closure of $Z$ in $X \times {\A}^1$, then the modulus condition implies
that $\ov Z \cap (X \times \{ t = 0 \}) = \emptyset$. This implies that
$V$ must be a proper subset and hence a finite subset. Since $V$ is 
irreducible, consequently $V$ must be a non-zero single point.

Hence $Z = W \times \{ * \}$ for a closed subvariety $W \subset X$, and a 
closed point $\{ * \} \in \G_m$. Conversely, any such variety is admissible. 
This classifies all admissible cycles $Z$ when $X$ is projective and $n=1$.

For $n > 1$, all we can say is that $Z$ is contained in $X \times V$, where 
$V$ is admissible in $\TZ_s (k, n;m)$ for a suitable $s$. 
\end{remk}
\subsection{Homotopy variety}\label{subsection:homotopy}
Now we want to construct the ``homotopy variety". First, we need the 
following simple result:
\begin{lem}\label{lem:SL_n}
Let $SL_{r+1, k}$ be the $(r+1) \times (r+1)$ special linear group over $k$, 
and let $\eta$ be the generic point of the $k$-variety $SL_{r+1, k}$. Let $K$ 
be its function field (this is a purely transcendental extension of $k$). Let 
$SL_{r+1, K} := SL_{r+1, k} \otimes _k K$ be the base change. Then, there is 
a morphism of $K$-varieties $\phi: \square ^1 _K \to SL_{r+1, K}$ such that 
$\phi (0)$ is the identity element, and $\phi (\infty)$ is the generic point 
$\eta$ considered as a $K$-rational point.
\end{lem}
\begin{proof} By a general result on the special linear groups, every element 
of $SL_{r+1, K}$ is generated by the transvections $E_{ij} (a)$, $i \not = j$, 
$a \in K$, that are $(r+1) \times (r+1)$ matrices whose diagonal entries are 
$1$, the $(i,j)$-entry is $a$, and all other entries are zero. 

For each pair $(i,j)$, the collection $\{ E_{ij} (a) | a \in K \}$ forms a 
one-parameter subgroup of $SL_{r+1, K}$ isomorphic to $\mathbb{G}_{a, K}$. 
Thus, for each fixed $b \in K$, define $\phi_{ij} ^b: \mathbb{A}^1 _K \to 
SL_{r+1, K}$ by $\phi_{ij} ^b (y) := E_{ij} (by)$. 

Express the $K$-rational point $\eta$ of $SL_{r+1, K}$ as the (ordered) 
product
\[
\eta = \prod_{l=1} ^p E_{i_l j_l} (a_l),\ \ \ \mbox{ for some }  i_l, j_l \in 
\{ 1, 2, \cdots, r+1 \}, \ a_l \in K,
\]
and define $\phi' : \mathbb{A}^1 _K \to SL_{r+1, K}$ by 
$\phi'= \prod_{l=1} ^p \phi _{i_l j_l} ^{a_l}$. By definition, we have 
$\phi' (0)= {\rm Id}$ and $\phi' (1) = \eta$. Composing with the automorphism 
$\sigma: \mathbb{P}^1 _K \to \mathbb{P} ^1 _K$ given by 
$y \mapsto {y}/({y-1})$, that isomorphically maps $\square^1 _K$ to 
$\mathbb{A}^1 _K$, we obtain $\phi = \phi' \circ \sigma| : \square^1 _K \to 
SL_{r+1, K}$. This $\phi$ satisfies the desired properties.
\end{proof}
Recall that one consequence of \lemref{lem:M5*} is that the additive cycle
complex with modulus $m$ can also be defined as a complex whose level $n$
term is the free abelian group of integral closed subschemes $Z \subset
X \times {\wt B}_n$ which have good intersection property with all faces,
and which satisfy the appropriate modulus condition on $X \times {\wh B}_n$.
The following lemma uses this particular definition of the additive cycle
complex. 
\begin{lem}\label{lem:HV}
Let $K$ be the function field of $SL_{r+1, k}$, and $\phi: \square^1 _K \to 
SL_{r+1, K}$ be as in the previous lemma. Let $SL_{r+1, K}$ act on 
$\mathbb{P}^r _K$ naturally. Consider the composition 
$H_n=p_{K/k} \circ pr_K '\circ \mu_{\phi}$ of morphisms 
\[
\xymatrix{ 
\mathbb{P}^r \times \mathbb{A}^1 \times \square^{n} _K \ar[r]^{\mu_{\phi}} &  
\mathbb{P}^r \times \mathbb{A}^1 \times \square^{n} _K \ar[r] ^{{\rm pr_K}'} &
\mathbb{P}^r \times \mathbb{A}^1 \times \square^{n-1} _K \ar[r] ^{ p_{K/k}} &
\mathbb{P}^r \times \mathbb{A}^1 \times \square^{n-1}_k}
\]
where
\[
\left\{\begin{array}{lll}
\mu_\phi (x, t, y_1, \cdots, y_n) := (\phi (y_1) x, t, y_1, \cdots, 
y_n), \\  
{\rm pr}' _K(x, t, y_1, \cdots, y_{n-1}) := (x, t, y_2, \cdots, y_{n-1}), \\ 
p_{K/k} : \mbox{ the base change.}
\end{array}
\right.
\]
Then for any $Z \in {\TZ}^q (\mathbb{P}^r_k, n; m)$, the cycle 
$H^* _n (Z)=\mu_{\phi} ^* \circ {{\rm pr} '} ^* (Z_K)$ is admissible, hence 
it is in ${\TZ}^q (\mathbb{P}^r _K, n+1;m)$. Similarly, $H_n ^*$ carries 
${\TZ}^q _{\mathcal{W}} (\mathbb{P} ^r _k, n;m)$ to 
${\TZ}^q _{\mathcal{W}_K} (\mathbb{P} ^r _K, n+1; m)$.
\end{lem}
\begin{proof} It is enough to prove the second assertion that for any 
irreducible admissible $Z$ in ${\TZ} ^q _{\mathcal{W}} (\mathbb{P} ^r, n; m)$, 
the variety $Z':=H_n ^* (Z)$, that we informally call the ``homotopy variety" 
of $Z$, satisfies the admissibility conditions of 
Definition~\ref{defn:AdditiveComplex}.
\\
\\
{\bf  Claim (1):} \emph{The variety $Z'$ intersects $W \times F_K$ properly 
for all 
$W \in \mathcal{W}$ and for each face $F$ of $B_{n+1}$.} \\

This follows from the arguments of S. Bloch and M. Levine in \cite{Bl1,Le} 
without modification. We provide its proof for sake of completeness. We use 
Lemma~\ref{lem:action} for this purpose. We may assume that $\mathcal{W}$ 
contains only one non-empty algebraic set $W$. There are cases to consider. \\
\noindent \emph{Case 1.} 
Suppose $F_K$ is of the form $F = \mathbb{A} ^1 \times \{ 0 \} \times F' _K$ 
for some face $F' _K \subset \square ^{n-1} _K$. In this case, 
$Z' \cap (W \times F_K)$ is nothing but 
$Z_K \cap (W \times \mathbb{A} ^1 \times F' _K)$ because 
$\phi(0) = {\rm Id} \in SL_{r+1, K}$. So, proper-intersection is obvious in 
this case.

\noindent \emph{Case 2.} 
Suppose $F_K$ is any other form. It comes from some $F \subset B_{n+1}$. 
We apply Lemma~\ref{lem:action} with $G = SL_{r+1, k}$, 
$X = \mathbb{P} ^r \times F$, $A = W \times F$, 
$B= {{\rm pr_k}' }^* (Z) \cap (\mathbb{P} ^r \times F)$, where $G$ acts on 
$X$ by acting trivially on $F$ and acting naturally on $\mathbb{P} ^r$. By 
Lemma~\ref{lem:action}, there is a non-empty open subset $U \subset SL_{r+1}$ 
such that for all $g \in U$, the intersection $g(A) \cap B$ is proper. By 
shrinking $U$ if necessary, we may assume that $U$ is invariant under taking 
the multiplicative inverses. Take $g = \eta ^{-1} \in U$, the inverse of the 
generic point. Thus, after base extension to $K$, the intersection of 
${\eta^{-1}} (W_K \times F_K)$ with 
${{\rm pr}' } ^* (Z_K) \cap (\mathbb{P} ^r \times F_K)$ is proper, which 
means ${\eta} ({{\rm pr}'} ^* (Z_K) \cap (\mathbb{P} ^r \times F_K))$ 
intersects properly with $W_K \times F_K$. But the intersection 
${\rm pr'} ^* (Z_K ) \cap (\mathbb{P} ^r \times F_K)$ is proper, as $Z$ was 
admissible. Hence, ${\eta}  ({\rm pr'} ^* (Z_K))$ intersects with 
$W_K \times F_K$ properly and consequently $Z'$ intersects with 
$W_K \times F_K$ properly. This proves the claim and hence $Z'$
has good intersection property. Thus we only need to show the modulus 
condition for $Z'$ to complete the proof of the lemma.
\\
\\
{\bf Claim (2):} \emph{$Z'$ satisfies the modulus condition on
${\P}^r \times {\wt B}_{n+1, K}$. }\\

We prove this using our containment lemma. In the following, we casually drop 
the automorphism $\tau : \mathbb{P} ^r \times \mathbb{A}^1 \times \square ^n 
\to \mathbb{P} ^r \times \mathbb{A}^1 \times \square ^n $ that maps 
$(x, t, y_1, \cdots, y_n)$ to $(x, t, y_2, \cdots, y_n, y_1)$ from our 
notations for simplicity.

Take $V = p(Z)$, where $p: \mathbb{P} ^r \times {\wt B}_n \to {\wt B}_n$ is 
the projection. Because $Z \subset p^{-1}(p(Z)) = \mathbb{P}^r \times V$, we 
have 
\begin{equation}\label{eqn:SL0*}
Z' = \mu_{\phi} ^* (Z \times \square_K ^1) \subset  \mu_{\phi} ^* 
(\mathbb{P} ^r \times V \times \square^1 _K) = \mathbb{P} ^r \times V \times 
\square^1 _K =: Z_1, {\rm say}.
\end{equation} 

Now, \propref{prop:P0} implies that $V$ is an irreducible admissible closed 
subvariety of ${\wt B}_n$. The flat pull-back property in turn implies that
$p^*([V]) = {\P}^r \times V$ is an irreducible admissible closed 
subvariety of ${\P}^r \times {\wt B}_n$. In particular, the modulus condition
holds for ${\P}^r \times V$. If $\ov V$ is the closure of $V$ in 
${\wh B}_n$, then commutativity of the diagram
\[
\xymatrix{
{\ov Z}^N_1 = \mathbb{P} ^r \times {\ov V}^N \times {\P}^1_K \ar[r] 
\ar[d] &
 \mathbb{P} ^r \times {\wh B}_{n+1, K} \ar[r] \ar[d] & 
{\wh B}_{n+1, K} \ar[d] \\
\mathbb{P} ^r \times {\ov V}^N \ar[r] & \mathbb{P} ^r \times 
{\wh B}_{n} \ar[r] & {\wh B}_{n}}
\]
now implies that $Z_1$ satisfies the modulus condition on ${\P}^r \times 
{\wt B}_{n+1, K}$ even though it is a degenerate additive cycle.
Furthermore, the admissibility of $Z$ and the fact that ${\mu}_{\phi}$ is 
an automorphism, imply that ${\ov Z'}$ intersects the Cartier divisors
$F^1_{n+1}$ and $F_{n+1, 0}$ properly. Thus we can use \eqref{eqn:SL0*} and
apply \propref{prop:restp} (with $``X" = {\P}^r_K$, 
$``Y" = Z'$ and $``V" = Z_1$) to 
conclude that $Z'$ satisfies the modulus condition. This completes 
the proof of the lemma.
\end{proof}  
\begin{lem}\label{lem:Hrelation}
The collection $H_{\bullet} ^*: {\TZ}^q (\mathbb{P} ^r _k , \bullet; m) \to 
{\TZ} ^q  (\mathbb{P} ^r _K, \bullet +1; m)$ is a chain homotopy satisfying
$\partial H^* + H^* \partial = p_{K/k} ^* (Z) - {\eta} (Z_K)$. 
The same is true for ${\TZ}_{\mathcal{W}} ^q$.
\end{lem}
\begin{proof} It is enough to prove the second assertion. This is 
straightforward: let 
$Z \in {\TZ}^q _{\mathcal{W}} (\mathbb{P} _k ^r , n ; m)$. 
Then
\begin{eqnarray*}
H ^*\partial (Z ) &=& H^* \left( \sum_{i=1} ^n (-1)^i 
( \partial_i ^{\infty} - \partial_i ^{0}) (Z) \right) \\
&=& \sum_{i=1} ^n (-1)^i  (\mu_{\phi} ^* {{\rm pr}'} ^* p_{K/k} ^*) 
(\partial_i ^{\infty} - \partial_i ^{0}) (Z)\\
&=& \sum_{i=1} ^n (-1)^i (\partial_{i+1} ^{\infty} - \partial_{i+1} ^{0}) 
(\mu_{\phi} ^* {{\rm pr}'} ^* p_{K/k} ^* (Z))\\
&=& - \sum_{i=2} ^{n+1} (-1)^i (\partial_i ^{\infty} - \partial_i ^{0} ) 
(H^* (Z)),\\
\partial H^* (Z) &=& \sum_{i=1} ^{n+1} (-1)^i 
(\partial_i ^{\infty} - \partial _i ^{0}) H^* (Z) \\
&=& \sum_{i=1} ^{n+1} (-1)^i (\partial _i ^{\infty} - \partial _i ^{0}) 
(H^* (Z))\\
&=& (-1) (\partial _1 ^{\infty} - \partial _1 ^{0} ) 
(H^* (Z)) + \sum_{i=2} ^{n+1} (-1)^ i 
(\partial _i ^{\infty} - \partial _i ^{0}) (H^* (Z)).
\end{eqnarray*} Hence, 
\[
(\partial H^* + H^* \partial )(Z) = (\partial _1 ^{0} - 
\partial _1 ^{\infty} ) (H^* (Z)) =  p_{K/k} ^* (Z) - {\eta} (Z_K).
\]
This proves the lemma.
\end{proof}
\subsection{Proof of the moving lemma for projective spaces}
\label{subsection:MLPP}
We are now ready to finish the proof of \thmref{thm:ML} for $\mathbb{P} ^r$.

By the Lemma~\ref{lem:Hrelation}, the base extension 
\[
p_{K/k} ^* : \frac{{\TZ} ^q (\mathbb{P}^r _k, \bullet ;m)}
{{\TZ}^q _{\mathcal{W}} 
(\mathbb{P} ^r _k, \bullet;m)} \to 
\frac{ {\TZ}^q (\mathbb{P} ^r _K, \bullet;m)}
{ {\TZ}^q _{\mathcal{W}_K} (\mathbb{P} ^r _K, \bullet;m)}
\]
is homotopic to the map ${\eta} p_{K/k} ^*$. Note for each admissible cycle 
$Z \in {\TZ} ^q (\mathbb{P} ^r _k, n; m)$, the cycle ${\eta} (Z_K)$ lies in 
${\TZ}^q _{\mathcal{W}} (\mathbb{P} ^r _K, n;m)$. We can prove it as before.

We may assume that $\mathcal{W}$ has only one nonempty algebraic set, say $W$.
In the Lemma~\ref{lem:action}, take $G = SL_{r+1}$, 
$X= \mathbb{P} ^r \times B_{n}$ where $G$ acts on $\mathbb{P} ^r _K$ 
naturally and $B_{n}$ trivially. Let $F$ be a face $B_n$. Let 
$A = W \times F$, $B = Z \cap (\mathbb{P} ^r \times F)$. Since $SL_{r+1}$ 
acts transitively on $\mathbb{P} ^r$, the map $G \times A \to X$ is 
surjective. Hence, by the Lemma~\ref{lem:action}, there is a non-empty open 
subset $U \subset G$ such that for all $g \in U$, the intersection 
$g(A) \cap B$ is proper in $X$. By shrinking $U$ further, we may assume that 
$U$ is closed under taking multiplicative inverse of $\eta$. Taking 
$g = \eta^{-1}$, the inverse of the generic point, we see that after base 
extension to $K$, the intersection of $\eta ^{-1} (W \times F)$ with 
$Z_K \cap (\mathbb{P} ^r \times F_K)$ is proper, which means 
${\eta} (Z_K \cap (\mathbb{P} ^r \times F_K))$ intersects $W_K \times F_K$ 
properly. Since $Z_K$ intersects with $\mathbb{P} ^r \times F_K$ properly by 
the assumption, we conclude that $\eta (Z_K)$ intersects $W_K \times F_K$ 
properly. 
Thus, $\eta (Z_K) \in {\TZ} ^q _{\mathcal{W}} (\mathbb{P} ^r _K, n; m)$. 
Hence, the induced map on the quotient 
\[
\eta p_{K/k} ^* : 
\frac{{\TZ} ^q (\mathbb{P}^r _k , \bullet ;m)}{{\TZ}^q _{\mathcal{W}} 
(\mathbb{P} ^r _k, \bullet;m)} \to 
\frac{ {\TZ}^q (\mathbb{P} ^r _K, \bullet;m)}
{ {\TZ}^q _{\mathcal{W}_K} (\mathbb{P} ^r _K, \bullet;m)}
\]
is zero. Hence the base extension $p_{K/k} ^*$ induces a zero map on homology 
since it is homotopic to the zero map. 

On the other hand, by the spreading lemma, \propref{prop:SL}, the chain map 
$p_{K/k} ^*$ is 
injective on homology. Hence the quotient complex 
${\TZ} ^q (\mathbb{P}^r _k , \bullet ;m)/{{\TZ}^q _{\mathcal{W}}} 
(\mathbb{P} ^r _k, \bullet;m)$ must be acyclic. This proves \thmref{thm:ML}
for the projective spaces. 
$\hfill \square$
\\
\section{Generic projections and moving lemma for projective varieties}
\label{section:moving for varieties}
\subsection{Generic projections}
This section begins with a review of some facts about linear projections. In 
combination with the moving lemma for $\mathbb{P}^r$, that we saw in the 
previous section, we prove the moving 
lemma for general smooth projective varieties.
\begin{lem}\label{lem:LP}
Consider two integers $N>r>0$. Then for each linear subvariety 
$L \subset \mathbb{P} ^N$ of dimension $N-r-1$, there exists a canonical 
linear projection morphism 
$\pi_L: \mathbb{P} ^N \backslash{L} \to \mathbb{P}^r$.
\end{lem}
\begin{proof}
Fix the coordinates $x  = (x_0; \cdots; x_N)$ of $\mathbb{P}^N$. A linear 
subvariety $L$ is given by $(r+1)$ homogeneous linear equations in $x$ whose 
corresponding $(N+1) \times (r+1)$ matrix $A$ has the full rank $r+1$. Take 
the reduced row echelon form of $A$ whose rows are the linear homogeneous 
functions $P_0 (x), \cdots, P_r (x)$ in $x$. 

For $x \in \mathbb{P}^N \backslash L$, define 
$\pi_L (x) := (P_0 (x); \cdots ; P_r (x))$. Since $x \not \in L$, we have 
some $P_i (x) \not = 0$ so that the map $\pi_L$ is well-defined. By 
elementary facts about reduced row echelon forms and row equivalences, the 
subvariety $L$ uniquely decides this map $\pi_L$ in this process.
\end{proof}
Let $X$ be a smooth projective $k$-variety. Let $r= \dim X$. Suppose that we 
have an embedding $X \hookrightarrow \mathbb{P}^N$ for some $N>r$. Consider 
$\pi_L : \mathbb{P} ^N \backslash L \to \mathbb{P}^r$. Whenever 
$L \cap X = \emptyset$, we have a finite morphism 
$\pi_{L,X}:= \pi_{L} |_X: X \to \mathbb{P}^r$. Such $L$'s form a non-empty 
open subset $Gr(N-r-1,N)_X$ of the Grassmannian $Gr (N-r-1, N)$.
Note that such a map ${\pi}_L$ is automatically flat since $X$ is smooth
({\sl cf.} \cite[Ex. III-10.9, p. 276]{Hart}). In particular, the pull-back
${\pi}^*_{L, X}$ and push-forward ${\pi}_{L, X *}$ are defined by
\thmref{thm:basic}.

For any closed integral admissible cycle $Z$ on $X \times B_n$, define 
$\widetilde{L} (Z)$ to be 
$$ \widetilde{L}(Z):= \pi_{L,X} ^* ({\pi_{L,X}} _* ([Z])) - [Z].$$ 
Extending this map linearly, this defines a morphism of complexes
\[
\widetilde{L} : {\TZ}^q (X, \bullet;m) \to {\TZ}^q (X, \bullet;m).
\]

\subsection{Chow's moving lemma}
Recall that for two locally closed subsets $A, B$ of pure codimension $a$ and 
$b$, the \emph{excess} of $A, B$ is defined to be 
\[
e(A, B): =  \max \{ a+b - {\rm codim}_X (A \cap B), 0 \}.
\]
That the intersection $A \cap B$ is proper means $e(A,B) = 0$. If $A, B$ are 
cycles, then we define $e(A,B) := e({\rm Supp} (A), {\rm Supp} (B))$. The 
excess 
measures how far an intersection is from being proper.
\begin{lem}[{{\sl cf.} \cite[Lemma~1.12]{KL}}]\label{lem:Chowlemma}
Let $X \subset \mathbb{P}^N$ be a smooth closed projective $k$-subvariety of 
dimension $r$. Let $Z, W$ be cycles on $X$. Then there is a non-empty open 
subscheme $U_{Z, W} \subset Gr(N-r-1, N)_X$ such that for each field 
extension $K \supset k$ and each $K$-point $L$ of $U_{Z, W}$, we have 
$$e (\widetilde{L} (Z), W) \leq \max \{ e(Z,W) -1, 0 \}.$$
\end{lem}

For its proof, see J. Roberts \cite[Main Lemma, p. 93]{Ro}, or 
\cite[Lemma 3.5.4, p. 96]{Le} for a slightly different but equivalent 
version. The point of the projection business is the following lemma:
\begin{lem}\label{lem:consequence of projection}
Let $X$ be a smooth projective $k$-variety, and let $\mathcal{W}$ be a finite 
set of locally closed algebraic subsets of $X$. Let $m, n\geq 1$, $q \geq 0$ 
be integers. Let $e: \mathcal{W} \to \mathbb{Z}_{\geq 0}$ be a set-theoretic 
function. Define $e-1: \mathcal{W} \to \mathbb{Z} _{\geq 0}$ by
\[
(e-1)(W):= \max \{ e(W) -1, 0 \}.
\]
Let $K$ be the function field of $Gr(N-r-1, N)$, and let 
$L_{gen} \in Gr (N-r-1, N) _X(K)$ be the generic point. 
Then, the map
\[
\widetilde{L}_{gen} :  Tz^q (X, \bullet;m) \to Tz^q (X_K, \bullet; m)
\]
maps ${\TZ}^q _{\mathcal{W}, e} (X, \bullet; m)$ to 
${\TZ}^q _{\mathcal{W}_K, e-1} (X_K, \bullet;m)$.
\end{lem}
\begin{proof} The arguments of \cite[Lemma 1.13, p. 84]{KL} or 
\cite[\S 3.5.6, p. 97]{Le} work in this additive context without change. The 
central idea is to use a variation of Chow's moving lemma as in 
Lemma~\ref{lem:Chowlemma}.
\end{proof}
\subsection{Proof of the moving lemma}
\begin{proof}[Proof of Theorem~\ref{thm:ML}]
Let $L_{gen}$ be the generic point of the Grassmannian $Gr(N-r-1, N)$ as in 
Lemma~\ref{lem:consequence of projection}. Then, for each function 
$e: \mathcal{W} \to \mathbb{Z}_{\geq 0}$, the morphism
\[
\widetilde{L}_{gen} = \pi ^* _{L_{gen}} \circ 
{\pi _{L_{gen}}}_* - p_{K/k} ^* : 
\frac{{\TZ}^q _{\mathcal{W}, e}(X, \bullet;m)}
{ {\TZ}^q _{\mathcal{W}, e-1} (X, \bullet;m)} \to 
\frac{{\TZ}^q _{\mathcal{W}_K, e} (X_K, \bullet;m)} 
{ {\TZ} ^q _{\mathcal{W}_K, e-1} (X_K, \bullet; m)}
\]
is zero. Hence $\pi^* L_{gen} \circ {\pi _{L_{gen}} }_*$ is equal to the base 
extension morphism $p_{K/k} ^*$ on the quotient complex.

On the other hand, $\pi^* L_{gen} \circ {\pi _{L_{gen}} }_*$ factors as
\[
\frac{{\TZ}^q _{\mathcal{W}, e}(X, \bullet;m)}
{ {\TZ}^q _{\mathcal{W}, e-1} (X, \bullet;m)} 
\overset{{\pi_{L_{gen}}}_*}{\longrightarrow} 
\frac{ {\TZ}^q _{\mathcal{W'}, e'} (\mathbb{P} ^r _K, \bullet; m)}
{{\TZ}^q _{\mathcal{W'}, e'-1} (\mathbb{P} ^r _K, \bullet;m)} 
\overset{{\pi _{L_{gen}} ^*}}{\longrightarrow} 
\frac{{\TZ}^q _{\mathcal{W}_K, e} (X_K, \bullet;m)}
{{\TZ} ^q _{\mathcal{W}_K, e-1} (X_K, \bullet; m)},
\]
where $\mathcal{W}'$ and $e'$ are defined as follows: 
for each $W \in \mathcal{W}$, the constructible subset 
$\pi_{L_{gen}} (W)$ can be written as
\[
\pi_{L_{gen}} (W) = W_1 ' \cup \cdots \cup W_{i_W}'
\] 
for some $i_W \in \mathbb{N}$ and locally closed irreducible sets $W_j '$ in 
$\mathbb{P} _K ^r$. Let 
$d_j = \codim_{\mathbb{P}^n _K} (W_j ') - \codim _X (C)$. 
Let $\mathcal{W}' = \{ W_j ' | W \in \mathcal{W} \}$. 
Define $e' : \mathcal{W}' \to \mathbb{Z} _{\geq 0}$ by the rule 
$e' (W_j '):= e(W) + d_j$. 
We have already shown in Section ~\ref{subsection:MLPP} that the moving lemma 
is true for all projective spaces.
In particular, for all functions 
$e' : \mathcal{W}' \to \mathbb{Z} _{\geq 0}$, the complex in the middle 
\[
\frac{{\TZ}^q _{\mathcal{W}'_K, e'} (\mathbb{P} ^r _K, \bullet ; m) }
{ {\TZ}^q _{\mathcal{W}'_K, e' -1} (\mathbb{P} ^r _K, \bullet ;m)}
\]
is acyclic (see Remark~\ref{remk:chain}). Hence, the base extension map
\[
p_{K/k}  ^*: \frac{{\TZ}^q _{\mathcal{W}, e}(X, \bullet;m)}
{{\TZ}^q _{\mathcal{W}, e-1} (X, \bullet;m)} \to  
\frac{{\TZ}^q _{\mathcal{W}_K, e} (X_K, \bullet;m)} 
{ {\TZ} ^q _{\mathcal{W}_K, e-1} (X_K, \bullet; m)}
\]
is zero on homology. Consequently, by induction, the base extension map
\[
p_{K/k} ^*: 
\frac{{\TZ}^q (X, \bullet;m)}{{\TZ}^q _{\mathcal{W}} (X, \bullet;m)} 
\to \frac{{\TZ}^q (X_K, \bullet ;m)}{{\TZ}^q _{\mathcal{W}_K} (X_K, \bullet;m)}
\]
is zero on homology. On the other hand, this map is also injective on 
homology by \propref{prop:SL}. This happens only when 
\[
\frac{{\TZ}^q (X, \bullet;m)}{{\TZ}^q _{\mathcal{W}} (X, \bullet;m)}
\]
is acyclic, that is, the inclusion
\[
{\TZ}^q _{\mathcal{W}} (X, \bullet;m) \to {\TZ}^q (X, \bullet ;m)
\]
is a quasi-isomorphism. This finishes the proof of \thmref{thm:ML}.
\end{proof}

\section{Application to contravariant functoriality}\label{section:CF}
In this section, we prove the following general contravariance property of
the additive higher Chow groups as an application of the moving lemma.
\begin{thm}\label{thm:funct}
Let $f: X \to Y$ be a morphism of quasi-projective varieties over $k$, where
$Y$ is smooth and projective. Then there is a pull-back map
$$f^*: {\TH}^q(Y, n; m) \to {\TH}^q(X, n; m)$$ such that for a composition
$X \xrightarrow{f} Y \xrightarrow{g} Z$ with $Y$ and $Z$ smooth and projective,
we have 
\[
{(g \circ f)}^* = f^* \circ g^* : {\TH}^q(Z, n; m) \to  {\TH}^q(X, n; m).
\]
\end{thm}     
Before proving this functoriality, we mention one more consequence 
of our containment lemma (\propref{prop:restp}).
\begin{cor}\label{cor:lci}
Let $X \xrightarrow{i} Y$ be a regular closed embedding of quasi-projective 
but not necessarily smooth varieties over $k$. Then there is a Gysin chain 
map of additive cycle complexes
\[
i^* : {\TZ}^q_{\{X\}}(Y, \bullet;m) \to {\TZ}^q(X, \bullet;m).
\]
\end{cor}
\begin{proof} Let $\iota: Z \subset Y \times B_n$ be a closed irreducible 
admissible subvariety in ${\TZ}_{\{X \}}^q (Y, n; m)$. By assumption, 
$Z$ intersects all faces $X \times F$ properly. Hence 
the abstract intersection product of cycles 
$(X\times B_n )\cdot Z = [\iota^* (X\times B_n)] \in z^q (X \times B_n)$ 
is well defined, and the intersection formula for the regular embedding
implies that this intersection product commutes with the boundary maps
(\cite[\S 2.3, \S 6.3]{F}). We want this cycle to be $i^* (Z)$. Thus we only 
need to show that each component
of $Z \cap (X \times B_n)$ satisfies the modulus condition in order for
$i^*$ to be a map of additive cycle complexes. Since $X \times {\wh B}_n$
clearly intersects $F^1_n$ and $F_{n, 0}$ properly on $Y \times {\wh B}_n$,
this modulus condition follows directly from \propref{prop:restp}, for $Z$ 
has the modulus condition.
\end{proof} 
\begin{proof}[Proof of Theorem~\ref{thm:funct}]
We do this by imitating the proof of \cite[Theorem~4.1]{Bl1}. So, 
let $f: X \to Y$ be a map as in Theorem~\ref{thm:funct}.
Such a morphism can be factored as
the composition $X \overset{gr_f}{\to} X \times Y \overset{pr_2}{\to} Y$, 
where $gr_f$ is the graph of $f$ and $pr_2$ is the projection.
Notice that $pr_2$ is a flat map and moreover, the smoothness of $Y$ implies 
that $gr_f$ is a regular closed embedding. Let ${\Gamma}_f \subset X \times Y$
denote the image of $gr_f$ which is necessarily closed. 

For $0 \leq i \leq \dim Y$, let $Y_i$ be the Zariski closure of the 
collection of all points $y \in Y$ such that $\dim f^{-1} (y) \geq i$. We use 
the convention that $\dim \emptyset = -1$. Let $\mathcal{W}$ be the 
collection of the irreducible components of all $Y_i$. Then $\sW$ is a 
finite collection. 
\\
\\  
{\bf Claim : } \emph{Let $Z \in {\TZ}^q _{\mathcal{W}} (Y, n;m)$ be an 
irreducible admissible closed subvariety of $Y \times B_n$. Then 
$(pr_2 \times {\rm Id}_{B_n})^{-1} (Z) = 
X \times Z$ in $X \times Y \times B_n$ is an admissible closed subset that 
intersects $\Gamma_f \times F$ properly in $X \times Y \times B_n$ for all 
faces $F \subset B_n$. This gives a chain map
\[
{pr_2}^* : {\TZ}_{\mathcal{W}} ^q (Y, \bullet ; m) \to 
{\TZ}^q _{\{\Gamma_f\}} (X \times Y, \bullet, m).
\]}

That $(pr_2 \times {\rm Id}_{B_n})^{-1} (Z) = X \times Z$ is admissible is 
obvious by \cite[\S 3.4]{KL}. Since $Z$ intersects $W \times F$ properly for 
all $W \in \mathcal{W}$ and faces $F \subset B_n$, we have
\[
\dim \widetilde{Z}_i \leq \dim Y_i + \dim F - q, \ \ \mbox{ where } 
\widetilde{Z}_i := Z \cap (Y_i \times F).
\]

Now, $(X \times Z) \cap (\Gamma_f \times F) = 
\cup_i X \times \widetilde{Z}_i$, and for each $i$ we have
\begin{eqnarray*}
\dim (X \times \widetilde{Z}_i) & = & \dim X + \dim \widetilde{Z}_i \\
&\leq & \dim X + \dim F - q = \dim (\Gamma_f \times F) - q.
\end{eqnarray*} 
Hence ${\rm codim}_{\Gamma_f \times F} (X \times Z) \cap (\Gamma_f \times F) 
\geq q$ and we have the desired map ${pr_2}^*: 
{\TZ}_{\mathcal{W}} ^q (Y, n; m) \to 
{\TZ}^q _{\{\Gamma_f\}} (X \times Y, n;m)$ for 
each $n \geq 1$.  That this gives a chain map is obvious since $f^*$ clearly 
commutes with the boundary maps. This proves the Claim.

The pull-back map $f^*$ is now given by composing $pr^*_2$ with the
Gysin map $gr^*$ of Corollary~\ref{cor:lci} and then using the moving lemma, 
Theorem~\ref{thm:ML}. The composition law can be checked directly from the
construction of $f^*$. This completes the proof of \thmref{thm:funct}.
\end{proof}
\section{Algebra structure on additive higher Chow groups}\label{section:WD}
In this section, we consider an algebra structure on the additive higher Chow 
groups of smooth projective varieties. This algebra structure corresponds to
the exterior product on the cohomology of the sheaves of 
absolute K{\"a}hler differentials on the smooth projective varieties.
We show that this algebra structure on the additive higher Chow groups is
compatible with the module structure on these groups for the ordinary
Chow ring of the variety. We draw some important consequences of this
towards the end of this section. We shall deduce our algebra structure 
on the additive higher Chow groups as a consequence of
the following general result whose proof will occupy most of this section.

\begin{prop}\label{prop:wedge-product}
Let $X$ and $Y$ be smooth projective varieties over a field $k$. 
Then there exists an external wedge product on the additive higher Chow groups 
\begin{equation}\label{eqn:EWP}
\ov {\wedge} : \TH^{q_1} ( X, n_1; m) \otimes_{\mathbb{Z}} 
\TH^{q_2} (Y, n_2; m) \to \TH^{q} (X \times Y, n ; m),
\end{equation}
where $q = q_1 + q_2 - 1$, $n = n_1 + n_2 - 1$, and 
$q_i, n_i, m \geq 1$ for $i = 1,2$. In the case of $X = Y$, one has  
\begin{equation}\label{eqn:anticommute}
\xi {\ov {\wedge}} \eta = 
(-1)^{(n_1-1)(n_2-1)} \eta \ov {\wedge} \xi
\end{equation} 
for all classes $\xi \in \TH^{q_1} (X, n_1;m)$ and $\eta \in \TH^{q_2} 
(X, n_2;m)$.
\end{prop}
\subsection{External wedge product}\label{subsection:external-wedge}
The external wedge product is based on the product map
$\mu : \G_m \times \G_m \to \G_m$ which clearly extends to the product
map 
\begin{equation}\label{eqn:product}
\mu : \G_m \times \P^1 \to \P^1.
\end{equation}
Note that this product defines a $\G_m$-action on $\P^1$ and hence is a
smooth map.

We define the external product at the level of cycle complexes in the
following way.
\begin{equation}\label{eqn:exproduct}
\xymatrix{ 
X \times \mathbb{G}_m \times \square^{n_1 -1} \times 
Y \times \mathbb{G}_m \times \square^{n_2 -1} 
\ar[r] ^{   \ \ \ \ \ \ \tau} \ar[rd]_{\mu} &  X \times Y \times 
\mathbb{G}_m \times \mathbb{G}_m \times \square^{n -1}  
\ar[d] ^{1 \times 1 \times \mu \times 1}\\
& X \times Y \times \mathbb{G}_m \times \square^{n -1},}
\end{equation}
where $\tau$ is the transposition map 
$(x,t,y, x', t', y') \mapsto (x, x', t, t', y, y')$. We denote the composite
map also by $\mu$.

Let $V_1 \in  {\TZ}^{q_1}(X, n_1 ; m)$ and $V_2  \in  {\TZ}^{q_2}(Y, n_2 ; m)$
be two irreducible admissible cycles. Define ${\mu}_*(V_1 \times V_2)$ to be
the Zariski closure of ${\mu}(V_1 \times V_2)$ in $X \times Y \times B_n$.
We first claim that ${\rm codim}_{X \times Y \times B_n}
({\mu}_*(V_1 \times V_2)) = q$, or, equivalently,
$\dim ({\mu}(V_1 \times V_2)) = \dim (V_1) + \dim (V_2)$.

This is obvious if one of the $V_i$'s lie in a fiber of the projection
map to $\G_m$. Otherwise, the modulus condition implies that none of these
can be of the form $W \times \G_m$. Thus, the set of points of $W$ such that 
the fiber of $V_i$ is $\G_m$ must be nowhere dense. In particular, there is a 
dense subset of closed points of $W$ such that the fiber of $V_i$ over this 
subset must be nowhere dense in $\G_m$. But then this fiber must be finite. 
Hence for any
$c \in \G_m$ in a dense open subset, there are points $(a, t_1) \in 
\G_m \times X \times \square^{n_1 -1}$ such that
$(a', t_1) \notin \G_m \times X \times \square^{n_1 -1}$ and $aa' \neq c$. 
Also, we have then $(ca^{-1}, t_2) \in \G_m \times Y \times 
\square^{n_1 -1}$ as $V_i$ map dominantly onto an open subset of $\G_m$.  
But then we see that $(a, ca^{-1}, t_1, t_2) \in V_1 \times V_2$ but
$(a', ca'^{-1}, t_1, t_2) \notin V_1 \times V_2$.
On the other hand, we have ${\mu}(a, ca^{-1}, t_1, t_2) =
{\mu}(a', ca'^{-1}, t_1, t_2)$.  
This implies that $V_1 \times V_2 \neq {\mu}^{-1}({\mu}(V_1 \times V_2))$.
Since $\mu$ is flat of relative dimension one, this implies that
$\dim ({\mu}(V_1 \times V_2)) = \dim (V_1) + \dim (V_2)$,
proving the claim. 

Thus we have shown that if $V_1 \in  {\TZ}^{q_1}(X, n_1 ; m)$ 
and $V_2  \in  {\TZ}^{q_1}(Y, n_2 ; m)$ are two irreducible admissible 
additive 
cycles, then ${\mu}_*(V_1 \times V_2)$ is a closed subvariety of 
$X \times Y \times B_n$ of codimension $q$. Our aim is to show the
admissibility of ${\mu}_*(V_1 \times V_2)$ as an additive cycle which we do
in several steps. Let us denote ${\mu}_*(V_1 \times V_2)$ by $Z$.
\begin{lem}\label{lem:proper-int}
The cycle $Z$ has the proper intersection property with all faces of
$B_n$.
\end{lem}
\begin{proof} To show the good intersection property, it is enough to 
intersect the Zariski-dense open subset $\mu (Z_1 \times Z_2)$ with 
$X \times Y \times F$ for any face $F$ of $B_n$. Write 
$F = \mathbb{G}_m \times F_1 \times F_2$ for some faces 
$F_1 \subset \square^{n_1 -1}$ and $F_2 \subset \square ^{n_2 -1}$.

Since the multiplication $\mu$ is equivariant with respect to all face maps 
$\partial_i ^{\epsilon}$ given by the intersection with a codimension 
$1$-face of $\square^{n_i-1}$, and since the faces $F_i$ are obtained by 
intersecting a multiple number of those codimension $1$-faces, we immediately 
see that $\mu (V_1 \times V_2)$ intersects $X \times Y \times F$ properly if 
$V_1$ intersects $X \times \mathbb{G}_m \times F_1$ properly and
$V_2$ intersects $Y \times \mathbb{G}_m \times F_2$ properly. But
this is indeed the case.
\end{proof}
\begin{prop}\label{prop:product-modulus} 
The cycle $Z$ satisfies the modulus condition in $X \times Y \times B_n$.
\end{prop}
\begin{proof}
For the structure morphisms $p: X \to {\rm Spec} (k)$ and 
$p': Y \to {\rm Spec} (k)$,
consider $W_1 = (p \times {\rm Id}_{B_{n_1}  }) (V_1)$ and 
$W_2 = (p' \times {\rm Id}_{B_{n_2}}) (V_2)$. 
By the admissibility of the projective images, Proposition~\ref{prop:P0},
$W_1$ and $W_2$ are admissible in $B_{n_1}$ and $B_{n_2}$ respectively.
In particular, $X \times W_1$ and $Y \times W_2$ are admissible cycles
by the flat pull-back ({\sl cf.} \thmref{thm:basic}). 
Now, we have
\[
V_1 \subset X \times W_1, \ \ V_2 \subset Y \times W_2,
\]
which implies that
\begin{equation} 
Z = \mu_* (V_1 \times V_2) \subset X \times Y \times \mu_* (W_1 \times W_2). 
\end{equation} 
Since $\ov Z$ intersects $F^1_n$ and $F_{n, 0}$ properly in $X \times Y \times
{\wh B}_n$, we can use \propref{prop:restp} to conclude that $Z$ has
the modulus condition if $\mu_* (W_1 \times W_2)$ has. 
That is, we reduce to the case when $X= Y = {\rm Spec} (k)$.

We first dispose of the case of the modulus condition $M_{sum}$ as it is
relatively straightforward.
Let $\ov{V_1} \subset \ov{B}_{n_1}, \ov{V_2}\subset \ov{B}_{n_2},$ and 
$\ov{Z}\subset \ov{B}_{n}$ be the Zariski closures of $V_1, V_2$, and $Z$
respectively. Take their normalizations $\nu_{\ov{V}_1} : \ov{V}_1 ^N \to 
\ov{V}_1$, $\nu_{\ov{V}_2} : \ov{V}_2 ^N \to \ov{Z}_2$, and $\nu_{\ov{Z}} : 
\ov{Z} ^N \to \ov{Z}$. By \cite[Lemma 3.1]{KL}, the product of two reduced 
normal finite type $k$-schemes is again normal over perfect fields. Thus, the 
morphism
\[
\nu:= \nu_{\ov{V}_1} \times \nu_{\ov{V}_2} : \ov{V}_1 ^N \times \ov{V}_2 ^N 
\to \ov{V}_1 \times \ov{V}_2 = \ov{V_1 \times V_2}
\]
is a normalization, under the identification $\mathbb{A}^2 \times 
\square^{n-1} = \mathbb{A} ^1 \times \square ^{n_1-1} \times \mathbb{A} ^1 
\times \square^{n_2 -1}.$ Thus, we can regard 
$\ov{V}_1 ^N \times \ov{V}_2 ^N$ as $\ov{V_1 \times V_2} ^N$. 
This gives the following diagram:
\[
\xymatrix{
\ov{V}_1 ^N \times \ov{V}_2 ^N \ar[r] ^{\nu} \ar[d]^{\ov{\mu} ^N} & 
\ov{V}_1  \times \ov{V}_2 \ar[r] ^{\iota_1 \times \iota_2 \ \ \ } 
\ar[d] ^{\ov{\mu}} & \mathbb{A}^2 \times \left( \mathbb{P} ^1 \right)^{n-1} 
\ar[d] ^{\mu} \\
\ov{Z}^N \ar[r] ^{\nu_{\ov{Z}}} & \ov{Z} 
\ar[r] ^{\iota \ \ \ \  \ \  } & \mathbb{A}^1 \times 
\left( \mathbb{P} ^1 \right) ^{n-1},}
\]
where $\ov{\mu}$ is the restriction $\mu|_{\ov{V}_1 \times \ov{V}_2}$, and 
$\ov{\mu} ^N$ is given by the universal property of the normalization 
$\nu_{\ov{Z}}$. Note that the restriction $\ov{\mu}: \ov{V}_1 \times 
\ov{V}_2 \to \ov{Z}$ is a surjective morphism. Hence, $\ov{\mu}^N$ is also 
surjective.

Let $(t_1, t_2, y_1, \cdots, y_{n-1}) \in \mathbb{A}^2 \times 
\left( \mathbb{P} ^1 \right) ^{n-1}$ and 
$(w, y_1, \cdots, y_{n-1}) \in 
\mathbb{A} ^1 \times \left( \mathbb{P} ^1 \right) ^{n-1}$ be the coordinates.

Consider the Cartier divisor 
$D := \sum_{i=1} ^{n-1} \{ t_i = 1 \} - (m+1) \{ w = 0 \}$ on 
$\mathbb{A}^1 \times \left( \mathbb{P} ^1 \right) ^{n-1}$. Then, as a Cartier 
divisor on $\mathbb{A} ^2 \times \left( \mathbb{P} ^1 \right) ^{n-1}$, we have 
\begin{eqnarray*}
\mu^* D & = &\sum_{i=1} ^{n-1} \{ y_i = 1 \} - (m+1) \{ t_1 = 0\} - 
(m+1) \{ t_2 = 0 \} \\
& = &  \left( \sum_{i=1 } ^{n_1 -1} \{ y_i = 1 \}  - (m+1) \{ t_1 = 0 \} 
\right) + 
\\
& & \left( \sum_{i={n_1}} ^{ n-1} \{ y_i = 1 \} - 
(m+1) \{ t_2 = 0 \} \right) \\
 & =: & D^1 + D^2,
\end{eqnarray*}  
where we note that $n-1 = (n_1 -1) + (n_2 -1)$. Note that by pulling back 
along $\nu = \nu _{\ov{V}_1} \times \nu_{\ov{V}_2}$, we see that 
\begin{equation}\label{eqn:middle} 
((\iota_1 \times \iota_2 ) \circ \nu)^* (\mu^* D ) \geq 0
\end{equation} 
since we have 
\begin{eqnarray*}
((\iota_1 \times 1) \circ (\nu_{\ov{V}_1 } \times 1) ) ^* D^1 \geq 0, \\
((1 \times \iota_2) \circ (1 \times \nu_{\ov{V}_2})) ^* D^2 \geq 0,
\end{eqnarray*}
by the modulus condition $M_{sum}$ of $V_1 \times B_{n_2}$ and 
$B_{n_1} \times V_2$, regarding $V_1 \times B_{n_2}$ as a 
cycle in ${\TZ}_{p_1 + n_2} (B_{n_2}, n_1; m)$, and similarly for 
$B_{n_1} \times V_2$.  

This inequality ~\eqref{eqn:middle} is equivalent to 
$(\ov{\mu} ^N )^*  \left( (\iota \circ \nu_{\ov{Z}}) ^* D \right) \geq 0$ 
by the commutativity of the diagram. Then, by Lemma~\ref{lem:surjm}
applied to the surjective morphism $\ov{\mu} ^N$, we get 
$(\iota \circ \nu_{\ov{Z}}) ^* D  \geq 0.$ This is the modulus $m$ condition 
$M_{sum}$ for $Z = \mu_* (V_1 \times V_2)$. 

Now we prove the $M_{ssup}$ condition for $Z$ if 
$V_i$'s satisfy this modulus condition. This is a much more delicate case and
we prove it in several steps.
First suppose that $V_1$ or $V_2$ is contained in the locus of $\{t=a\}$ for a
closed point $a \in \mathbb{G}_m$.
By symmetry, we may then assume that $V_1 = \{ a \} \times W_1$ for a 
closed irreducible subvariety $W_1 \subset \square ^{n_1 -1}$ 
intersecting all faces of $\square^{n_1-1}$ properly. Then, 
\[
\mu (V_1 \times V_2) = a^{-1} * 
( \tau (W_1 \times V_2)),
\]
where $\tau$ is the transposition 
\[
\square^{n_1 -1} \times \mathbb{G}_m \times 
\square^{n_2 -1} \to \mathbb{G}_m \times \square^{n_1 -1} 
\times \square^{n_2 -1},
\]
and $*$ is the action of $\mathbb{G}_m$ as in ~\eqref{eqn:MCC**}.
This is already closed in $ B_{n-1}$, so we have
\[
\mu_* (V_1 \times V_2) = \mu (V_1 \times V_2) = 
a^{-1} * ( \tau (W_1 \times V_2)).
\]
Furthermore, the modulus condition for $V_2$ implies the modulus condition for 
$a^{-1} * ( \tau (W_1 \times V_2))$. Hence, $\mu_* (V_1 \times V_2)$ has the
modulus condition $M_{ssup}$.

Hence, for the rest of the proof, we assume that neither $V_1$ nor $V_2$ lies 
in the 
loci of $\{ t =a\}$ for some $a \in \mathbb{G}_m$. In particular, the images of
$V_1$ and $V_2$ are open dense subsets of $\G_m$. 
 
Let $V'_1 = {\G}_m \times {\square}^{n_1-1}$ and let $\wt Z =
\ov {{\mu}(V'_1 \times V_2)}$, the closure of ${\mu}(V'_1 \times V_2)$ in
$\G_m \times B_n$. Note that $V'_1$ is just a closed subvariety
and not an admissible cycle. Note further that $Z$ is a closed subvariety of 
$\wt Z$ and moreover $\ov Z$ intersects the divisors $F^1_n$ and $F_{n, 0}$ 
properly in $\P^1 \times {\wh B}_n$ since our $V_1$ and $V_2$ have this 
property. Hence by \propref{prop:restp}, to prove the modulus condition for 
$Z$, it suffices to prove the modulus 
condition for the closed subvariety $\wt Z$. So, from now on, we shall replace 
$V_1$ by $V'_1$ and $ Z$ by $\wt Z$, while we call the new ones as still 
$V_1$ and $Z$, respectively.

We set $B:= {(\P^1)}^{n-1}$, and let ${\ov V}_1$ and ${\ov V}_2$ be the 
closures of $V_1$ and $V_2$ in $\mathbb{G}_m \times (\P^1)^{n_1 -1}$ and 
$\wh{B}_{n_2}$, respectively. 

Let $p : \P^1 \times {(\P^1)}^{n-1} \to {(\P^1)}^{n-1}$ and $q : 
\P^1 \times {(\P^1)}^{n-1} \to \P^1$ be the projections. Let 
$W = {\mu}(V_1 \times V_2)$ and $W' = 
{\ov {\mu}}(\ov V_1 \times \ov V_2)$, where
${\ov {\mu}} = {\mu} \times {\rm Id}_B : \G_m \times \P^1 \times B \to
\P^1 \times B$ is the extension of $\mu$ as defined in ~\eqref{eqn:product}.
Then we get a commutative diagram
\begin{equation}\label{eqn:main}
\xymatrix{
V_1 \times V_2 \ar[d]_{\mu} \ar[r] & \ov V_1  \times \ov V_2 \ar[r] 
\ar[d]^{\ov {\mu}} & \G_m \times \P^1 \times B 
\ar[ddd]^{\ov {\mu}}
\\
W \ar[d] \ar[r] & W' \ar[ddr]^{\phi} \ar[d] & \\
Z\ar[r] \ar[d] & Z' \ar[dr]_{{\phi}'} & \\
\G_m \times {\square}^{n-1} \ar[rr] & & \P^1 \times B,}
\end{equation}
where $Z'$ is the closure of $W'$ in $\P^1 \times {(\P^1)}^{n-1}$, and 
$\phi, \phi'$ are the inclusions.
Note that all the arrows except $\mu$ and ${\ov {\mu}}$ are injective,
and $\mu$ and ${\ov {\mu}}$ are surjective.

Notice that by the given assumptions, the composition under the first 
projection ${\ov V}_2 \subset {\wh B}_{n_2} \to \mathbb{P} ^1$ is surjective
since this map is projective and hence a closed map, while $V_2$ is not in 
the locus of $\{ t =a\}$ for some $a \in \mathbb{G}_m$ so that the map is 
dominant.

First note that $W$ is open in $W'$. This is because $W = W' \cap 
(B_n)$ and $B_{n }$ is open in ${\wh B}_{n}$.
Likewise, $W'$ is open in $Z'$. This is because $W' = Z' \cap
{\rm Image}(\ov{\mu})$, where ${\rm Image}(\ov {\mu})$ is open because 
$\ov {\mu}$ is flat. Hence, the Zariski closure of $W$ in ${\wh B}_{n}$ 
is equal to the closure of $W'$, which is $Z'$ by definition.
Since $Z$ is the Zariski closure of $W$ in $B_{n}$, which is open in 
${\wh B}_{n}$, the Zariski closure ${\ov Z}$ of $Z$ in ${\wh B}_{n}$ 
is consequently $Z'$. In other words, the Zariski closures of 
$W, Z, W'$ in ${\wh B}_{n}$ are all equal to $Z'$.

Let $\ov V'_2 = \tau \left({(\P^1)}^{n_1-1} \times \ov V_2 \right)
\subset \wh{B}_n$ for the transposition
$$\tau : (\P^1)^{n_1 -1} \times \P^1 \times (\P^1)^{n_2 -1} 
\overset{\sim}{\to} \P^1 \times \P^{n_1 -1} \times \P^{n_2 -1}.$$ 
Notice that $W'$ is the orbit $ \mathbb{G}_m \cdot \ov V'_2$ in $\wh{B}_n$, 
where the action of $\mathbb{G}_m$ on ${(\P^1)}^{n-1}$ is by the identity. In 
particular, $Z'$ is the orbit closure of $\ov V'_2$. Let $W_p: = p(W')$. \\
{\bf Claim (1): } \emph{$W_p =p (W')= p(V'_2)$ and it is closed in 
$B = {(\P^1)}^{n-1}$.}

Consider the following diagram.
\[
\xymatrix{
\G_m \times \P^1 \times B \ar[r]^{\ov {\mu}} \ar[d]^{p'} & 
\P^1 \times B \ar[d]^{p}
\\
\G_m \times B \ar[r]_{r} & B}
\]
Note that the lower horizontal arrow and the two vertical arrows are the 
projection
maps. Since $p$ is projective, $p(\ov V'_2)$ is closed in $B$. Next we have,
\[
p(W') = p \circ {\ov {\mu}}(\G_m \times \ov V'_2) = r \circ p' (\G_m \times
\ov V'_2) = r(\G_m \times p(\ov V'_2)) = p(\ov V'_2)
\]
since $p'$ is identity 
on $\G_m$ and $r$ is the projection. This proves Claim (1). \\
{\bf Claim (2): } \emph{There is a non-empty open subset $U \subset W_p$ such 
that 
$\mathbb{G}_m \times U \subset W'$ as an open subset.}

From Claim (1), we have a surjection $\ov V'_2  \to W_p$. Since 
$\ov V'_2$ is irreducible, so is $W_p$. Now recall the following 
well-known generic flatness theorem. For its proof, see 
\cite[Theorem 5.12, p.123]{FGA}:
\begin{thm}\label{thm:GF}
Let $f: X \to Y$ be morphism of noetherian schemes of finite 
type over $k$, where $Y$ is integral. Then there exists a non-empty open 
subset $U' \subset Y$ such that $f^{-1} (U') \to U'$ is flat.
\end{thm}
Using this theorem and the openness of a flat map, we see that the image of 
the open set $\ov V'_2 \cap q^{-1}(\G_m) \cap p^{-1}(U') \to U'$
is open in $U'$ and hence in $W_p$. Notice that since the map 
$\ov V_2  \to \P^1$ is surjective, $\ov V'_2 \cap q^{-1}(\G_m)$ in not empty. 
Let $U$ be this image in $W_p$. 

Now, by the choice of $U$, we see that for each $u \in U$, the fibre 
$p^{-1}(u)$ meets
$\ov V'_2 \cap q^{-1}(\G_m)$ non-trivially. This implies that the orbit of
$\ov V'_2 \cap p^{-1}(U)$ contains at least $\G_m  \times U$ (it might also
contain some points of $\{0, \infty\} \times U$). In particular, we conclude
that $\G_m \times U \subset W'$. Since $\G_m \times U$ is open in 
$p^{-1}(U) =  \P^1  \times U$, it must be open in $W'$, too.  
This proves Claim (2).

Claim (1) implies that $Z' \subset \mathbb{P} ^1 \times W_p$ and
Claim (2) implies that there is a non-empty open subset 
$\mathbb{G}_m \times U \subset Z' \subset \mathbb{P} ^1 \times W'$. 
Since $Z'$ is closed in $\mathbb{P} ^1 \times W_p$ and irreducible, and 
since $\mathbb{G}_m \times U$ is open dense in $\mathbb{P} ^1 \times W_p$
which is also irreducible, we conclude that $Z' = \P^1 \times W_p$.\\
{\bf Claim (3): } \emph{Let $S := Z' \backslash W'$. Then 
$Z' \cap q^*({\{t = 0\}})$ is irreducible and 
${\rm codim}_{Z'} (S \cap q^*{\{t = 0\}}) \ge 2$, where $t$ is the coordinate 
of $\P^1$.}

Since we have just seen that $Z' = \P^1 \times W_p$,
the closed subscheme $Z'_0 = Z' \cap q^*({\{t = 0\}})$ is in fact 
$W_p \times \{0\}$ and hence irreducible as $W_p$ is so.
This also implies that ${\rm dim}(Z'_0) = {\rm dim}(W_p) =
{\rm dim}(Z') - 1$. Now we recall that the map $\ov V_2 \to \P^1$ is 
surjective as a consequence of our assumption. Hence, $\ov V'_2 \to \P^1$
is also surjective. This implies in particular that $W' \to \P^1$ is surjective
too. This in turn shows that $W' \cap q^*({\{t = 0\}}) = 
W' \cap (Z' \cap  q^*({\{t = 0\}})) = W' \cap Z'_0$ is a non-empty open subset
of $Z'_0$. Since we have just shown that $Z'_0$ is irreducible, this implies
that ${\rm dim}_k(S \cap Z'_0) = {\rm dim}_k(Z'_0 \backslash (W' \cap Z'_0))
\le {\dim}_k(Z'_0) -1$. Thus we get
\[
{\rm dim}_k(S \cap Z'_0) \le  {\dim}_k(Z'_0) -1 = {\rm dim}_k(Z') - 2.
\]
This proves Claim (3).\\
{\bf Claim (4): } \emph{For the composition 
\[
\nu_{W'}: {W'}^N \to W' \overset{\phi}{\to} {\wh B}_{n-1},
\]
where the first arrow is the normalization, there exists an index  
$i \in \{ 1, \cdots, n_2 -1 \}$ for which we have on ${W'}^N$
\[
\nu_{W'}^*\left[\{ y_i = 1 \} - (m+1) \{ t = 0 \}\right] \geq 0.
\]}

Consider the following normalization diagram:
\begin{equation}\label{eqn:normal}
\xymatrix@C.8cm{
{\ov V^N_1  \times \ov V^N_2} \ar[d]_{{\ov {\mu}^N}} \ar[r] &
{\ov V_1  \times \ov V_2} \ar[d]^{{\ov {\mu}}} \ar[r] &
{\G_m \times \P^1 \times {(\P^1)}^{n-1}} \ar[d]^{\ov {\mu}} \ar[r]^{pr} &
{\G_m \times \P^1 \times {(\P^1)}^{n_2-1}} \ar[d]^{\ov {\mu}}
\\
{W'}^N \ar[r] & W' \ar[r]^{\phi} & {\P^1 \times {(\P^1)}^{n-1}}
\ar[r]^{pr} & \P^1 \times {(\P^1)}^{n_2-1}.}
\end{equation}
Here the horizontal arrows in the last square are the obvious projection
maps. We also note that ${\ov V^N_1  \times \ov V^N_2}$ 
is the normalization of ${\ov V_1  \times \ov V_2}$ by \cite[Lemma~3.1]{KL}.
We have seen that $\ov {\mu}$ is a surjective map of irreducible 
varieties and hence dominant. This gives the map of the corresponding 
normalizations, which must also be surjective. 
Since the modulus condition $M_{ssup}$ holds for $V_2$, there is an
$1 \le i \le n_2 -1$ such that the Cartier divisor ${\nu}^*_2 [\{y_i =1\} - 
(m+1)\{t = 0\}] \ge 0$ on $\ov V^N_2$.
This implies that in the diagram 
\[
\xymatrix{
{\ov V^N_1  \times \ov V^N_2} \ar[d]^{v} \ar[r]^{{\nu}_{1,2}} &
{\G_m \times \P^1 \times {(\P^1)}^{n-1}} \ar[d]^{s} \\
\ov V^N_2 \ar[r]^{{\nu}_2} & \P^1 \times {(\P^1)}^{n_2-1},}
\]
we have 
\begin{equation}\label{eqn:one}
{\nu}^*_{1,2} \circ s^*[\{y_i =1\} - (m+1)\{t = 0\}] \ge 0
\ \ {\rm on} \ \ \ov V^N_1 \times \ov V^N_2, 
\end{equation}where $v, s$ are the obvious projections, and $\nu_{1,2}$ is 
the composition of the first two arrows of the upper part of the 
Diagram \eqref{eqn:normal}.

Next we observe that as $\G_m$ acts trivially on ${(\P^1)}^{n-1}$, one has
${\ov {\mu}}^*(\{y_i =1\}) = \G_m \times \P^1 \times {(\P^1)}^{n-2} 
\times (\{y_i =1\}) = s^*(\{y_i =1\})$. We also observe that 
${\ov {\mu}}^*(\{t = 0\}) =  \G_m \times (\{t = 0\}) \times {(\P^1)}^{n-1} 
= s^*(\{t = 0\})$. In particular, we have for $1 \le  i \le n_2 -1$,
\[
s^*[\{y_i =1\} - (m+1)\{t = 0\}] =
{\ov {\mu}}^* \circ pr^*[\{y_i =1\} - (m+1)\{t = 0\}]
\]
\[
\hspace*{4.5cm} = {\ov {\mu}}^*[\{y_i =1\} - (m+1)\{t = 0\}].
\]
Combining this with ~\eqref{eqn:one}, we conclude that
${\nu}^*_{1,2} \circ{\ov {\mu}}^*[\{y_i =1\} - (m+1)\{t = 0\}] \ge 0$
and hence  
${\ov {\mu}^N}^* \circ {\nu}^*_{W'}[\{y_i =1\} - (m+1)\{t = 0\}] \ge 0$. 
We now apply \lemref{lem:surjm} to the surjective morphism $\bar{\mu} ^N$ of 
normal integral $k$-varieties to conclude that 
${\nu}^*_{W'}[\{y_i =1\} - (m+1)\{t = 0\}] \ge 0$ on ${W'}^N$.
This proves Claim (4).

Now, we have the final statement of this lengthy proposition:\\
{\bf Claim (5): } \emph{The modulus condition $M_{ssup}$ holds for $Z$.}

Since we have shown that the Zariski closure of $Z$ in ${\wh B}_{n}$ is 
$Z'$, we need to show that the modulus condition holds
on ${Z'}^N$. Consider the following commutative diagram
\begin{equation}\label{eqn:final1}
\xymatrix{
{W'}^N = f^{-1}(W') \ar[r] \ar[d] & {Z'}^N \ar[d]^{f} & f^{-1}(S) \ar[d] 
\ar[l] \\
W' \ar[r] & Z' & S, \ar[l]}
\end{equation}
where $f$ is a normalization, and we recall that $S = Z' \backslash W'$. 
Since $W'$ is as open subset of
$Z'$ as shown above, we have $f^{-1}(W') = {W'}^N$. 
Let ${\nu}_{Z'}: {Z'}^N \overset{f}{\to} Z' \overset{\phi'}{\to} \wh{B}_n$ be 
the composite. We need to show that for some $1 \le i \le n-1$, we have
\begin{equation}\label{eqn:final2}
[{\nu}^*_{Z'}(\{y_i -1\})] \ge (m+1)[{\nu}^*_{Z'}(\{t = 0\})]
\end{equation}
as Weil divisors on ${Z'}^N$.
Since ${Z'}^N$ is an irreducible normal 
variety and the relation ~\eqref{eqn:final2} holds on ${W'}^N$ by 
Claim (4), the same relation will hold on ${Z'}^N$
if and only if no component of ${\nu}^*_{Z'}(\{t = 0\})$ is contained
in $f^{-1}(S)$. However, we have shown in Claim (3) that 
$\dim (S \cap {{\phi}'}^*(\{t =0\})) \le \dim (Z') -2$.
On the other hand, since $f$ is a finite map, 
if a component $D$ of ${\nu}^*_{Z'}(\{t = 0\})$ is 
contained in $f^{-1}(S)$, then $f(D) \subset S \cap {\phi}^*(\{t =0\})$
and we get
\[
\dim({Z'}^N) - 1 = \dim(D) = \dim(f(D))
\le \dim (S \cap {{\phi}'}^*(\{t =0\}))
\]
\[ 
\hspace*{9cm} \le \dim (Z') -2
= \dim({Z'}^N) - 2,
\]
where the second and the last equalities hold by the finiteness and 
surjectivity of $f$. This gives a contradiction. 
This proves the modulus condition $M_{ssup}$ for $Z$,  thus, Claim (5). This 
completes the proof of the proposition.
\end{proof}
\begin{cor}\label{cor:wedge-AD}
Let $X$ and $Y$ be smooth projective varieties. Let 
$V_1 \in  {\TZ}^{q_1}(X, n_1 ; m)$ and $V_2  \in  {\TZ}^{q_2}(Y, n_2 ; m)$
be two irreducible admissible cycles. Then $Z = {\mu}_*(V_1 \times V_2)$
is an admissible additive cycle in ${\TZ}^{q}(X \times Y, n ; m)$.
\end{cor} 
\begin{proof}
This follows immediately from \lemref{lem:proper-int} and
\propref{prop:product-modulus}.
\end{proof}
\subsection{Shuffle products}\label{subsection:shuffle}
For an integer $r \geq 1$, let ${\rm \mathbb{P}erm}_r$ be the group of 
permutations on the set $\{ 1, \cdots, r \}$. For integers 
$s, p_1, p_2, \cdots, p_s \geq 1$, a $(p_1, \cdots, p_s)$-{\sl shuffle} is a 
permutation 
$\sigma \in {\rm \mathbb{P}erm}_{p_1 + \cdots + p_s}$ satisfying the 
following properties:
\[
\left\{\begin{array}{lll}
& \sigma (1) \le \cdots \le \sigma (p_1) \\
& \ \ \ \ \  \ \ \ \ \ \ \vdots  \\
& \sigma (p_1 + \cdots + p_{i-1} + 1 ) \le \cdots \le \sigma 
(p_1 + \cdots + p_{i-1} + p_i), \\
& \ \ \ \ \ \ \ \ \ \ \ \vdots \\
& \sigma (p+1 + \cdots + p_{s-1} + 1 ) \le \cdots \le 
\sigma (p_1 + \cdots + p_{s-1} + p_s)
\end{array}
\right.
\]
Note that since $\sigma$ is an automorphism, each of the above inequalities
is in fact strict unless some $p_i = 1$ in which case one has
$\sigma (p_1 + \cdots + p_{i-1} + 1 ) = \sigma 
(p_1 + \cdots + p_{i-1} + p_i)$.

The set of all $(p_1, \cdots, p_s)$-shuffles is denoted by 
${\rm \mathbb{P}erm}_{(p_1, \cdots, p_s)}$. Here, 
${\rm \mathbb{P}erm}_r = 
{\rm \mathbb{P}erm}_{ \underset{r}{(\underbrace{1, \cdots, 1})}}$, and 
$|{\rm \mathbb{P}erm}_{(p_1, \cdots, p_s)}| = 
\frac{ (p_1 + \cdots + p_s)!}{p_1 ! \cdots p_s !}.$
A permutation $\sigma \in {\rm \mathbb{P}erm}_{n-1}$ acts compatibly on the 
spaces $\square^{n-1}$ and ${(\P^1)}^{n-1}$ via 
$$\sigma \cdot (t_1, \cdots, t_{n-1}) = 
(t_{\sigma^{-1} (1)}, \cdots, t_{\sigma^{-1} (n-1)}).$$
This generalizes to spaces of the form $Y \times \square^{n-1}$ 
and $Y \times {(\P^1)}^{n-1}$ via the trivial action of $\sigma $ on $Y$. 
We obtain induced actions on the groups of algebraic cycles such as 
${\TZ}^q (X, n;m)$, or $z^q (X, n-1)$, etc.

For a permutation $\sigma$, the sign $\sgn (\sigma)$ is $+1$ if $\sigma$ is 
even, and $\sgn (\sigma)$ is $-1$ if odd.

\subsubsection{Permutation identities}
The following basic identities on permutations 
play important roles in proving the associativity of 
the wedge product, and in proving that the differential operator defined in 
Section~\ref{section:DOL}, is a graded derivation for the wedge 
product.
\begin{prop}\label{prop:perm-identity}
Let $r,s,t \geq 1$ be integers. Then in the group ring 
$\mathbb{Z}[{\rm \mathbb{P}erm}_{r+s+t}]$, we have two equations
\begin{eqnarray*}
\sum_{\nu \in {\rm \mathbb{P}erm}_{(r,s,t)}} \sgn(\nu) \nu &=& 
\sum_{\tau \in {\rm \mathbb{P}erm}_{(r+s, t)}} 
\sum_{\sigma \in {\rm \mathbb{P}erm}_{(r,s)}} \sgn (\tau) 
\sgn (\sigma) \tau \cdot (\sigma \times {\rm Id_3}),\\
\sum_{\nu \in {\rm \mathbb{P}erm}_{(r,s,t)}} \sgn(\nu) \nu &=& 
\sum_{\tau' \in {\rm \mathbb{P}erm}_{(r, s+t)}} 
\sum_{\sigma' \in {\rm \mathbb{P}erm}_{(s,t)}} \sgn (\tau') 
\sgn (\sigma ') \tau' \cdot ({\rm Id}_1  \times \sigma '),
\end{eqnarray*}
where ${\rm Id}_1$ is the identity function of the set $\{1, \cdots, r \}$, 
and ${\rm Id}_3$ is the identity function of the set 
$\{ r+s+1, \cdots, r+s+t \}$.
\end{prop}
We need the following two simple lemmas.

\begin{lem}\label{lem:perm*1}
Let $r,s,t \geq 1$ be integers. For 
$\tau \in {\rm \mathbb{P}erm}_{(r+s, t)}$ and 
$\sigma \in {\rm \mathbb{P}erm}_{(r,s)}$, the permutation 
$\tau \cdot (\sigma \times {\rm Id}_3)$ is in 
${\rm \mathbb{P}erm}_{(r,s,t)}$. Furthermore, the set-theoretic function
\[
\phi : {\rm \mathbb{P}erm}_{(r+s, t)} \times {\rm \mathbb{P}erm}_{(r,s)} \to 
\mathbb{P}erm_{(r,s,t)}
\]
\[
(\tau, \sigma) \mapsto \tau \cdot (\sigma \times {\rm Id}_3)
\]
is a bijection.
\end{lem}
\begin{proof} The first part is obvious. For the second part, consider the 
following:
\begin{claim}\label{claim:perm*2}
If $\phi (\tau_1, \sigma_1) = \phi (\tau_2, \sigma_2)$, then 
$\tau_1 = \tau_1$ and $\sigma_1 = \sigma_2$. That is, $\phi$ is injective.
\end{claim}
We are given $\tau_1 \cdot (\sigma_1 \times {\rm Id}_3) = 
\tau_2 (\sigma_2 \times {\rm Id}_3)$. Since 
$\sigma_1 \times {\rm Id}_3, \sigma_2 \times {\rm Id}_3$ do not touch the 
set $S_3:=\{ r+s+1, \cdots, r+s+t \}$, we get 
$\tau_1 |_{S_3} = \tau_2 |_{S_3}$. However, $\tau_i$ are $(r+s,t)$-shuffles 
so that they are strictly increasing on $\{1, \cdots, r+s \}$. This forces 
$\tau_1 |_{\{ 1, \cdots, r+s \}} = \tau_2 |_{\{ 1, \cdots, r+s \}}$. Hence, 
$\tau_1 = \tau_2$. This implies $\sigma_1 = \sigma_2$. Thus the Claim is 
proved.

Notice that for the function $\phi$, the domain and the target have equal 
cardinalities:
\[
|{\rm \mathbb{P}erm}_{(r+s, t)} |\times |{\rm \mathbb{P}erm}_{(r,s)}| = 
\frac{ (r+s+t)!}{ (r+s)! t!} \cdot \frac{ (r+s)!}{r! s!} = 
\frac{ (r+s+t)!}{ r! s! t!} = |{\rm \mathbb{P}erm}_{(r,s,t)}|.
\]
Since $\phi$ is an injective function, this shows that it is automatically 
bijective.
\end{proof}
\begin{lem}\label{lem:perm*3}
Let $r,s,t \geq 1$ be integers. For 
$\tau' \in {\rm \mathbb{P}erm}_{(r, s+t)}$, 
$\sigma' \in {\rm \mathbb{P}erm}_{(s,t)}$, the permutation 
$\tau' \cdot ({\rm Id}_1 \times \sigma')$ is in 
${\rm \mathbb{P}erm}_{(r,s,t)}$. Furthermore, the set-theoretic function
\[
\psi : { \rm \mathbb{P}erm}_{(r, s+t)} \times {\rm \mathbb{P}erm}_{(s,t)} 
\to {\rm \mathbb{P}erm}_{(r,s,t)}
\]
\[
(\tau' , \sigma') \mapsto \tau' \cdot ({\rm Id}_1 \times \sigma')
\]
is a bijection.
\end{lem}

\begin{proof}
Its proof is essentially identical to that of Lemma~\ref{lem:perm*1}.
\end{proof}
\begin{proof}[Proof of Proposition \ref{prop:perm-identity}] This obviously 
follows from Lemmas ~\ref{lem:perm*1} and 
~\ref{lem:perm*3} by observing that 
$\sgn (\tau \cdot (\sigma \times {\rm Id}_3)) = \sgn (\tau) \sgn (\sigma)$, 
and $\sgn (\tau' \cdot ({\rm Id}_1 \times \sigma )) = \sgn (\tau') 
\sgn (\sigma')$. This proves the proposition.
\end{proof}
For the Leibniz rule later, we need the following as well as the above 
results:
\begin{defn}\label{defn:P*}
For permutations $\sigma \in {\rm \mathbb{P}erm}_{n}$ and $\tau \in  
{\rm \mathbb{P}erm}_{(1, n)}$ with 
$\tau(1) = i \in \{ 1, \cdots, n \}$, define the permutation 
$\sigma_{\tau} =\sigma[i] \in {\rm \mathbb{P}erm}_{n+1}$ by sending
\[
j \in \{ 1, \cdots, n+1\} \mapsto \tuborg \sigma (j) & \mbox { if } j < i, \\ 
j & \mbox{ if } j = i , \\ \sigma (j-1) & \mbox{ if } j > i.\sluttuborg 
\]
\end{defn}
\begin{lem}\label{lem:Per*} 
Let $\sigma \in {\rm \mathbb{P}erm}_{(r,s)}$ and $\tau \in  
{\rm \mathbb{P}erm}_{(1, r+s)}$. Then the 
product $\sigma_{\tau} \cdot \tau$ in ${\rm \mathbb{P}erm}_{r+s+1}$ is a 
$(1, r, s)$-shuffle, {\sl i.e.}, 
$\sigma_{\tau} \cdot \tau \in  {\rm \mathbb{P}erm}_{(1, r, s)}$.
Furthermore, the set-theoretic map
\[
\phi_1:  {\rm \mathbb{P}erm}_{(r,s)} \times  {\rm \mathbb{P}erm}_{(1, r+s)} 
\to {\rm \mathbb{P}erm}_{(1, r, s)}
\]
\[
\left( \sigma , \tau \right) \mapsto \sigma_{\tau} \cdot \tau
\]
is a bijection.
\end{lem}
\begin{proof}
The first statement is obvious. For the second statement, the surjectivity 
part is obvious by keeping track of where $1$ is sent. But since both sides 
have the cardinality $\frac{ (r+s)!}{ r! s!} \frac{ (r+s + 1)!} {(r+s)!} = 
\frac{ (r+s+1)!}{r! s!}$, the map $\phi_1$ must be bijective.
\end{proof}
\begin{lem}\label{lem:Per*0}
In the group ring $\mathbb{Z} [{\rm \mathbb{P}erm}_{r+s+1}]$, we have
\begin{eqnarray*}\sum_{\sigma \in  {\rm \mathbb{P}erm}_{(r,s)} } (\sgn 
(\sigma)) 
\left( \sum_{\tau \in  {\rm \mathbb{P}erm}_{(1, r+s)}} (\sgn (\tau)) 
\sigma_{\tau} \cdot 
\tau \right)  = \sum_{\nu \in  {\rm \mathbb{P}erm}_{(1, r, s)}} 
(\sgn (\nu)) \nu.
\end{eqnarray*}
\end{lem}
\begin{proof}
Note that $\sgn (\sigma_{\tau} \cdot \tau) = \sgn (\sigma) \sgn (\tau)$. 
Thus, together with the Lemma~\ref{lem:Per*}, we get the desired result.
\end{proof}
\subsection{Pre-wedge product via shuffles}\label{subsection:prewedge}
Let $X$ and $Y$ be smooth projective varieties. Consider the groups 
${\TZ}^{q_1} (X, n_1; m)$ and ${\TZ}^{q_2} (Y, n_2; m)$. Let 
$n = n_1 + n_2 -1$ and $q = q_1 + q_2 -1$. Consider the group of 
cubical higher Chow cycles, $z^{q+1} (X \times Y \times 
\mathbb{G}_m \times \mathbb{G}_m, n-1)$, i.e. cycles of codimension $q+1$ in 
$X \times Y \times \mathbb{G}_m \times 
\mathbb{G}_m \times \square^{n-1}$ that intersect all faces of 
$\square^{n-1}$ properly, modulo the degenerate cycles.

\begin{defn}\label{defn:shuffle*}
For two irreducible admissible cycles $V_1 \in {\TZ}^{q_1} (X, n_1 ; m)$, 
and $V_2 \in {\TZ}^{q_2} (Y, n_2; m)$,
the \sl{shuffle product} $V_1 \times_{sh} V_2$ is defined as a cycle in 
$z^{q+1} (X \times Y \times \mathbb{G}_m \times \mathbb{G}_m, n-1)$, 
given by the equation
\begin{equation}\label{eqn:shuffle0}
V_1 \times _{sh} V_2 := 
\sum_{\sigma \in {\rm \mathbb{P}erm}_{(n_1 -1, n_2 -1)}} \sgn(\sigma)  
\ \sigma \cdot (V_1 \times V_2).
\end{equation}
\end{defn}
We can extend this definition $\mathbb{Z}$-bilinearly to get a homomorphism
\[
\times_{sh}: {\TZ}^{q_1} (X, n_1;m) \otimes_{\mathbb{Z}} 
{\TZ}^{q_2} (Y, n_2; m) 
\to z^{q+1} (X \times Y \times \mathbb{G}_m \times \mathbb{G}_m, n-1).
\]
The image of this map is the group of $(n_1-1, n_2 -1)$-shuffles.
\begin{lem}\label{lem:shuffle*1}
For the cycles $V_1, V_2$ above and for 
$\sigma \in {\rm \mathbb{P}erm}_{(n_1 -1, n_2 -1)}$, one has that
${\mu}_*\left(\sigma(V_1 \times V_2)\right) \in
{\TZ}^q(X \times Y, n ; m)$.
\end{lem}
\begin{proof} We first observe that $\sigma$ induces an automorphism of
${(\P^1)}^{n-1}$ which preserves ${\square}^{n-1}$ and acts trivially
on $X \times Y \times \G_m \times \G_m$. In particular, the actions of
${\mu}$ and $\sigma$ commute. Hence we have ${\mu}\left(\sigma
(V_1 \times V_2)\right) = \sigma\left({\mu}(V_1 \times V_2)\right)$,
which in turn implies that ${\mu}_*\left(\sigma
(V_1 \times V_2)\right) = \sigma\left({\mu}_*(V_1 \times V_2)\right)$.
The lemma now follows from Corollary~\ref{cor:wedge-AD}.
\end{proof}

This lemma allows one to define the {\sl pre-wedge product}
$V_1 \ov {\wedge} V_2$ of $V_1$ and $V_2$ in 
${\TZ}^{q} (X \times Y, n; m)$ by the equation
\begin{eqnarray*}
V_1 \ov {\wedge} V_2 &:=& \mu_* (V_1 \times_{sh} V_2) =  
\sum_{\sigma \in {\rm \mathbb{P}erm}_{(n_1 -1, n_2 -1)}} \sgn(\sigma)  
\mu_* ( \sigma \cdot (V_1 \times V_2))\\
&=&  \sum_{\sigma \in {\rm \mathbb{P}erm}_{(n_1 -1, n_2 -1)}} \sgn(\sigma)   
\sigma \cdot (\mu_* (V_1 \times V_2))
\end{eqnarray*}
As before, extend it $\mathbb{Z}$-bilinearly to get a homomorphism
\begin{equation}\label{eqn:shuffle*3}
\ov {\wedge} : {\TZ}^{q_1} (X, n_1;m) \otimes_{\mathbb{Z}} 
{\TZ}^{q_2} (Y, n_2; m) \to {\TZ}^{q} (X \times Y, n;m).
\end{equation}
The image of this map is the group of $(n_1 -1, n_2 -1)$-pre-wedges of 
codimension $q$ and modulus $m$. The group of pre-wedges is simply the image 
under $\mu_*$ of the group of shuffles.
\begin{cor}\label{cor:pre-wedge associativity}
For $V_i \in {\TZ}^{q_i}(X_i, n_i; m)$, we have $(V_1 \times_{sh} V_2) 
\times_{sh} V_3 = V_1 \times_{sh} (V_2 \times_{sh} V_3)$ in
${\TZ}^q(X_1 \times X_2 \times X_3, n ;m)$ for appropriate $q$ and $n$.
The same is true for $\ov {\wedge}$.
\end{cor}
\begin{proof}
This follows from the associativity of $\mu$ and 
Proposition~\ref{prop:perm-identity}.
\end{proof}
\begin{lem}\label{lem:shuffle*4}
For two cycles $\xi \in {\TZ}^{q_1} (X, n_1; m)$ and $\eta \in {\TZ}^{q_2} 
(X, n_2 ; m)$, we have equations
\begin{eqnarray}
\partial (\xi \bar{\wedge} \eta) = (\partial \xi) \bar{\wedge} \eta + 
(-1)^{n_1 -1} \xi \bar{\wedge} (\partial \eta), \\
\xi \bar{\wedge} \eta = (-1)^{(n_1 -1)(n_2 -1)} \eta \bar{\wedge} \xi,
\end{eqnarray} 
where $\partial$ is the boundary map in the definition of additive higher
Chow groups.
\end{lem}
\begin{proof}
For both of the equations, it is enough to prove it for $\times_{sh}$, where 
we use the fact that $\partial$ and $\mu_*$ commute for the first. 
But, actually both are just purely combinatorial statements.
\end{proof}
\begin{proof}[Proof of Proposition~\ref{prop:wedge-product}]
The external product structure in ~\eqref{eqn:EWP} follows
directly from the pre-wedge product of cycles in ~\eqref{eqn:shuffle*3}
and from the first identity of Lemma~\ref{lem:shuffle*4}.
If $X = Y$, the anti-commutativity follows directly from the second
identity of Lemma~\ref{lem:shuffle*4}.
\end{proof}

We now prove the following main result of this section and its consequences.
\begin{thm}\label{thm:wedge-product*}
Let $X$ be a smooth projective variety over a field $k$. 
Then there exists an internal wedge product on the additive higher Chow groups 
of $X$.
\begin{equation}\label{eqn:EWP2}
{\wedge}_X : \TH^{q_1} (X, n_1; m) \otimes_{\Z} 
\TH^{q_2} (X, n_2; m) \to \TH^{q} (X, n ; m),
\end{equation}
where $q = q_1 + q_2 - 1$, $n = n_1 + n_2 - 1$, and 
$q_i, n_i, m \geq 1$ for $i = 1,2$, which is associative and
satisfies the equation
\begin{equation}\label{eqn:anticommute*}
\xi {\wedge} _X \eta = 
(-1)^{(n_1-1)(n_2-1)} \eta {\wedge}_X \xi.
\end{equation} 
for all classes $\xi \in \TH^{q_1} (X, n_1;m)$ and $\eta \in \TH^{q_2} 
(X, n_2;m)$.

This wedge product is natural with respect to the pull-back maps of additive 
higher Chow groups and satisfies the projection formula
\begin{equation}\label{eqn:projection-wedge}
f_*\left(a {\wedge}_X f^*(b)\right) = f_*(a) {\wedge}_Y b
\end{equation}
for a morphism $f: X \to Y$ of smooth projective varieties.
\end{thm}
\begin{proof}
Consider the diagonal map $\Delta_X : X \to X \times X$. Since $X$ is smooth 
projective, so is $X \times X$. Hence, by applying Theorem~\ref{thm:funct}
to $\Delta_X$, we get the pull-back map 
\[
\Delta^* _X : \TH^{q}(X \times X, n; m) \to \TH^{q} (X, n; m),
\]
for all integer $q \geq 0$. Composing with the pre-wedge product 
$\ov {\wedge}$, we 
have
\[
\xymatrix{ 
\TH^{q_1} (X, n_1; m) \otimes \TH^{q_2} (X, n_2 ; m) 
\ar[dr] _{\wedge_X} \ar[r] ^{\ \ \ \ \ \ \ \ \ \ \  \ov {\wedge}} &  
\TH^{q} ( X \times X, n; m) \ar[d] ^{\Delta ^* _X} \\
& \TH^{q} (X, n; m),}
\]
where the induced map $\ov {\wedge}$ is well-defined by 
Proposition~\ref{prop:wedge-product}. Now we define 
$\wedge_X := \Delta^* _X \circ \ov {\wedge}$. This gives the desired wedge 
product by the second equation of Lemma~\ref{lem:shuffle*4}. The associativity
follows from Corollary~\ref{cor:pre-wedge associativity}.

We now show the naturality of the wedge product and the projection 
formula. We first observe that if $X \xrightarrow{f} Y  \xrightarrow{g} Z$
are morphisms of smooth projective varieties, then it follows from the
contravariance property of the additive higher Chow groups ({\sl cf.}
Theorem~\ref{thm:funct}) that the naturality of the wedge product with
$f^*$ and $g^*$ implies the same with ${(g \circ f)}^*$.
Similarly, the projection formula for $f$ and $g$ implies that
\[
\begin{array}{lll}
{(g \circ f)}_*[a {\wedge}_X {(g \circ f)}^*(b)] & = &
(g_* \circ f_*)[a {\wedge}_X (f^* \circ g^*(b))] \\
& = & g_*[f_*(a) {\wedge}_Y g^*(b)] \\
& = & (g_* \circ f_*)(a) {\wedge}_Z b \\
& = & {(g \circ f)}_*(a) {\wedge}_Z b
\end{array}   
\]
Hence by factoring a map $f : X \to Y$ as a composite of the closed embedding
$X \inj X \times Y$ and the projection $X \times Y \to Y$, it suffices to
prove the naturality and projection formula when $f$ is one of these two types
of morphisms. 

To prove the naturality, we can use the contravariance property of the
additive higher Chow groups and the construction of the wedge product 
above to reduce to proving the naturality for the pre-wedge product
in Proposition~\ref{prop:wedge-product}. In this case, we only need to show
that for the flat map $f : X \to Y$, the diagram
\begin{equation}\label{eqn:flatC}
\xymatrix@C.8pc{
z^{q}(Y \times Y \times \G_m \times \G_m , n-1)
\ar[r]^{\ \ \ \ \ {\mu}_*} \ar[d]_{f^*} & z^{q}(Y \times Y \times \G_m, n-1) 
\ar[d]^{f^*} \\
z^{q}(X \times X \times \G_m \times \G_m , n-1)
\ar[r]_{\ \ \ \ \ {\mu}_*} & z^{q}(X \times X \times \G_m, n-1)}
\end{equation}
commutes, which is immediate from the definition of $f^*$ and ${\mu}_*$.

If $f$ is a closed embedding, we can use Theorem~\ref{thm:ML} to
replace ${\TZ}^q(Y, \bullet ; m)$ by ${\TZ}^q_{\{X\}}(Y, \bullet ; m)$.
Then, the pull-back map is induced by the Gysin map $f^*$ of
Corollary~\ref{cor:lci}. In this case, we have for the irreducible admissible
$V_i \in {\TZ}^{q_i}_{\{X\}}(Y, n_i;m)$ that $V_1 \times V_2 
\in  {\TZ}^{q+1}_{\{X \times X\}}(Y \times Y, n ;m)$ and 
$f^*(V_1 \times V_2) = f^*(V_1) \times f^*(V_2)$. Thus we only need to show
that the Diagram~\eqref{eqn:flatC} commutes where 
$z^{q}(Y \times Y \times \G_m \times \G_m , n-1)$ (resp.
$z^{q}(Y \times Y \times \G_m, n-1)$) is
replaced by $z^{q}_{\{X \times X \times \G_m \times \G_m\}}
(Y \times Y \times \G_m \times \G_m , n-1)$
(resp. $z^{q}_{\{X \times X \times \G_m \}}
(Y \times Y \times \G_m, n-1)$). But this follows easily once we know that
the diagram
\begin{equation}\label{eqn:flatC0} 
\xymatrix{
X \times X \times \G_m \times \G_m \ar[r] \ar[d]_{\mu} &
Y \times Y \times \G_m \times \G_m \ar[d]^{\mu} \\
X \times X \times \G_m \ar[r] & Y \times Y \times \G_m}
\end{equation}
is in fact a Cartesian square. This proves the naturality with pull-backs.

To prove the projection formula for the closed embedding $X 
\xrightarrow{f} Y$, we can again assume that the cycles under consideration
intersect $X$ or $X \times X$ properly. Then we have for
$V_1 \in z^{q_1}(X \times \G_m, n_1-1)$ and $V_2 \in 
z^{q_2}_{\{X \times \G_m \}}(Y \times \G_m, n_2-1)$, 
\[
\begin{array}{lll}
f_*[{\Delta}^*_X \{{\mu}^{X \times X}_*(V_1 \times [X] \cdot V_2)\}] & = &
{\Delta}^*_Y[{(f \times f)}_* 
\{{\mu}^{X \times X}_*(V_1 \times [X] \cdot V_2)\}] \\
& = & {\Delta}^*_Y[{\mu}^{Y \times Y}_* \{{(f \times f)}_* 
(V_1 \times [X] \cdot V_2)\}] \\
& = & {\Delta}^*_Y[{\mu}^{Y \times Y}_* \{f_*(V_1) \times ([X] \cdot V_2)\}] \\
& = & {\Delta}^*_Y[{\mu}^{Y \times Y}_* \{f_*(V_1) \times V_2)\}],
\end{array}
\] 
where the first equality follows from the fact that the the left square 
in the diagrams
\begin{equation}\label{eqn:flatC1}
\xymatrix{
X \ar[r]^{{\Delta}_X} \ar[d]_{f} & X \times X \ar[d]^{f \times f} &
X \times Y \ar[r]^{{\rm Id}_X \times {\Delta}_Y} \ar[d]_{f} & W \ar[d]^{f'} \\
Y \ar[r]^{{\Delta}_Y} & Y \times Y & Y \ar[r]^{{\Delta}_Y} & Y \times Y}
\end{equation}
is Cartesian and the last equality holds since ${\Delta}^*$ commutes with
${\mu}_*$, as follows from Theorem~\ref{thm:ML}.
This proves the projection formula for the closed embedding.

Finally, we prove the projection formula for the projection map
$f : Z = X \times Y \to Y$. Let $W = X \times Y \times Y$ and let 
${\mu}^W : W \times \G_m \times \G_m \to W \times \G_m$ be the product
map. Let $p : Z \times Z \to W$ be the projection map.
Then for any irreducible admissible cycles $V_1 \in
z^{q_1}(Z \times \G_m, n_1-1)$ and $V_2 \in z^{q_2}(Y \times \G_m, n_2-1)$,
we have
\[
\begin{array}{lll} 
f_*[{\Delta}^*_Z \{{\mu}^{Z \times Z}_*(V_1 \times f^*(V_2))\}] & = &
f_*[{\Delta}^*_Z \{{\mu}^{Z \times Z}_*(V_1 \times (X \times V_2))\}] \\
& = & f_*[{\Delta}^*_Z \{p^*\left({\mu}^{W}_*(V_1 \times V_2)\right)\}] \\
& = &
f_*[{(id_X \times {\Delta}_Y)}^* \left({\mu}^{W}_*(V_1 \times V_2)\right)] 
\\
& = & {\Delta}^Y_*[f'_*\left({\mu}^{W}_*(V_1 \times V_2)\right)] 
\ \ \ \ \ \ \ \ \ \ (*) \\
& = & {\Delta}^Y_*[{\mu}^{Y \times Y}_*\{{f'}_*(V_1 \times V_2)\}] \\
& = & {\Delta}^Y_*[{\mu}^{Y \times Y}_*\{{f}_*(V_2) \times V_1)\}],
\end{array}
\]
where the equality $(*)$ follows from the right Cartesian square in
~\eqref{eqn:flatC1}.
This proves the projection formula for the projection map. This completes the 
proof of the theorem.
\end{proof}
For a smooth projective variety $X$ over $k$,
let $\TH_n (X) = \bigoplus_q \TH^q (X, n+1; m)$ and let $\TH(X)  = 
\bigoplus_{n \geq 0} \TH_n (X)$. Let $A(X) = \bigoplus_q CH^q(X)$ be the
ordinary Chow ring of $X$. As an immediate consequence of 
Theorem~\ref{thm:wedge-product*}, we have :
\begin{cor}\label{cor:wedge-product**}
For $X$ as above, there is a wedge product structure on ${\TH}(X)$ 
\[
\TH(X) {\otimes}_{A(X)} \TH(X) \xrightarrow{\wedge} \TH(X),
\]
that makes $\TH(X)$ a graded-commutative algebra.
\end{cor}
\begin{proof}
This follows immediately from Theorem~\ref{thm:wedge-product*} once we know 
that the pre-wedge product in ~\eqref{eqn:shuffle*3} is bilinear over the ring
$A(X)$, where the $A(X)$-module structure on $\TH(X)$ is given by 
Theorem~\ref{thm:basic}. But this can be easily checked from the 
construction of the shuffle product in ~\eqref{eqn:shuffle*3}.
\end{proof}
As another consequence of Theorem~\ref{thm:wedge-product*}, we get the
following result which was widely expected in view of the belief
that the additive higher Chow groups compute the relative $K$-theory of
the infinitesimal thickenings of smooth varieties.
\begin{cor}\label{cor:module}
Let $X$ be a smooth projective variety over a field $k$ such that
${\rm char}(k) \neq 2$. Then for any $q, n, m \ge 1$, the group 
$\TH^q(X, n;m)$ is a ${\W}_m(k)$-module, where
${\W}_m(k)$ is the ring of generalized Witt-vectors of length $m$ over
$k$. In particular, $\TH^q(X, n;m)$ is naturally a $k$-vector space if
${\rm char}(k) = 0$.
\end{cor}
\begin{proof}
The follows immediately from Theorem~\ref{thm:wedge-product*} by
considering the composite map
\[
\xymatrix{
\TH^{1} (k, 1; m) \otimes_{\Z} 
\TH^{q} (X, n; m) \ar[r]^{p^* \otimes {\rm Id}} \ar[dr] &
\TH^{1} (X, 1; m) \otimes_{\Z} \TH^{q} (X, n; m) \ar[d]^{\wedge} \\
& \TH^{q} (X, n; m),}
\]
where $p : X \to {\rm Spec}(k)$ is the structure map, and using the
isomorphism ${\W}_m(k) \xrightarrow{\cong} \TH^{1} (k, 1; m)$
({\sl cf.} \cite{R}). That this gives a module structure, also follows from
Theorem~\ref{thm:wedge-product*}.
In characteristic zero, ${\W}_m(k)$ is itself a $k$-module.
\end{proof} 
\section{Differential operator on additive higher Chow groups}
\label{section:dga}
We have shown in Section~\ref{section:WD} that the additive higher Chow
groups of a smooth projective variety have a structure of
naturally defined commutative graded algebra. Our main goal in the remaining
part of this paper is to show that these are also 
equipped with differential operators, one of which turns this algebra into a 
{\sl differential graded algebra}. We construct one of these differential
operators in this section.

Let $X$ be a smooth projective variety. Let 
$\G^{\times}_m$ denote the variety $\G_m \backslash \{1\}$. 
We have natural inclusions of open sets $\G^{\times}_m \inj 
\square \inj \P^1$. For $n, m \ge 1$,
define the map 
\begin{equation}\label{eqn:diagonal} 
\phi_n : X \times \G^{\times}_m \times \square^{n-1} \to 
X \times \G_m \times \square^{n}
\end{equation}
\[
(x, t, y_1, \cdots, y_{n-1}) \mapsto (x, t, t^{-1}, y_1, \cdots , y_{n-1}).
\]
Note that $\phi_n$ is not a closed immersion. Rather, it is the composite of
the closed immersion $X \times \G^{\times}_m \times \square^{n-1} \inj
X \times \G^{\times}_m \times \G^{\times}_m \times \square^{n-1}$
followed by the open immersion 
$X \times \G^{\times}_m \times \G^{\times}_m \times \square^{n-1}
\inj X \times \G_m \times \square^{n}$. 
For once and all, we fix the coordinates $(t, y_1, \cdots, y_n)$ of
$\G_m \times \square^{n} \subset \P^1 \times {(\P^1)}^n$.
For any irreducible cycle $Z \subset  X \times \G_m \times 
\square^{n-1}$, let $Z^{\times}$ denote its restriction to the open set
$X \times \G^{\times}_m \times \square^{n-1}$.
Our first observation is the following.
\begin{lem}\label{lem:closed}
For any irreducible admissible cycle $Z \in \un{\TZ}^q(X, n;m)$,
$\phi_n(Z^{\times})$ is closed in $X \times \G_m \times \square^{n}$.
\end{lem}
\begin{proof} We look at the Zariski closure $W:=\ov {\phi_n(Z^{\times})}$ of 
$\phi_n (Z^{\times})$ in the bigger space $X \times \G_m \times {(\P^1)}^n$, 
and see what happens.
The image of $Z^{\times}$ is clearly closed in
$X \times \G^{\times}_m \times \G^{\times}_m \times \square^{n-1}$.

Hence $W \backslash \phi_n(Z^{\times})$ must be contained in 
$\{ t = 1 \} \cup \{ y_1 = 1 \}$. By the definition of $\phi_n$, if a point 
in $W \backslash \phi_n (Z^{\times})$ intersects $\{ t= 1 \}$, then it must 
also intersect $\{y_1 = 1\}$ in
$X \times \G_m \times {(\P^1)}^n$. Hence $W \backslash \phi_n (Z^{\times})$ 
is in fact contained in $\{ t = 1 \} \cap \{ y_1 = 1 \}$ in 
$X \times \mathbb{G}_m \times (\P^1)^n$. In particular, 
$W \backslash \phi_n(Z^{\times})$ cannot intersect with 
$ X \times \G_m \times \square^{n}$. Hence if $W'$ is the Zariski closure of 
$\phi_n (Z^{\times})$ in $X \times \mathbb{G}_m \times \square^n$, then 
$W' \backslash \phi_n (Z^{\times})$, which is a subset of 
$W \backslash \phi_n (Z^{\times})$, does not intersect 
$X \times \mathbb{G}_m \times \square^n$, either. This shows that 
$W' = \phi_n (Z^{\times})$. Hence, $\phi_n(Z^{\times})$ is closed.
\end{proof}
We shall often write the morphisms such as $\phi_n$ in the sequel simply as
rational maps on the ambient space and also write $\phi_n(Z^{\times})$ as 
$\phi_n(Z)$.
\begin{lem}\label{lem:DRmod}
For $Z$ as in Lemma~\ref{lem:closed}, the closed subvariety 
$V := \phi_n(Z^{\times})$
satisfies the modulus condition.
\end{lem}
\begin{proof} Consider the following commutative diagram.
\begin{equation}\label{eqn:DRmod0}
\xymatrix@C.8cm{
{\ov Z}^N \ar[r]^{\wt {\phi_n}} \ar[d] \ar@/_1pc/[dd]_{f} &
{\ov V}^N  \ar[d] \ar@/^1pc/[dd]^{g} \\
{\ov Z} \ar[r] \ar[d] & {\ov V} \ar[d] \\
{X \times \P^1 \times {(\P^1)}^{n-1}} \ar[r]_{\ov {\phi_n}} &
{X \times \P^1 \times {(\P^1)}^{n}}}
\end{equation}
Here ${\ov {\phi_n}}(x, t, y_1, \cdots, y_{n-1}) = 
(x, t, t^{-1}, y_1, \cdots, y_{n-1})$ is the natural extension of $\phi_n$.
Note also that $\wt {\phi_n}$ is induced by the dominant map $Z^{\times} \to  
V$, which in turn gives the map ${\ov Z} \to {\ov V}$ as $\ov {\phi_n}$ is
closed. In particular, $\wt {\phi_n}$ is projective and surjective.

Next, it is easy to check from the description of $\ov {\phi_n}$ that 
${\ov {\phi_n}}^*(F_{n+1, 0}) = F_{n, 0}$ and ${\ov {\phi_n}}^*(F^1_{n+1, i})
= F^1_{n, i-1}$ for $i \ge 2$. In particular, the modulus condition 
$M_{ssup}$ for $Z$ implies that 
\[
{\wt {\phi_n}}^* \circ g^*[F^1_{n+1, i+1} - (m+1)F_{n+1, 0}]
= f^* \circ {\ov {\phi_n}}^*[F^1_{n+1, i+1} - (m+1)F_{n+1, 0}]
\]
\[
\hspace*{5cm} =
f^*[F^1_{n, i} - (m+1)F_{n, 0}] \ge 0
\]
for some $1 \le i \le n-1$. We conclude from the surjectivity of
${\wt {\phi_n}}$ and from an easy variant of Proposition~\ref{prop:P0} that 
$g^*[F^1_{n+1, i+1} - (m+1)F_{n+1, 0}] \ge 0$ for some $1 \le i \le n-1$, 
which is the $M_{ssup}$ condition for $V$.

If $Z$ satisfies the modulus condition $M_{sum}$, then following the same
argument as above, we get
\begin{eqnarray*}
& & {\wt {\phi_n}}^* \circ g^*[F^1_{n+1} - (m+1)F_{n+1, 0}] \\
& \ge & \left(\sum_{i=1} ^n
{\wt {\phi_n}}^* \circ g^*[F^1_{n+1, i}]\right) - 
{\wt {\phi_n}}^* \circ g^*[(m+1)F_{n+1, 0}] \\
& = & \left( \sum_{i=2} ^n
f^* \circ {\ov {\phi_n}}^*[F^1_{n+1, i}]\right) - 
f^* \circ {\ov {\phi_n}}^*[(m+1)F_{n+1, 0}] \\
& = & f^*[F^1_n - (m+1)F_{n,0}] \geq 0.
\end{eqnarray*}
We again conclude from the surjectivity of
${\wt {\phi_n}}$ and from an easy variant of Proposition~\ref{prop:P0} that 
$g^*[F^1_{n+1} - (m+1)F_{n+1, 0}] \ge 0$. This shows the modulus condition 
$M_{sum}$ for $V$.
\end{proof}
\begin{prop}\label{prop:DR}
For any irreducible admissible cycle $Z \in \un{\TZ}^q(X, n;m)$,
$\phi_n(Z^{\times})$ defines an admissible irreducible cycle in 
$ \un{\TZ}^{q+1}(X, n+1;m)$, that we denote by $\delta(Z)$. Furthermore, 
$\delta$ and
$\partial$ satisfy the relation $\partial \delta + \delta \partial = 0$.
\end{prop}
\begin{proof} We first prove the following.\\
{\bf Claim : } \emph{$(1) \ \ {\partial}^{\epsilon}_{n+1, 1} \circ 
{\phi}_n = 0$ for $\epsilon = 0, \infty$.
\\
$(2) \ \ {\partial}^{\epsilon}_{n+1, i} \circ {\phi}_n
= {\phi}_{n-1} \circ {\partial}^{\epsilon}_{n, i-1}$ for $i \ge 2$ and
$\epsilon = 0, \infty$.}

Here $\partial_{n+1, i} ^{\epsilon}$ is the $i$-th face 
$\partial_i ^{\epsilon}$ on $B_{n+1}$. It is easy to see from the definition 
of $\phi_n$ that ${\phi}_n(Z)$ intersects
$\{y_1 = 1\}$ if and only if it intersects $\{t = 1\}$. This clearly implies 
$(1)$.

For $(2)$, we can again observe from the definition of $\phi_n$ that it just
shifts the coordinates of $\square^{n-1}$ by one. In particular, for
$i \ge 2$ and $\epsilon = 0, \infty$, the diagram
\begin{equation}\label{eqn:cart}
\xymatrix{
\G^{\times}_m \times \square^{n-2} \ar[r]^{\iota^{\epsilon}_{i-1}} 
\ar[d]_{\phi_{n-1}} &
\G^{\times}_m \times \square^{n-1} \ar[d]^{\phi_n} \\
\G_m \times \square^{n-1} \ar[r]_{\iota^{\epsilon}_i} & 
\G_m \times \square^n}
\end{equation} 
is Cartesian.
Hence we have for $i \ge 2$,
\[
{\partial}^{\epsilon}_{n+1, i} \circ {\phi}_n (Z^{\times}) 
= {\partial}^{\epsilon}_{n+1, i} \left(\phi_n^{-1}(F^{\epsilon}_{n+1,i}) 
\cdot Z^{\times}\right) = \phi_{n-1}\left(F^{\epsilon}_{n,i-1}
\cdot Z^{\times}\right) 
\]
\[
\hspace*{3cm} = \phi_{n-1}
\left({(F^{\epsilon}_{n,i-1} \cdot Z)}^{\times}\right) 
= \phi_{n-1} 
\left({\left(\partial^{\epsilon}_{n, i-1}(Z)\right)}^{\times}\right).
\]
This proves the Claim.

Using the Claim and the proper intersection property of $Z$, 
we see immediately that $\delta(Z)$ has proper intersections with faces.
We also conclude from Lemma~\ref{lem:DRmod} that $\delta(Z)$ satisfies
the modulus condition. Since $\delta$ does not change the dimension,
we have shown that $\delta(Z)$ is in ${\TZ}^{q+1}(X, n+1;m)$.

Finally, if we denote the operator $\delta$ at the level $n$ of 
${\TZ}^{q}(X, \bullet ;m)$ by $\delta_n$, then we have
\begin{eqnarray*}
\partial \circ {\delta}_n (Z) & = & 
\sum_{i=1} ^n {(-1)}^i [{\partial}^{\infty}_{n+1, i}
\circ {\delta}_n(Z) - {\partial}^{0}_{n+1, i} \circ {\delta}_n(Z)] \\
& = & \sum_{i=2} ^n {(-1)}^i 
[{\delta}_{n-1} \circ {\partial}^{\infty}_{n, i-1}(Z) -
{\delta}_{n-1} \circ {\partial}^{0}_{n, i-1}(Z)] \\
& = & -\left( \sum_{i=1} ^{n-1} {(-1)}^i 
[{\delta}_{n-1} \circ {\partial}^{\infty}_{n, i}(Z) -
{\delta}_{n-1} \circ {\partial}^{0}_{n, i}(Z)]\right) \\
& = & - {\delta}_{n-1}\left( \sum_{i=1} ^{n-1} {(-1)}^i 
[{\partial}^{\infty}_{n, i}(Z) - {\partial}^{0}_{n, i}(Z)]\right) \\ 
& = & - {\delta}_{n-1} \circ \partial (Z),
\end{eqnarray*}
where the second equality follows from the above Claim. This 
proves the proposition.
\end{proof}
\begin{cor}\label{cor:DR*}
For every $q \ge 1$, $\delta$ defines a chain map
\[
\delta : {\TZ}^{q}(X, \bullet ;m) \to {\TZ}^{q+1}(X, \bullet ;m)[1].
\]
\end{cor}
\begin{proof}
For an irreducible admissible cycle $Z \in \un{\TZ}^q(X, n;m)$, we define 
$\delta(Z)$ as in Proposition~\ref{prop:DR} and then extend linearly to 
$\un{\TZ}^q(X, n;m)$. It is clear from ~\eqref{eqn:diagonal} that $\delta$
preserves the degenerate cycles. Now the corollary follows from 
Proposition~\ref{prop:DR}.
\end{proof} 
\subsection{Computation of $\delta^2$}\label{subsection:Szero}
Our next goal is show that $\delta^2$ is zero to make it into a differential
operator on the additive higher Chow groups.
We achieve this by explicitly constructing certain 
admissible cycles which bound $\delta^2(Z)$ for any irreducible admissible
cycle $Z$.
We first define certain 2-cycles in $z^2(\G_m, 3)$ which are all two
dimensional analogues of variants of B. Totaro's 1-cycles in \cite{T}.  
For any general point $t \in \G_m$, the parameter $u$ will always denote 
$t^{-1}$ in this part of the section.

For $1 \le j \le 4$, let $\Gamma^1_j \subset \G_m \times \square^3$ be the 
the 2-cycles defined by the rational maps
$\psi_j^1 : \G_m \times \square \to \G_m \times \square^3$ given as
follows:
\begin{equation}\label{eqn:Tcycles1}
\begin{array}{lll}
\psi^1_1(t, x) & = & \left(t, u, x, \frac{(1-u)x - (1-u)/(1-t)}
{x - (1-u)/(1-t)}\right) \\
\psi^1_2(t, x) & = & \left(t, u, x, \frac{(1-u)x - 1}{x-1}\right) \\
\psi^1_3(t, x) & = & \left(t, x, \frac{ux - 1}{x-1}, 1-u\right) \\
\psi^1_4(t, x) & = & \left(t, x, 1-x, \frac{u - x}{1-x}\right) 
\end{array} 
\end{equation}
for $t \in \G_m \backslash \{1\}$ and $x \in \square \backslash \{0\}$.
We similarly define the 2-cycles $\Gamma^2_j \subset \G_m \times \square^3$
given by the rational maps
\begin{equation}\label{eqn:Tcycles2}
\begin{array}{lll}
\psi^2_1(t, x) & = & \left(t, u^2, x, \frac{(1-u^2)x - (1-u^2)/(1-t^2)}
{x - (1-u^2)/(1-t^2)}\right) \\
\psi^2_2(t, x) & = & \left(t, u^2, x, \frac{(1-u^2)x - 1}{x-1}\right) \\
\psi^2_3(t, x) & = & \left(t, x, \frac{u^2x - 1}{x-1}, 1-u^2\right) \\
\psi^2_4(t, x) & = & \left(t, x, 1-x, \frac{u^2 - x}{1-x}\right) \\
\psi^2_5(t, x) & = & \left(t, x, \frac{ux - u^2}{x-u^2}, -u^2\right) \\
\psi^2_6(t, x) & = & \left(t, u, x, \frac{ux + u^2}{x+u^2}\right)
\end{array} 
\end{equation}
for $t \in \G_m \backslash \{1, -1\}$ and $x \in \square \backslash \{0\}$.

For an irreducible 1-cycle $\alpha \subset \G_m \times \square^2$
which is admissible and defined by a rational map 
\begin{equation}\label{eqn:para}
\phi : \G_m \to \G_m \times \square^2
\end{equation}
\[ 
\phi(t) = (\phi(t)(0), \phi(t)(1), \phi(t)(2)),
\]
we shall often write $\alpha$ by the parametrization
$(\phi(t)(0), \phi(t)(1), \phi(t)(2))$ to simplify the notations.

It is now easy to check from the definitions that all $\Gamma^l_j$
are closed in $\G_m \times \square^3$ and they in fact define admissible
cycles in $z^2(\G_m, 3)$ ({\sl cf.} Lemma~\ref{lem:closed}).
Moreover, one can also check in a straightforward
way (or using the computations in \cite[Section~2]{T}) that these cycles
have the following boundaries:
\begin{equation}\label{eqn:Tcycles3}
\begin{array}{lll}
\partial \Gamma^l_1 & = & 
\left(t, u^l,  \frac{1-u^l}{1-t^l}\right) - \left(t, u^l, 1-u^l\right) -
\left(t, u, \frac{1}{1-t^l}\right) \\
\partial \Gamma^l_2 & = & 
\left(t, u^l, 1-t^l\right) - \left(t, u^l, \frac{1}{1-t^l}\right) \\
\partial \Gamma^l_3 & = & 
\left(t, u^l, 1-t^l\right) - \left(t, t^l, 1-t^l\right) \\
\partial \Gamma^l_4 & = & 
\left(t, u^l, 1-u^l\right) \\
\partial \Gamma^2_5 & = & 
2 \left(t, u, -u^2\right) - \left(t, u^2, -u^2\right) \\
\partial \Gamma^2_6 & = & 
\left(t, u, -u^2\right) - \left(t, u, u\right) - \left(t, u, -u\right)
\end{array}
\end{equation}  

Since $u= t^{-1}$, note that
\begin{equation}\frac{1 - u^l}{1-t^l} = \frac{ 1- t^{-l}}{1- t^l} = 
\frac{t^l -1}{1-t^l} \cdot t^{-l} = - t^{-l} = - u^l.\end{equation} Hence, we 
have 
\[
\left(t, u^l, -u^l\right) = \left(t, u^l, (1-u^l)/(1-t^l)\right).
\]
Using \eqref{eqn:Tcycles3} together with this, we see at once that for 
$l= 1, 2$, 
\begin{equation}\label{eqn:Tcycles4}
\begin{array}{lll}
\left(t, u^l, - u^l\right) & = & \partial \Gamma^l_1 - \partial 
\Gamma^l_2 - \partial \Gamma^l_3
+ 2 \Gamma^l_4.
\end{array}
\end{equation}
We also obtain from ~\eqref{eqn:Tcycles3} that 
\begin{equation}\label{eqn:Tcycles5}
\begin{array}{lll}
\left(t, u^2, - u^2\right) & = & 2(t, u, u) -  \partial \Gamma^2_5 + 
2 \Gamma^2_6 + 2 (t, u, -u).
\end{array}
\end{equation} 
Combining ~\eqref{eqn:Tcycles4} and  ~\eqref{eqn:Tcycles5} together, we obtain
that as an element of $z^2(\G_m, 2)$, 
\begin{equation}\label{eqn:Tcycles6}
\begin{array}{lll}
 2(t, u, u) & = & \partial \Gamma^2_1 - \partial \Gamma^2_2 - \partial
\Gamma^2_3 + 2 \partial \Gamma^2_4 + \partial \Gamma^2_5 
- 2 \partial \Gamma^2_6 \\
& & - 2 (\partial \Gamma^1_1 - \partial \Gamma^1_2 - \partial \Gamma^1_3
+ 2 \partial \Gamma^1_4) \\
& = :& \partial \Gamma.
\end{array}
\end{equation} 
Let $X$ be a smooth projective variety. For any admissible additive cycle
$Z \in  \un{\TZ}^{q}(X, n;m)$, we can naturally consider it as a higher
Chow cycle in $z^q(X \times \G_m, n-1)$. For any $\Gamma^l_j \in
z^2(\G_m, 3)$, we get the exterior product $Z \times \Gamma^l_j
\in z^{q+2}(X \times \G_m \times \G_m, n+2)$. Moreover, it also follows
from the definition of these cycles that $Z \times \Gamma^l_j$ intersects
$X \times \G_m$ properly under the diagonal embedding $\Delta_{\G_m}:
X \times \G_m \to X \times \G_m \times \G_m$ given by
$(x, t) \mapsto (x, t, t)$.
Since $X$ is smooth, we get the pull-back cycle 
$Z \star \Gamma^l_j = {\Delta}^*_{\G_m}(Z \times \Gamma^l_j) \in
z^{q+2}(X \times \G_m, n+2)$.
\begin{lem}\label{lem:boundary}
The cycle $Z \star \Gamma^l_j$ lies in $\un{\TZ}^{q+2}(X, n+3 ;m)$ under the
natural inclusion $\un{\TZ}^{q+2}(X, n+3 ;m) \inj 
z^{q+2}(X \times \G_m, n+2)$.
\end{lem}
\begin{proof} We only need to show that $Z\star \Gamma^l_j$ satisfies the
modulus condition. For this, we observe that $Z\star \Gamma^l_j$ is the 
closure of image of $Z' = \square \times Z$ under the rational map
\begin{equation}\label{eqn:boundary*}
\Psi^l_j : X \times \G_m \times \square^n \to  X \times \G_m \times 
\square^{n+2}
\end{equation}
\[
\Psi^l_j(x, t, y, y_1, \cdots, y_{n-1}) =
(x, \psi^l_j(0), \psi^l_j(1), \psi^l_j(2), \psi^l_j(3), y_1, \cdots, y_{n-1})
\]
in the notation of ~\eqref{eqn:para}. We now follow the proof of 
Lemma~\ref{lem:DRmod} to prove the modulus condition for $Z'$. Let $V' =
Z \star \Gamma^l_j = \ov {\Psi^l_j(Z')}$.  We
consider the following commutative diagram.
\begin{equation}\label{eqn:boundary*0} 
\xymatrix@C.8cm{
{\ov Z}^N \times \P^1 \ar[r]^{\wt {\Psi^l_j}} \ar[d]_{f'} &
{\ov V'}^N \ar[d]^{g'} \\
{X \times \P^1 \times \P^1 \times {(\P^1)}^{n-1}} \ar[r]_{\ov {\Psi^l_j}} &
{X \times \P^1 \times {(\P^1)}^{n+2}}}
\end{equation}
Here $f ' = f \times {\rm Id}$, where $f : {\ov Z}^N \to X \times \P^1
\times {(\P^1)}^{n-1}$ is the normalization map for $\ov Z$ as in the
Diagram ~\eqref{eqn:DRmod0}.
Note that the map ${\ov {\Psi^l_j}}$ is defined since all the rational 
maps 
$\psi^l_j$ naturally extend to morphisms $\psi^l_j: \P^1 \times \P^1 \to
\P^1 \times {(\P^1)}^3$.

If $Z$ satisfies the modulus condition $M_{ssup}$, then there is some 
$1 \le i \le n-1$ such that 
$[f^*(y_i = 1) - (m+1)f^*(t = 0)] \ge 0$ on ${\ov Z}^N$. Since
$f'$ is identity on $\P^1$, this implies that 
$[{f'}^*(y_i = 1) - (m+1){f'}^*(t = 0)] \ge 0$ on ${\ov Z}^N$
for some $2 \le i \le n$. Since ${\ov {\Psi^l_j}}$ is identity on the
last $(n-1)$ copies of $\P^1$, we conclude that
${f'}^* \circ {\ov {\Psi^l_j}}^*[F^1_{n+3, i} - (m+1)F_{n+3, 0}] \ge 0$
for some $4 \le i \le n+2$, which in turn gives 
${\wt {\Psi^l_j}}^* \circ {g'}^*[F^1_{n+3, i} - (m+1)F_{n+3, 0}] \ge 0$
on ${\ov Z}^N \times \P^1$. Since ${\wt {\Psi^l_j}}$ is projective and 
surjective, we conclude from Proposition~\ref{prop:P0} that
${g'}^*[F^1_{n+3, i} - (m+1)F_{n+3, 0}] \ge 0$ for some
$4 \le i \le n+2$. This prove the modulus condition $M_{ssup}$ for 
$Z\star \Gamma^l_j$. 

If $Z$ satisfies the modulus condition $M_{sum}$,
then we use the same argument as above plus the proof of the $M_{sum}$
part of Lemma~\ref{lem:DRmod} to complete the proof of the lemma.
\end{proof}
Our main interest about $\delta^2$ is the following.
\begin{prop}\label{prop:Dsquare}
Assume that ${\rm char}(k) \neq 2$ and let 
$\alpha \in \un{\TZ}^{q}(X, n ;m)$ be a cycle such that 
$\partial (\alpha) = 0$. Then $\delta^2(\alpha) = 0$ as a homology class
in $\TH^{q+2}(X, n+2 ; m)$.
In particular, $\delta$ descends to a natural map of additive higher Chow 
groups $\delta : \TH^q(X, \bullet ; m) \to  \TH^{q+1}(X, \bullet ; m)[1]$ such
that $\delta^2 = 0$.
\end{prop}
\begin{proof} The last part of the proposition follows from 
Corollary~\ref{cor:DR*} once we prove the first part. 
Since $\delta^2 $ is equal to a boundary in $\un{\TZ}^{q}(X, \bullet ;m)$ if 
and only if it is a boundary of an admissible additive cycle in 
$z^{q}(X \times \G_m, \bullet-1)$, we can work with the
latter complexes. We begin with the following.\\
{\bf Claim : } \emph{For any $\alpha \in \un{\TZ}^{q}(X, n ;m)$
and $\Gamma$ as in ~\eqref{eqn:Tcycles6}, one has 
\[
\partial (\alpha \times \Gamma) = \alpha \times \partial \Gamma
- \partial \alpha \times \Gamma
\]
in $z^{q+2}(X \times \G_m \times \G_m, n+2)$.}

This is an elementary computation. Since $\Gamma$ is a $\mathbb{Z}$-linear 
combination
of $\Gamma^l_j$'s, it suffices to prove the claim for each $\Gamma^l_j$.
We can further assume that $\alpha$ is
represented by an irreducible cycle $Z$. Then we note that for 
$\epsilon \in \{ 0 , \infty \}$, $l \in \{ 1, 2 \}$,
\[
\partial^{\epsilon}_i(Z \times \Gamma^l_j) = \left\{\begin{array}{ll}
Z \times \partial^{\epsilon}_i \Gamma^l_j & \mbox{if $1 \le i \le 3$} \\
\partial^{\epsilon}_{i-3}(Z) \times \Gamma^l_j & \mbox{if $4 \le i \le n+2$}.
\end{array}
\right.
\]
This in turn gives
\begin{eqnarray*}
\partial (Z \times \Gamma^l_j) & = & 
\sum_{i=1} ^{n+2} {(-1)}^i
[\partial^{\infty}_i(Z \times \Gamma^l_j) - 
\partial^{0}_i(Z \times \Gamma^l_j)] \\
& = & Z \times \left\{\sum_{i=1} ^3  {(-1)}^i
[\partial^{\infty}_i(\Gamma^l_j) - \partial^{0}_i(\Gamma^l_j)]\right\}
 \\
& & 
+ \left\{\stackrel{n+2}{\underset {i = 4}{\sum}} {(-1)}^i
[\partial^{\infty}_{i-3}(Z) - \partial^{0}_{i-3}(Z)]\right\} \times 
\Gamma^l_j \\
& = & Z \times \partial (\Gamma^l_j) + 
\left\{\stackrel{n+2}{\underset {i = 4}{\sum}} {(-1)}^i
[\partial^{\infty}_{i-3}(Z) - \partial^{0}_{i-3}(Z)]\right\} \times 
\Gamma^l_j \\
& = & Z \times \partial (\Gamma^l_j) +
\left\{\stackrel{n-1}{\underset {i = 1}{\sum}} {(-1)}^{i+3}
[\partial^{\infty}_i(Z) - \partial^{0}_i(Z)]\right\} \times \Gamma^l_j \\
& = &  Z \times \partial (\Gamma^l_j) - \partial Z \times \Gamma^l_j.
\end{eqnarray*}
This proves the claim.

Next, we see from the definition of $\phi_n$ in ~\eqref{eqn:diagonal} that
for any irreducible admissible cycle $Z \in \un{\TZ}^{q}(X, n ;m)$,
$\delta^2(Z)$ is just the image of $Z$ under the rational map
\begin{equation}\label{eqn:Dsquare*}
X \times \G_m \times \square^{n-1} \to X \times \G_m \times \square^{n+1}
\end{equation}
\[
(x, t, y_1, \cdots , y_{n-1}) \mapsto 
(x, t, t^{-1}, t^{-1}, y_1, \cdots, y_{n-1})
\]
\[
\hspace*{5cm} = (x, t, u, u, y_1, \cdots, y_{n-1}).
\]
In particular, we see from ~\eqref{eqn:boundary*} that $\delta^2(Z)$ is
the cycle $Z \star (t, u, u) = \Delta^*_{\G_m}\left(Z \times (t, u, u)\right)$.
Hence for any $\alpha \in \un{\TZ}^{q}(X, n ;m)$, we have
$\delta^2(\alpha) = \alpha \star (t, u, u)$ as an element of
$z^{q+2}(X \times \G_m, n+1)$. Since the diagram
\[
\xymatrix{
\G_m \times \square^{n+1} \ar[r]^{{\iota}^{\epsilon}_i} 
\ar[d]_{\Delta_{\G_m}} & 
\G_m \times \square^{n+2} \ar[d]^{\Delta_{\G_m}} \\
\G_m \times \G_m \times \square^{n+1} \ar[r]_{{\iota}^{\epsilon}_i} &
\G_m \times \G_m \times \square^{n+2}}
\]
is Cartesian, we see in $z^{q+2}(X \times \G_m, n+1)$ that for any
$\alpha \in \un{\TZ}^{q}(X, n ;m)$ with $\partial (\alpha) = 0$,
we have
\[
\begin{array}{llll}
2 \delta^2(\alpha) & = & 2 \alpha * (t, u, u) & \\
& = & \alpha * 2 (t, u, u) & \\
& = & \alpha * \partial \Gamma & (\mbox{by ~\eqref{eqn:Tcycles6}}) \\
& = & \Delta^*_{\G_m}\left(\alpha \times \partial \Gamma \right) & \\
& = & \Delta^*_{\G_m}\left(\alpha \times \partial \Gamma - 
\partial \alpha \times \Gamma \right) & \\    
& = & \Delta^*_{\G_m}\left(\partial (\alpha \times \Gamma) \right) &
(\mbox{by the Claim}) \\
& = & \partial \left(\Delta^*_{\G_m} (\alpha \times \Gamma) \right) & \\
& = & \partial (\alpha * \Gamma). 
\end{array} 
\] 
Since $\alpha \star \Gamma \in \un{\TZ}^{q+2}(X, n+3 ;m)$ by 
Lemma~\ref{lem:boundary}, we conclude that $2 \delta^2(\alpha) = 0$
as a class in $\TH^{q+2}(X, n+2, m)$. Since ${\rm char}(k) \neq 2$, we
conclude from Corollary~\ref{cor:module} that the homology class of
$\delta^2(\alpha)$ is zero in $\TH^{q+2}(X, n+2, m)$.
\end{proof}  
The following is the main result of this section.
\begin{thm}\label{thm:preDGA}
Let $X$ be a smooth projective variety over a field $k$ such that
char$(k) \neq 2$. Then the additive higher Chow groups
$\left(\TH(X), \wedge \right)$ is a graded-commutative
algebra which is equipped with a differential operator $\delta$ of degree one
satisfying $\delta^2 = 0$. Moreover, this differential operator commutes with
the pull-back and push-forward maps of additive higher Chow groups. 
\end{thm}
\begin{proof}
It follows directly from Corollary~\ref{cor:wedge-product**} and
Proposition~\ref{prop:Dsquare}. The commutativity of $\delta$ with the
pull-back and push-forward maps can be directly checked from its
definition.
\end{proof}
\begin{remk}\label{remk:char2} It seems that the assumption 
char$(k) \neq 2$ in Corollary~\ref{cor:module} and Theorem~\ref{thm:preDGA}
is not serious and can be removed using the infinite pro-$l$ extension of the 
field for $l \neq 2$. We do not go into this here.
\end{remk}

\section{Differential operator and Leibniz rule}\label{section:DOL}
In this section we introduce another differential operator on the
additive cycle complexes. This differential is an analogue of the Connes' 
boundary operator in the theory of Hochschild and cyclic homology 
({\sl cf.} \cite[Chapter~2]{Loday}).
We shall show that this satisfies the Leibniz rule
for the wedge product on the admissible cycle classes. We shall comment about
the relation between the two differential operators towards the end of this
section.

For $1 \le i \le n$, let $\sigma_{i}$ be the permutation 
\[
\sigma_i(j) =
\left\{\begin{array}{ll}
i \ \ \ \ \ \ \ \mbox{if $j=1$} \\
j-1 \ \ \mbox{if $2 \le j \le i$} \\
j \ \ \ \ \ \ \ \mbox{if $j > i$}.
\end{array}
\right.
\]
Let $\delta_i : X \times B_n \to X \times B_{n+1}$ be the rational map
$\phi^i_n = \sigma_{i} \circ \phi_n$, where $\phi_n$ is defined in 
~\eqref{eqn:diagonal}. In particular, we have $\phi^1_n = \phi_n$. Since 
$\sigma_i$ defines an automorphism of $X \times \ov B_n$ which preserves the 
modulus condition and proper intersection, it follows from 
Proposition~\ref{prop:DR} that for any admissible cycle 
$Z \in \un{\TZ}^q(X, n;m)$, $V = \phi^i_n(Z^{\times})$ defines an admissible 
cycle $\delta_i(Z) = V \in \un{\TZ}^{q+1}(X, n+1;m)$. Thus $\delta_i =
\sigma^*_i \circ \delta$. We put 
\begin{equation}\label{eqn:alt}
\delta_{alt} = \stackrel{n}{\underset {i =1} \sum} (-1)^i 
\delta_i : \ {\TZ}^q(X, n;m) \to {\TZ}^{q+1}(X, n+1;m).
\end{equation}
These $\delta_i$'s satisfy the following identities.
\begin{lem}\label{lem:deltaidentities}
For $i, j \in \{ 1, \cdots, n\}$, we have
\begin{equation}\label{eqn:Pind}
\left\{\begin{array}{lll}
& \delta_i \delta _j = \delta _{j+1} \delta _i , & \mbox{ if } i \leq j, \\
& \delta_i \delta_j = \delta_j \delta_{i-1}, & \mbox{ if } i > j.
\end{array}
\right.
\end{equation}
\end{lem}
\begin{proof} This is obvious from the definition of $\delta_i$'s.
\end{proof}
\begin{lem}\label{lem:del2}
We have $\delta^2_{alt} = 0$.
\end{lem}
\begin{proof}
Indeed $\delta^2_{alt}$ is,
\[
\left( \sum_{i=1} ^{n+1} (-1)^i \delta_i \right) 
\left(\sum_{j=1} ^{n} (-1)^j \delta_j \right) = 
\sum_{i \leq j} (-1)^{i+j} \delta_i \delta _j + \sum_{i \geq j+1} (-1)^{i+1} 
\delta_i \delta_j 
\]
\[
= \sum_{i \leq j} (-1)^{i+1} \delta_i \delta _j + \sum_{i \geq j+1} 
(-1)^{i+j} \delta _j \delta_{i-1}.
\]
For the right hand side, use the substitution $i-1 = j'$ and $j = i'$ so that 
we have
\[
\delta^2_{alt} = \sum_{i \leq j } (-1)^{i+j} \delta_i \delta _j + 
\sum_{i' \leq j' } (-1)^{i' + j' + 1} \delta_{i'} \delta_{j'} = 0.
\]
\end{proof}
One major drawback of $\delta_{alt}$ is that unlike $\delta = \delta_1$,
it does not have good commutativity (or anti-commutativity) relations with 
the boundary operator $\partial$ of the additive cycle complex.
We shall show in the next section that $\delta_{alt}$ still defines 
an operator on the homology groups. At this stage, we note that 
$\delta_{alt}$ and $\partial$ satisfy the following properties.
\begin{lem}\label{lem:pdelta} The following identities hold, where $\epsilon 
\in \{ 0, \infty \}$:
\begin{equation}\label{eqn:partialF}
\left\{\begin{array}{lll}
& \partial_i ^\epsilon\delta_k = \delta_{k-1} \partial_i ^\epsilon, & 
\mbox{ if } i <k, \\
& \partial_i ^\epsilon \delta_k = 0 , & \mbox{ if } i = k,\  \\
& \partial_i ^\epsilon \delta_k = \delta_k \partial ^\epsilon _{i-1}, & 
\mbox{ if } i > k. 
\end{array}
\right.
\end{equation}
Equivalently,
\begin{equation}\label{eqn:deltaF}
\left\{\begin{array}{lll}
& \delta_k \partial_i ^\epsilon = \partial_{i+1} ^\epsilon \delta_k , & 
\mbox{ if } k \leq 
i, \\
& \delta_k \partial_i ^\epsilon = \partial_i ^ \epsilon \delta_{k+1}, & 
\mbox{ if } k \geq i.
\end{array}
\right.
\end{equation}
In particular, $\partial_{i+1} ^\epsilon \delta_i = \partial_i ^\epsilon 
\delta_{i+1}$.
\end{lem}
\begin{proof} This is straightforward, while it takes some patience to keep 
track of the indices correctly.
\end{proof}
We will come back to this issue about the interaction of $\delta_{alt}$ with 
$\partial$ in the next section. See Lemma \ref{reason for normalization} and 
Question \ref{ques:QI}.

\subsubsection{Leibniz rule}
We now show that the differential $\delta_{alt}$ 
is in fact a derivation for the wedge product on the additive cycle complex. 
We first define a new operator on a pair of additive cycles which is the 
cycle theoretic analog of the {\sl cyclic shuffle product} in the
Hochschild complex in \cite[Section~4.3.2]{Loday}.
Recall that this cyclic shuffle product is used to show that the
Conne's boundary operator is a derivation for the wedge product on the
Hochschild homology. We prove here the analogous statement for the 
additive higher Chow groups. 

Consider the rational map
\begin{equation}\label{eqn:prod*}
{\mu}' : X \times X \times \G_m \times \G_m \times
{\square}^{n_1 + n_2 -2} \times \square \to
X \times X \times \G_m \times {\square}^{n_1 + n_2}
\end{equation}
\[
{\mu}' \left(x, t_1, t_2, y_1, \cdots , y_{n_1 + n_2 -2}, y\right)
= \left(x, t_1t_2, y, \frac{t_1y - 1}{t_1t_2y-1}, y_1, \cdots , 
y_{n_1 + n_2 -2}\right).
\]
For two irreducible admissible cycles $Z_i \in {\TZ}^{q_i}(X, n_i, m)$ for
$i = 1,2$, let $Z_1 {\times}' Z_2$ be the closure of 
${\mu}'\left((Z_1 \times Z_2)\times \square\right)$ in
$X \times X \times \G_m \times {\square}^{n_1 + n_2}$, where we omit a 
suitable transposition from our notations. As before, we put
$n = n_1+n_2-1$ and $q = q_1+q_2-1$. 
\begin{prop}\label{prop:prod*0}
$Z_1 {\times}' Z_2$ is an admissible cycle in $\TZ^{q}(X \times X , n+2;m)$.
\end{prop}
\begin{proof} We first prove the modulus condition for 
$Z = Z_1 {\times}' Z_2$. We consider the commutative diagram
\begin{equation}\label{eqn:prod*1}
\xymatrix{
X \times X \times \G_m \times \G_m \times
{\square}^{n_1 + n_2 -2} \times \square \ar[r]^{\hspace*{1cm} \mu '}\ar[d] &
X \times X \times \G_m \times {\square}^{n_1 + n_2} \ar[d] \\
X \times X \times \G_m \times \G_m \times
{\square}^{n_1 + n_2 -2} \ar[r]_{\mu} & 
X \times X \times \G_m \times {\square}^{n_1 + n_2-2},}
\end{equation}
where the vertical arrows are the natural projections. In particular, we
get the map $Z \to \mu_*(Z_1 \times Z_2)$ under the projection map. 
Let $\ov Z$ and $\ov {\mu_*(Z_1 \times Z_2)}$ denote the closures of $Z$
and $\mu_*(Z_1 \times Z_2)$ in $X \times {\widehat B}_{n+2}$ and
$ X \times {\widehat B}_n$ respectively. 
Thus we get a commutative diagram 
\begin{equation}\label{eqn:prod*2}
\xymatrix{
\ov Z \ar[r] \ar[d] & X \times X \times {\widehat B}_{n+2} \ar[d]^{p} \\
\ov {\mu_*(Z_1 \times Z_2)} \ar[r] & X \times X \times {\widehat B}_{n}.}
\end{equation}
We have shown in Corollary~\ref{cor:wedge-AD} that $\mu_*(Z_1 \times Z_2)$
satisfies the modulus condition.
Since $p^*(F^1_{n, i}) = F^1_{n+2,i}$ and $p^*(F_{n, 0}) = F_{n+2, 0}$,
we see that the modulus condition for $\mu_*(Z_1 \times Z_2)$ implies the
same for $p^*(\mu_*(Z_1 \times Z_2))$. The modulus condition for 
$Z$ now follows from Proposition~\ref{prop:restp}.

Now we compute the various boundaries of $Z$. It is easy to see from
~\eqref{eqn:prod*} that 
\[
\partial_1^0(Z) = 0, \ \partial_1^{\infty}(Z) = \sigma_{n_1} \cdot \left(
\mu_*\left(Z_1 \times \delta(Z_2)\right)\right),
\]
\[
\partial_2^0(Z) = \mu_*\left(\delta(Z_1) \times Z_2\right), \ 
\partial_2^{\infty}(Z) = \delta\left(\mu_*(Z_1 \times Z_2)\right).
\]
For $3 \le i \le n+1$, we have
\[
\partial_i^{\epsilon}(Z) =
\left\{\begin{array}{ll}
\partial_{i-2}^{\epsilon}(Z_1) {\times}' Z_2  
\ \ \ \ \ \ \ \mbox{if $3 \le i \le n_1+1$} \\
Z_1 {\times}' \partial_{i-n_1-1}^{\epsilon}(Z_1)  
\ \ \mbox{if $n_1+2 \le i \le n+1$}.
\end{array}
\right.
\]
Since $Z_i$'s are admissible cycles, the above automatically imply the
proper intersection property of $Z$.
\end{proof} 
Using Proposition~\ref{prop:prod*0}, we can define the our 
{\sl cyclic shuffle product} as 
\begin{equation}\label{eqn:csp}
Z_1 \bar\wedge' Z_2:= 
\sum_{\nu \in {\rm \mathbb{P}erm}_{(1, n_1-1, n_2 -1)} }
\sgn (\nu) \nu \cdot (Z_1 \times ' Z_2) \in \TZ^{q+1} (X \times X, n+2;m),
\end{equation} where the permutations $\nu \in 
{\rm \mathbb{P}erm}_{(1, n_1 -1, n_2 -1)}$ act on the given set of $n$ objects 
$\{(1,2), 3, \cdots, n+1\}$ in the obvious way, treating the element $(1,2)$ 
as a single 
object. This induces the action $\nu \cdot (\xi \times ' \eta)$.
We extend this bilinearly to get the cyclic shuffle product
\begin{equation}\label{eqn:csp0}
\TZ^{q_1} (X, n_1;m) \otimes \TZ^{q_2} (X, n_2;m) \xrightarrow{\bar\wedge'}
\TZ^{q+1} (X \times X, n+2;m).
\end{equation} 
\begin{prop}[Leibniz rule]\label{prop:derivation}
Let $\xi \in {\TZ}^{q_1} (X, n_1; m), \eta \in 
{\TZ}^{q_2} (X, n_2 ;m)$. Then, in the group ${\TZ}^{q+1} 
(X \times X, n+1; m)$, we have
\begin{equation}\label{eqn:derivation}
\delta_{alt}(\xi \bar\wedge \eta) - (\delta_{alt}\xi) \bar\wedge \eta - 
(-1)^{n_1 -1} \xi \bar\wedge (\delta_{alt} \eta)
\end{equation}
\[
\hspace*{5cm} = 
\partial (\xi \bar \wedge' \eta) - (\partial \xi) \bar \wedge ' \eta - 
(-1)^{n_1 -1} \xi \bar\wedge ' (\partial \eta),
\]
where $n = n_1  + n_2-1$, $q= q_1 + q_2 -1$.
\end{prop}
\begin{proof}
It follows from Proposition~\ref{prop:prod*0} that 
\begin{equation}\label{eqn:derivation0}
\begin{array}{lll}
\partial\left(\xi \times ' \eta \right) & = & 
\stackrel{n+1}{\underset {i =1} \sum} (-1)^i 
(\partial_i ^{\infty} - \partial_i ^{0})\left(\xi \times ' \eta \right) \\ 
& = & \delta\left(\mu_*(\xi \times \eta)\right) - [\sigma_{n_1} \cdot \left(
\mu_*\left(\xi \times \delta \eta \right)\right) + 
\mu_*\left(\delta \xi \times \eta \right)] \\
& & + 
\stackrel{n_1+1}{\underset {i =3} \sum} (-1)^i
\{(\partial_{i-2}^{\infty} - \partial^0_{i-2})(\xi)\} {\times}' \eta \\ 
& & 
+ \stackrel{n+1}{\underset {i =n_1+2} \sum} (-1)^i
\xi {\times}' \{(\partial_{i-n_1-1}^{\infty} - \partial^0_{i-n_1-1})(\eta)\}
\\
& = & \delta\left(\mu_*(\xi \times \eta)\right) -
[\sigma_{n_1} \cdot \left(\mu_*\left(\xi \times \delta \eta \right)\right) + 
\mu_*\left(\delta \xi \times \eta \right)] \\
& & + \partial \xi {\times}' \eta + (-1)^{n_1-1} \xi {\times}' \partial \eta.
\end{array} 
\end{equation}
In particular, we have
\begin{equation}\label{eqn:derivation1}
\delta\left(\mu_*(\xi \times \eta)\right) - 
\mu_*\left(\delta \xi \times \eta \right) - 
\sigma_{n_1} \cdot \left(\mu_*\left(\xi \times \delta \eta \right)\right)
\end{equation}
\[
\hspace*{7cm} = \partial\left(\xi \times ' \eta \right) - 
\partial \xi {\times}' \eta - (-1)^{n_1-1} \xi {\times}' \partial \eta.
\]

Since the desired identity ~\eqref{eqn:derivation} of the proposition and 
~\eqref{eqn:derivation1}
differ only by the action of various permutations, it is now enough to show
that the identity holds on the coordinates of $X \times B_{n+2}$.

One ingredient in the proof is the application of 
Proposition~\ref{prop:perm-identity} and Lemma~\ref{lem:Per*0}
to the triple shuffles ${\rm \mathbb{P}erm}_{(1,r,s)}$. Observe that the map 
$\delta_{alt}$ can be written as a sum over the set 
${\rm \mathbb{P}erm}_{(1, n)}$ of 
double shuffles. Indeed, for the coordinate $(t, y_1, \cdots, y_n)$, we have
\begin{eqnarray*}
\delta_{alt} (t, y_1, \cdots, y_n) &=&  \sum_{i=1} ^{n+1} (-1)^i ( t, y_1, 
\cdots, y_{i-1} , \underset{i^{\rm th}}
{\underbrace{\frac{1}{t}}}, y_i, \cdots, y_n) \\
& = & - \sum_{\tau \in {\rm \mathbb{P}erm}_{(1, n)}} (\sgn (\tau)) \tau \cdot 
\left( t, \frac{1}{t}, y_1, \cdots, y_n\right).
\end{eqnarray*}

We first compute the term on the left hand side of the
identity \eqref{eqn:derivation} of the proposition on the level the 
coordinates of $B_{n_1}$, 
$B_{n_2}$. Since the variety $X$ doesn't play a role in the calculation, 
we shrink points of $X$ from our notations. 
Let $\xi = (t_1, y_1, \cdots, y_{n_1-1}), 
\eta = (t_2, y_{n_1}, \cdots, y_{n-1})$. We then have
\[
\left\{\begin{array}{lll}
\xi \times \eta = (t_1, t_2, y_1, \cdots, y_{n-1} ) \\
 \xi \times' \eta = C_{t_1,t_2} \times (y_1, \cdots, y_{n-1}).
\end{array}
\right.
\]
Here $C_{t_1,t_2}:= \{t_1\} {\times}' \{t_2\} \subset \G_m \times \square^2$ 
is the 
parameterized curve  
\[
C_{t_1,t_2} = \left\{\left(t_1t_2, y, \frac{t_2 y-1}{t_1t_2y-1}\right)| y \in 
k \right\}
\]
for any given points $t_1, t_2 \in \G_m$. As shown before, this is an
admissible 1-cycle and its boundary is given by 
({\sl cf.} \cite[Lemma~2.5]{P2})
\begin{equation}\label{eqn:bC}
\partial C_{t_1,t_2} = \left(t_1t_2, \frac{1}{t_1}\right) +
\left(t_1t_2, \frac{1}{t_2}\right) - \left(t_1t_2, \frac{1}{t_1t_2}\right).
\end{equation}This property will play an important role in the calculation. 

To simplify the notations, we introduce new indices $r, s, u$ by letting 
$r: = n_1 -1, s := n_2-1,$ and $ u := n-1 =r+s$. 

Then, by a direct calculation we have for the first term of 
\eqref{eqn:derivation},
\begin{eqnarray*}
& & \delta_{alt}(\xi \wedge \eta) \\
&=& \delta_{alt} \left( \sum_{\sigma \in {\rm \mathbb{P}erm}_{(r, s)}} 
(\sgn (\sigma)) \sigma \cdot (\mu_* \left( \xi \times \eta \right)) \right) \\
&=& \delta_{alt} \left( \sum_{\sigma \in {\rm \mathbb{P}erm}_{(r, s)}} 
(\sgn (\sigma)) \sigma \cdot (t_1t_2, y_1, \cdots, y_u) \right) \\
&=& \sum_{\sigma \in {\rm \mathbb{P}erm}_{(r, s)}} (\sgn (\sigma)) 
\delta_{alt} (t_1t_2, y_{\sigma^{-1} (1)}, \cdots, y_{\sigma^{-1} (u))})\\
&=& - \sum_{\sigma \in {\rm \mathbb{P}erm}_{(r, s)}} (\sgn(\sigma)) \
\sum_{\tau \in {\rm \mathbb{P}erm}_{(1, u)}} (\sgn (\tau)) \tau \cdot 
(xy, \frac{1}{t_1t_2} , y_{\sigma ^{-1} (1)}, \cdots, y_{\sigma ^{-1} (u)} )\\
&=& - \sum_{\sigma \in {\rm \mathbb{P}erm}_{(r, s)}} (\sgn (\sigma)) 
\sum_{\tau \in {\rm \mathbb{P}erm}_{(1, u)}} (\sgn (\tau)) 
(\sigma_{\tau} \cdot \tau) \cdot (t_1t_2, \frac{1}{t_1t_2}, y_1, \cdots, y_u) 
\\
&=& - \left( \sum_{\nu \in {\rm \mathbb{P}erm}_{(1, r, s)}} (\sgn (\nu)) 
\nu\right)  \cdot (t_1t_2, \frac{1}{t_1t_2}, y_1, \cdots, y_u),
\end{eqnarray*}
where $\sigma_{\tau}$ and the last equality are from Lemma~\ref{lem:Per*0}.

The second term of \eqref{eqn:derivation} is, 
\begin{eqnarray*}
& &(\delta_{alt} \xi) \wedge \eta \\
&=& \mu \left( \sum_{\sigma \in {\rm \mathbb{P}erm}_{(r +1, s)}} 
(\sgn (\sigma)) \sigma \cdot \left( (\delta_{alt} \xi) \times \eta \right)
\right) \\
&=& -\mu_*  \sum_{\sigma \in {\rm \mathbb{P}erm}_{(r +1, s)}} (
\sgn (\sigma)) \sigma \cdot   \\ 
& &  \left( \sum_{\tau \in (_{(1, r)} }(\sgn (\tau)) \tau 
\cdot (t_1, \frac{1}{t_1}, y_1, \cdots, y_{r} )  \right) 
\times (t_2, y_{r +1}, 
\cdots, y_{u} ) \\
&=& -\left( \sum_{\sigma \in {\rm \mathbb{P}erm}_{(r + 1, s)}} 
(\sgn (\sigma)) \sigma \right) \cdot \\ 
& & \left( \sum_{\tau \in {\rm \mathbb{P}erm}_{(1, r)}} (\sgn (\tau)) 
(\tau \times {\rm Id}_{s})  \right) 
\cdot (t_1t_2, \frac{1}{t_1}, y_1, \cdots, y_u) 
\\
&=& - \left( \sum_{\nu \in {\rm \mathbb{P}erm}_{(1, r, s)}} (\sgn (\nu)) 
\nu \right) \cdot (t_1t_2, \frac{1}{t_1}, y_1, \cdots, y_u),
\end{eqnarray*}
where the last equality follows from Proposition~\ref{prop:perm-identity}.

Before we compute the third term of \eqref{eqn:derivation}, first note that 
\begin{eqnarray*}
\delta_{alt} (t_2, y_{r+1}, \cdots, y_n) 
&=& \sum_{i=r+1} ^{u+1} (-1)^{i-r} ( t_2, y_{r+1}, \cdots, y_{i+r-1} , 
\underset{i+r^{\rm th}}{\underbrace{\frac{1}{t_2}}}, y_{i+r}, \cdots, y_n) \\
&=& - (-1)^r \sum_{\tau \in {\rm \mathbb{P}erm}_{(1, s)}} (\sgn (\tau)) 
\tau \cdot ( t_2, \frac{1}{t_2}, y_{r+1}, \cdots, y_u).
\end{eqnarray*}
Hence, for the third term of \eqref{eqn:derivation}, we have
\begin{eqnarray*}
& & (-1)^{r} \xi \wedge (\delta_{alt} \eta) \\
&=& - \left( \sum_{\sigma \in {\rm \mathbb{P}erm}_{(r, s +1)}} 
(\sgn (\sigma)) \sigma \right) \cdot \\ 
& & \left( \sum_{\tau \in {\rm \mathbb{P}erm}_{(1, s)}} (\sgn (\tau)) 
({\rm Id}_{r} \times \tau) \right) 
\cdot (t_1t_2, \frac{1}{t_2}, y_1, \cdots, y_u )\\
&=& - \left( \sum _{\nu \in {\rm \mathbb{P}erm}_{(1, r, s)}} 
(\sgn (\nu)) \nu \right) \cdot (t_1t_2, \frac{1}{t_2}, y_1, \cdots, y_u),
\end{eqnarray*}
where the last equality follows from Proposition~\ref{prop:perm-identity}. 

Thus, the left hand side of the equation \eqref{eqn:derivation} is 
compactified into
\begin{equation}\label{eqn:compactified}
  \delta_{alt}(\xi \wedge \eta) - (\delta_{alt} \xi) 
\wedge \eta - (-1)^{r} \xi \wedge (\delta_{alt} \eta)\end{equation}
\begin{equation*}
= - \sum_{\nu \in {\rm \mathbb{P}erm}_{(1, r, s)}} 
(\sgn (\nu))\nu  \cdot  \left( (t_1t_2, \frac{1}{t_1}) + 
(t_1t_2, \frac{1}{t_2}) - (t_1t_2, \frac{1}{t_1t_2})\right) 
\times (y_1, \cdots, y_u).
\end{equation*}  
On the other hand, for the coordinate points, we have for the first term of 
the right hand side of \eqref{eqn:derivation},
\begin{eqnarray}\label{eqn:first right hand side}
\partial (\xi \bar\wedge' \eta) &=&  \partial 
\left( \sum_{\nu \in {\rm \mathbb{P}erm}_{(1, r, s)}} 
(\sgn (\nu)) \nu \right) \cdot C_{t_1,t_2} \times (y_1, \cdots, y_{u}).
\end{eqnarray} 
Since we have by ~\eqref{eqn:bC} the equation
\[
\left( t_1t_2, \frac{1}{t_1}\right) + \left( t_1t_2, \frac{1}{t_2}\right) - 
\left( t_1t_2, \frac{1}{t_1t_2} \right) = \partial C_{t_1,t_2},
\]
for each $\nu \in {\rm \mathbb{P}erm}_{(1,r,s)}$, 
the four faces of $\nu \cdot (\xi \times' \eta)$ in the sum 
\eqref{eqn:first right hand side} that interact with 
$C_{t_1,t_2}$, {\sl i.e.}, $\partial_i ^ \epsilon$ with $i \in \nu (1,2)$, 
cancel out 
the corresponding terms of \eqref{eqn:compactified}. This process cancels all 
terms in \eqref{eqn:compactified}, thus all the terms of the left hand side 
of \eqref{eqn:derivation}. Hence, we need to see what happens for the 
remaining 
faces of \eqref{eqn:first right hand side}. But we have already seen in 
~\eqref{eqn:derivation0} that
for any general admissible cycles $\xi$ and $\eta$, 
\begin{eqnarray*} 
\sum_{i=3} ^{n+1} (-1)^i (\partial_i ^{\infty} - \partial_i ^{0}) 
(\xi \times ' \eta)  = \left( \sum_{i=3} ^{n_1 +1} 
(-1)^i (\partial_{i-2} ^{\infty} - \partial_{i-2} ^{0}) \xi \right) 
\times ' \eta \\
+ \xi \times' \left( \sum_{i=n_1 + 2}  ^{n+1} (-1)^i 
(\partial_{i-n_1 -1} ^{\infty} - \partial_{i-n_1 -1} ^{0}) \eta \right) \\
= (\partial \xi) \times' \eta + (-1)^{n_1 -1} \xi \times' (\partial \eta).
\end{eqnarray*}
We apply this argument for the faces $\partial_i ^\epsilon$ with 
$i \not \in \nu (1,2)$ to each $\nu \cdot ( \xi \times ' \eta)$, where 
$\nu \in {\rm \mathbb{P}erm}_{(1,r,s)}$, and take the signed sum. This gives 
the remaining terms $(\partial \xi) \bar\wedge' \eta + 
(-1)^{n_1 -1} \xi \bar{\wedge} ' (\partial \eta)$ of the right hand side of 
~\eqref{eqn:derivation}. 
This completes the proof of Proposition~\ref{prop:derivation}.
\end{proof}

\section{Normalized additive cycle complex}\label{section:NACom}
We have seen in the previous section that the differential operator
$\delta_{alt}$ on the additive cycle complex has all the nice properties
except that it does not commute (or anti-commute) with the boundary map 
$\partial$. In this section, we rectify this anomaly by introducing the 
normalized version of the additive cycle complex. This is analogous to the 
similar construction of S. Bloch in \cite[Theorem 4.4.2]{Bl3}. It turns out 
that $\delta_{alt}$
indeed has good behaviors with respect to the boundary operator of the
normalized complex. Our final goal is then achieved by showing that
the homology of the normalized additive cycle complex does not change
our additive higher Chow groups. We begin with the following 
construction of M. Levine which appeared in \cite{Le0} to study 
Bloch's higher Chow groups. This is essentially equivalent to the method of 
S. Bloch in \cite{Bl3}. We suitably adapt this Levine's construction to the 
additive
world in what follows next.  

\subsection{Homotopy variety}
In the following construction, we shall make an identification between 
$\square$ and $\A^1$
via the map 
\begin{equation}\label{eqn:SA}
\square \to \A^1 ; \ \ \ y \mapsto 1/(1-y).
\end{equation}
This gives the isomorphism $\left(\P^1, \{0, 1, \infty \}\right) \cong
\left(\P^1, \{1, \infty , 0 \}\right)$. The boundary map of the corresponding 
cycle complex
under this identification is given by $\sum_i (-1)^i
(\partial^0 _i - \partial^1 _i)$.
Let $X$ be a smooth projective variety and let $i_n : W^X_n \to X \times \G_m
\times \square^{n+1} \times \mathbb{P} ^1$ be the closed subvariety defined 
by the
equation
\begin{equation}\label{eqn:H0}
t_0(1-y_n)(1-y_{n+1}) = t_0 - t_1,
\end{equation}
where $(y_1, \cdots, y_n)$ are the coordinates of $\square^n$ and $(t_0: t_1)$
are the homogeneous coordinates of $\P^1$.
Let ${\pi}_n : W_n^X \to X \times \G_m \times \square^{n}$ be the map 
defined by 
\begin{equation}\label{eqn:H00}
{\pi}_n(x, t, y_1, \cdots, y_{n+1}, (t_0:t_1)) =
(x, t, y_1, \cdots, y_{n-1}, y_n + y_{n+1} - y_n y_{n+1}).
\end{equation}

Let $\left((u^1_0: u^1_1), \cdots, (u^{n+1}_0: u^{n+1}_1)\right)$ denote
the homogeneous coordinate of ${(\P^1)}^{n+1}$. We identify $\square^{n+1}$
to the open subset of ${(\P^1)}^{n+1}$ given by 
$\stackrel{n+1}{\underset {i = 1} \prod} \{u^i_0 \neq 0\}$, and
we set $y_i = {u^i_1}/{u^i_0}, y = {t_1}/{t_0}$.
In terms of these homogeneous coordinates, the projectivization $\ov {W^X_n}$
of $W^X_n$ in $X \times \square \times  {(\P^1)}^{n+1} \times \P^1$ is given 
by the equation
\begin{equation}\label{eqn:H01}
t_0(u^{n}_0 - u^n_1)(u^{n+1}_0 - u^{n+1}_1) = u^n_0 u^{n+1}_0(t_0 - t_1).
\end{equation}

Let ${\ov \theta}_n : X \times \square \times  {(\P^1)}^{n+1} \times \P^1
\to X \times \square \times  {(\P^1)}^{n-1} \times \P^1$ be the natural
projection map given by
\begin{equation}\label{eqn:mod*}
{\ov {\theta}}_n \left(x, t, (u^1_0;u^1_1), \cdots, (u^{n+1}_0:u^{n+1}_1),
(t_0:t_1)\right) = 
\end{equation}
\[
\hspace*{7cm} \left(x, t, (u^1_0;u^1_1), \cdots, (u^{n-1}_0:u^{n-1}_1),
(t_0:t_1)\right)
\]
and let $\theta_n$ be its restriction to the open set $ X \times \G_m \times
\square^{n+1} \times \square$. Here we identify $\G_m$ as $\square \backslash
\{1\}$. Let ${\ov \pi}_n : \ov {W^X_n} \to X \times \square \times  
{(\P^1)}^{n-1} \times \P^1$ be the restriction of ${\ov \theta}_n$ to
$\ov {W^X_n}$.

Let $p_n : X \times \G_m \times \square^{n+1} \times \P^1\to
X \times \G_m \times \square^{n+1}$ be the natural projection.
\begin{lem}\label{lem:smooth}
$W^X_n \cap \{t_0 = 0 \} = \emptyset$ and hence $W^X_n$ is in fact
contained in the open subset $X \times \G_m \times \square^{n+1} \times 
\square$. The variety $\ov {W^X_n}$ (and hence $W^X_n$) is smooth. Moreover,
$\pi_n$ and ${\ov \pi}_n$ are flat and surjective morphisms of relative 
dimension one.
\end{lem}    
\begin{proof} The first assertion is immediate from the defining equation of
$W^X_n$. Using this assertion,  we can write the restriction of $p_n$ on 
$W^X_n$ as
\[
W^X_n \inj X \times \G_m \times \square^{n+1} \times \square \to 
X \times \G_m \times \square^{n+1},
\]
where the first inclusion is given by the equation $y = 
y_n + y_{n+1} - y_n y_{n+1}$. Since $X$ is smooth, it is now easy to see
using the Jacobian criterion that $W_n^X$ is smooth and the above
composite map is an isomorphism. 
Furthermore, under this isomorphism, the map ${\pi}_n$ is just the 
projection $(x, t, y_1, \cdots, y_{n}, y) \mapsto
(x, t, y_1, \cdots, y_{n-1}, y)$, as can be checked from the equation of
$W^X_n$. This also shows that $\pi_n$ is in fact smooth and surjective map
of relative dimension one. To prove the smoothness of $\ov {W^X_n}$, we can
check it locally on an open set of points with coordinates
$\left(x, t, (u^1_0: u^1_1), \cdots, (u^{n+1}_0: u^{n+1}_1), 
(t_0: t_1)\right)$ where either of $u^n_{i}, u^{n+1}_i, t_i$ is non-zero
for $i = 0, 1$. In any such open set, $\ov {W^X_n}$ has the equation of the
form that defines $W^X_n$ and hence is smooth. It is also easy to check
using these local coordinates that ${\ov \pi}_n$ is of relative dimension one.
Moreover, as ${\ov \theta}_n$ is projective and $\pi_n$ is surjective, we
see that ${\ov \pi}_n$ is projective and surjective. In particular,
it is flat ({\sl cf.} \cite[Exercise III-10.9]{Hart}). 
This proves the lemma.
\end{proof}
\begin{lem}\label{lem:H02}
The diagram 
\begin{equation}\label{eqn:mod0}
\xymatrix@C.8cm{
W^X_n \ar[r]^{i_n \hspace*{1.5cm}} \ar[dr]_{\pi_n} & 
X \times \G_m \times \square^{n+1} \times \square 
\ar[r]^{j_n} \ar[d]^{\theta_n} &
X \times \square \times {(\P^1)}^{n+1} \times \P^1 \ar[d]^{{\ov \theta}_n} \\
&  X \times \G_m \times \square^{n} \ar[r]_{j'_n} &
X \times \square \times {(\P^1)}^n}
\end{equation}
commutes.
\end{lem}
\begin{proof} By Lemma~\ref{lem:smooth}, $W^X_n$ is contained in the open
subset $X \times \G_m \times \square^{n+1} \times \square$, where it is given
by the equation $y = y_n + y_{n+1} - y_n y_{n+1}$. It is clear from 
the definition of $\theta_n$ in ~\eqref{eqn:mod*} that the triangle on the 
left commutes. The right square commutes by the definitions of $\theta_n$ and 
${\ov \theta}_n$. Hence the outer trapezium also commutes.
\end{proof}
\begin{lem}\label{lem:H03}
Let $Z \subset X \times \G_m \times \square^{n}$ be a closed subvariety
which satisfies the modulus condition $M_{sum}$. Let $Z' = 
 {(i_n)}_*\left({\pi_n}^*(Z)\right)$. Then $Z'$ also satisfies the 
modulus condition $M_{sum}$. 
\end{lem}
\begin{proof}
Let $\ov Z$ and $\ov Z'$ denote the closures of $Z$ and $Z'$ in
$X \times \square \times {(\P^1)}^n$ and 
$X \times \square \times {(\P^1)}^{n+1} \times \P^1$ respectively.
Let ${\ov Z}^N$ and ${\ov Z'}^N$ denote the normalizations of $\ov Z$ and
$\ov Z'$ respectively. Using Lemmas~\ref{lem:smooth}, ~\ref{lem:H02}
and the projectivity of the map ${\ov \theta}_n$, we see that 
${\ov \theta}_n (\ov Z') = \ov Z$ in the Diagram~\eqref{eqn:mod0}.
Since $\ov {W^X_n}$ smooth, we get the following commutative diagram:
\begin{equation}\label{eqn:mod1}
\xymatrix@C.8cm{
{\ov Z'}^N \ar[r]_{g} \ar[d]^{f} \ar@/^1pc/[rr]^{\nu_Z'} & 
\ov {W^X_n} \ar[r]_{{\ov i}_n \hspace*{1.5cm}} \ar[dr]_{{\ov \pi}_n}
& X \times \square \times {(\P^1)}^{n+1} \times \P^1 \ar[d]^{{\ov \theta}_n} 
\\
{\ov Z}^N \ar[rr]_{\nu_Z} & & 
X \times \square \times {(\P^1)}^n,}
\end{equation}
where $f$ is the map of normal $k$-schemes induced by the surjective map 
$\ov Z' \to \ov Z$.
As before, let $F^{\infty}_{n+1, i}$ denote the Cartier divisor on 
${(\P^1)}^n$ defined by $\{y_i = \infty\}$ for $1 \le i \le n$ and we have
similar Cartier divisors $F^{\infty}_{n+2, i}$ on ${(\P^1)}^{n+1}$
for $1 \le i \le n+1$.
It is then easy to see from the defining equation of $\ov {W^X_n}$ in
~\eqref{eqn:H01} that 
\begin{equation}\label{eqn:mod2}
\begin{array}{lll} 
{\ov \pi}^*_n (F^{\infty}_{n+1, n}) & = &
{\ov i}^*_n (y = \infty) \\
& = & {\ov i}^*_n (t_0 = 0) \\
& \le & {\ov i}^*_n [(u^n_0 = 0) + (u^{n+1}_0 = 0)] \\
& = &  {\ov i}^*_n [(y_n = \infty) + (y_{n+1} = \infty)] \\ 
& = & {\ov i}^*_n [F^{\infty}_{n+2, n} +
F^{\infty}_{n+2, n+1}].
\end{array}
\end{equation}
Since ${\ov Z'}^N \to \ov {W^X_n}$ is a map of normal $k$-schemes, and since
${\ov \theta}^*_n (F^{\infty}_{n+1, i}) = F^{\infty}_{n+2, i}$ for
$1 \le i \le n-1$, we have
\begin{equation}\label{eqn:mod3}
\begin{array}{lll} 
\nu^*_{Z'} \circ {\ov \theta}^*_n (F^{\infty}_{n+1}) & = &
\stackrel{n}{\underset {i = 1} \sum} \nu^*_{Z'} \circ {\ov \theta}^*_n
(F^{\infty}_{n+1, i}) \\
& = & \nu^*_{Z'}[\stackrel{n-1}{\underset {i = 1} \sum} F^{\infty}_{n+2, i}]
+ g^* \circ {\ov \pi}^*_n (F^{\infty}_{n+1, n}) \\
& \le &
\nu^*_{Z'}[\stackrel{n-1}{\underset {i = 1} \sum} F^{\infty}_{n+2, i}]
+ g^* \circ {\ov i}^*_n [F^{\infty}_{n+2, n} + F^{\infty}_{n+2, n+1}] \\  
& = & \nu^*_{Z'}[\stackrel{n+1}{\underset {i = 1} \sum} F^{\infty}_{n+2, i}] \\
& = & \nu^*_{Z'} (F^{\infty}_{n+2}).
\end{array}
\end{equation}
Now the modulus condition for $Z$ implies that 
$\nu^*_{Z}[(m+1) F_{n+1, 0}] \le \nu^*_{Z}[F^{\infty}_{n+1}]$ which
implies that $f^* \circ \nu^*_{Z}[(m+1) F_{n+1, 0}] \le 
f^* \circ \nu^*_{Z}[F^{\infty}_{n+1}]$. This in turn implies that
$\nu^*_{Z'} \circ {\ov \theta}^*_n [(m+1) F_{n+1, 0}] \le
\nu^*_{Z'} \circ {\ov \theta}^*_n [F^{\infty}_{n+1}]$.
Since ${\ov \theta}^*_n (F_{n+1, 0}) = F_{n+2, 0}$, we conclude from 
~\eqref{eqn:mod3} that
$\nu^*_{Z'}[(m+1) F_{n+2, 0}] \le \nu^*_{Z'} [F^{\infty}_{n+2}]$
which proves the modulus condition for $Z'$.
\end{proof}

For any closed subvariety $Z \subset X \times \G_m \times \square^{n}$, let
\begin{equation}\label{eqn:WX}
W^X_n(Z) : = {(p_n)}_* \circ {(i_n)}_* \circ \pi^*_n (Z).
\end{equation}
Note that $W^X_n(Z)$ is a closed subvariety of 
$X \times \G_m \times \square^{n+1}$ since $p_n$ is projective.
\begin{lem}[{\sl cf.} \cite{Le0}]\label{lem:intersection}
For $Z$ as above, one has \\
$(1) \ \ W^X_n(Z) \cdot \{y_n = 0\} = Z = W^X_n(Z) \cdot \{y_{n+1} = 0\}$.
\\
$(2) \ \ Z \in z^q(X \times \G_m, n) \Rightarrow
W^X_n(Z) \in  z^q(X \times \G_m, n+1)$.
\end{lem} 
\begin{proof} Let $Z' = {(i_n)}_* \circ \pi^*_n (Z)$. Let $\square_i$ denote
the $i$th factor of $\square^{n+1}$ in $X \times \G_m \times \square^{n+1}
\times \square$ in Diagram~\eqref{eqn:mod0}. Then we see from 
Diagram~\eqref{eqn:mod0} that 
\begin{equation}\label{eqn:int0}
Z' = (Z \times \square_{n} \times \square_{n+1})
\cdot W^X_n = (Z \times \square_{n} \times \square_{n+1})
\cdot \{y = y_n + y_{n+1} - y_n y_{n+1}\}.
\end{equation}
Combining this with the equation of $W^X_n$ in
~\eqref{eqn:H0}, we see that 
\[
Z' \cdot \{y_n = \epsilon \} =
(Z \times \square_{n} \times \square_{n+1}) \cdot \{y = y_{n+1}\} \ \
{\rm for} \ \ \epsilon = 0, 1.
\]
Thus we get
\[
W^X_n(Z) \cdot \{y_n = 0\} =
p_n(Z') \cdot \{y_n = 0\} =
{(p_n)}_*[(Z \times \{0\} \times \square_{n+1}) \cdot
\{y = y_{n+1}\}] = Z.
\]
Using the same steps, we also see that
\[
W^X_n(Z) \cdot \{y_n = 1\} = (Z \cdot\{y = 1\}) \times \square.
\]
Since $Z' \cdot \{y_{n+1} = \epsilon \} =
(Z \times \square_{n} \times \square_{n+1}) \cdot \{y = y_{n}\}$,
the same calculation as before shows that
$W^X_n(Z) \cdot \{y_{n+1} = 0\} = Z$ and
$W^X_n(Z) \cdot \{y_{n+1} = 1\} = (Z \cdot\{y = 1\}) \times \square$.
This proves $(1)$. This in particular shows that the intersection
$W^X_n(Z) \cdot \{y_i = \epsilon \}$ is proper for $i = n, n+1$ and
$\epsilon = 0, 1$.

Now we calculate the other boundaries of $W^X_n(Z)$. It follows directly
from ~\eqref{eqn:int0} that for $1 \le i \le n-1$, one has 
$W^X_n(Z) \cdot \{y_i = \epsilon \} = W^X_{n-1}(Z \cdot \{y_i = \epsilon\})$.
Since $\pi_{n-1}$ is flat of relative dimension one
as shown in Lemma~\ref{lem:smooth}, we see that this intersection is proper.
This proves $(2)$, thus the lemma.
\end{proof} 
\begin{prop}\label{prop:addm}
For the modulus condition $M_{sum}$, let 
$Z \in \un{\TZ}^q(X, n+1;m)$ be an irreducible admissible cycle. Then
$W^X_n(Z) \in \un{\TZ}^q(X, n+2;m)$.
\end{prop}
\begin{proof} This follows by combining Lemmas~\ref{lem:H03} and
~\ref{lem:intersection}.
\end{proof}
Using Proposition~\ref{prop:addm}, we can define for every $n \ge 0$,
a group homomorphism
\begin{equation}\label{eqn:WX0}
W^X_n : \un{\TZ}^q(X, n+1;m)_{sum} \to \un{\TZ}^q(X, n+2;m)_{sum}
\end{equation}
by extending $W^X_n$ linearly. This homomorphism has the property that
$\partial^{0}_i \circ W^X_n = {\rm Id}$ for $i = n, n+1$ as shown in
Lemma~\ref{lem:intersection}.

\subsection{Normalized additive cycle complex}\label{subsection:NACcom*}
We now define the normalized version of our additive 
cycle complexes and study their properties. 
These complexes are the additive analogues of the similar constructions
of S. Bloch and M. Levine \cite{Bl3, Le0} for higher Chow groups. 
It turns out that the normalized
additive cycle complex has better properties. One expects that
the resulting additive Chow groups also form a kind of Witt complex so that
all the expected maps from the relative $K$-groups of the
infinitesimal deformations of smooth schemes to the additive higher Chow
groups actually factors through these normalized additive higher Chow
groups. Another outcome of using the normalized additive cycle complex
is the invention of the motivic version of the cycilc homology of
A. Connes \cite{Loday}. We shall deal with this aspect in the end of this 
section. 
\begin{defn}\label{defn:NACC}
Let $X$ be a smooth projective variety and let $M$ be any of the
modulus conditions $M_{sum}, M_{sup}, M_{ssup}$. For $n, m \ge 1$,
let ${\TZ}^q_N(X, n;m)$ be the subgroup of ${\TZ}^q(X, n;m)$
of cycles $\alpha$ such that $\partial^{0}_i(\alpha) = 0$
for $1 \le i \le n-1$ and $\partial^{\infty}_i(\alpha) = 0$ for
$2 \le i \le n-1$. Using the simplicial structure of the additive cycle
complex, it easy to see that for $\alpha \in {\TZ}^q_N(X, n;m)$,
one has $\partial^{\infty}_{1} \circ \partial^{\infty}_{1}(\alpha) = 0$.
Writing $\partial^{\infty}_{1}$ as $\partial^N$, we thus get a 
subcomplex
$\left({\TZ}^q_N(X, \bullet ;m), \partial^N\right)$ of
$\left({\TZ}^q(X, \bullet;m), \partial \right)$. We shall call
${\TZ}^q_N(X, \bullet ;m)$ the {\sl ``normalized additive cycle complex''}
for any given modulus condition.
\end{defn}
We define the {\sl normalized} additive higher Chow groups of $X$ by
${\TH}^q_N(X, n ;m) = H_n\left({\TZ}^q_N(X, \bullet ;m)\right)$.

The point of using this normalized complex is the following lemma regarding 
its interaction with $\delta_{alt}$ in the previous section.

\begin{lem}\label{reason for normalization}
$(1)$ For all $i \in \{ 1, \cdots, n\}$, we have 
$\delta_i (\TZ^q_N(X, n;m)) \subset \TZ^{q+1}_N(X, n+1;m)$. \\
$(2)$ For ${\partial}^N = \partial^{\infty}_1$, we have
$\delta_{alt} {\partial}^N + {\partial}^N \delta_{alt} = 0$ on
$\TZ^*_N(X, \bullet ;m)$.

\end{lem}

\begin{proof}$(1)$ follows immediately from
~\eqref{eqn:partialF}. To prove $(2)$, we have for any in $\TZ^q_N(X, n;m)$, 
\[
\begin{array}{lll}
\delta_{alt} {\partial}^N & = & 
\delta_{alt} \partial^{\infty}_1 \\
& = & \stackrel{n+1}{\underset {i =1} \sum} (-1)^i \delta_i 
\partial^{\infty}_1 \\
& = & \stackrel{n+1}{\underset {i =1} \sum} (-1)^i
\partial^{\infty}_1 \delta_{i+1} \\
& = & \partial^{\infty}_1 \left(\stackrel{n+1}{\underset {i =1} \sum} (-1)^i
\delta_{i+1}\right) \\
& = & - \partial^{\infty}_1 \delta_1 -
\partial^{\infty}_1 \left(\stackrel{n+2}{\underset {i =2} \sum} (-1)^i
\delta_{i}\right) \\
& = & - \partial^{\infty}_1 \left(\stackrel{n+2}{\underset {i =1} \sum} (-1)^i
\delta_{i}\right) \\
& = & - \partial^N \delta_{alt},
\end{array}
\]
where the third equality holds from ~\eqref{eqn:deltaF} and the fifth one
follows from ~\eqref{eqn:partialF}. This finishes the proof of $(2)$. 
\end{proof}

It was shown in \cite[Theorem 4.4.2]{Bl3} that the normalized cycle complex 
for the usual higher
Chow groups is quasi-isomorphic to the original cycle complex. This prompts one
to ask the following.
\begin{ques}\label{ques:QI}
Is the natural inclusion $\left({\TZ}^q_N(X, \bullet ;m), \partial^N\right)
\inj \left({\TZ}^q(X, \bullet;m), \partial \right)$ of complexes a
quasi-isomorphism?
\end{ques}
Our next goal is to show that the answer to this Question \ref{ques:QI} is 
indeed
affirmative for the modulus condition $M_{sum}$. We also derive some
crucial consequences of this for the additive higher Chow groups.
Although we can not prove this for the other two modulus conditions at
this moment for certain technical reasons, we definitely expect this to be 
the case for all modulus conditions. 

For $n \ge 1$, let $C(n-1)= \oplus_q \un{\TZ}^q (X, n; m)$. 
Let $D(n-1) \subset C(n-1)$ be the subgroup of degenerate cycles. Let 
$C(n-1)^{0} = \bigcap_{i=1} ^{n-1}\ker (\partial_i ^{0}) 
\subset C(n-1)$.
Note that $\bigoplus _{n \geq 1} C(n-1)^{0}$ is a subcomplex of 
$\bigoplus _{n \geq 1} C(n-1)$ with respect to the boundary map 
$\partial = \stackrel{n-1}{\underset {i =1} \sum} (-1)^i 
\partial_i ^{\infty}$. 
We shall write this subcomplex as $\left(C(*), \partial \right)$. 
\begin{prop}\label{prop:direct}
$C(n-1) = D(n-1) \oplus C(n-1)^{0}$. Thus, we can identify 
$$\bigoplus_q {\TZ}^q (X, n;m) = \bigoplus_q
\underline{\TZ} ^q (X, n;m)/ \underline{\TZ} ^q (X, n;m)_{\rm degn}$$ with its 
subgroup $C(n-1)^{0}=\bigcap_{i=1} ^{n-1} \ker (\partial_i ^{0})$.
\end{prop}
\begin{proof}
We first prove that $C(n-1) = D(n-1) + C(n-1)^{0}$. 
Let $z \in C(n-1)$, and suppose that $\partial_{r+1} ^{0} (z) = 
\cdots = \partial_{n-1} ^{0} (z) = 0$, to use a backward induction 
argument on the subscripts. 
Let $z' := z - \pi_r \circ \partial_r ^{0} (z)$, where $\pi_r$ is the 
pull-back via the projection $(x, t, y_1, \cdots, y_{n-1}) 
\mapsto (x, t, y_1, \cdots, \widehat{y_i}, \cdots, y_{n-1})$. One easily 
checks that $\partial _{\nu} ^j \circ \pi_{\nu} = 1$. Hence, 
$\partial_r ^{0} (z') = \partial_r ^{0} (z) - 
\partial_r ^{0} (\pi \circ \partial_r ^{0} (z)) = 
\partial_r ^{0} (z) - \partial_r ^{0} (z) = 0$. For $s>r$, one 
first easily sees that $\partial_{\mu -1} ^k \circ \partial_{\nu}^j = 
\partial_{\nu} ^{j} \circ \partial _{\mu} ^k$ and $\pi_{\nu} \circ 
\partial_{\mu -1} ^j = \partial _{\mu} ^j \circ \pi _{\nu}$ for 
$\nu < \mu$ from the standard cubical identities. Hence, using these two and 
the induction hypothesis that $\partial _s ^{0} (z) = 0$, we obtain 
$\partial_s ^{0} (z') = \partial _s ^{0} (z) - 
\partial_s ^{0} ( \pi_r \circ \partial _r ^{0} (z)) = 
0 - \pi_r \circ \partial_{s-1} ^{0}\circ \partial_r ^{0} (z) = 
- \pi_r \partial_r ^{0} \partial_s ^{0} (z) = 0.$ Hence, by 
induction we may write $z$ as a sum of elements in $D(n-1)$ and 
$C(n-1)^{0}$.

To prove that the sum is direct, let $r$ be the minimum such that there is a 
non-zero $z \in C(n-1)^{0}$ with $z = \sum_{i=1} ^r \pi_i w_i$ for some 
$w_i$. Since $\partial_r ^{0} = 0$ and $\partial_r ^{0} \circ 
\pi_r = 1$, by applying $\partial_r ^{0}$ to $z$, we obtain
\[
w_r = - \sum_{i=1} ^{r-1} \partial_r ^{0} \circ \pi_i w_i = 
- \sum_{i=1} ^{r-1} \pi_i \left( \partial_{r-1} ^{0} w_i \right),
\]
where for the second equation we used the cubical identity $\pi_{\nu} \circ 
\partial_{\mu-1} ^j=  \partial_{\mu} ^j \circ \pi_{\nu}$ for $\nu < \mu$. 
Hence, by plugging this back into the expression of $z$, we have
\[
z = \sum_{i=1} ^{r-1} \pi_i w_i + \pi_r w_r = \sum_{i=1} ^{r-1} \pi_i 
\left(w_i - \pi_{r-1} \circ \partial_{r-1}^{0} w_i \right),
\]
where for the second equation we used the cubical identity 
$\pi_{\nu} \circ \pi_{\mu -1} = \pi_{\mu} \circ \pi_{\nu}$ for $\nu < \mu$. 
This contradicts the minimality of $r$. Hence the sum is direct. This proves 
the proposition.
\end{proof}
\begin{thm}\label{thm:QIsum}
Let $X$ be a smooth projective variety. Then for the modulus condition
$M_{sum}$, the natural inclusion ${\TZ}^q_N(X, \bullet ;m) 
\inj {\TZ}^q(X, \bullet;m)$ of complexes is a quasi-isomorphism.
\end{thm}
\begin{proof} In this proof, we again temporarily 
identify $\square$ with $\A^1$
via the isomorphism of ~\eqref{eqn:SA}. Recall that this gives the isomorphism 
$\left(\P^1, \{0, 1, \infty \}\right) \cong
\left(\P^1, \{1, \infty , 0 \}\right)$. 
Thus we need to show that the inclusion
$\left({\TZ}^q_N(X, \bullet ;m), \partial^N\right)
\inj \left({\TZ}^q(X, \bullet ;m), \partial \right)$ is a quasi-isomorphism, 
where $\partial = \sum_{i=1}^{n-1} (-1)^i (\partial^0_i - \partial^1_i)$.
We make the appropriate identification for $C(*)^0$ as well.

Using Proposition~\ref{prop:direct}, we only need to show that the
inclusion ${\TZ}^q_N(X, \bullet ;m) \inj C(*)^0$ is a quasi-isomorphism.
In particular, we can assume that for all cycles 
$\alpha \in {\TZ}^q(X, n+1 ;m)$,
one has $\partial^1_i (\alpha) = 0$ for $1 \le i \le n$.
For $i \ge 0$, let 
\[C(*)_{i} ^0 =
\{\alpha \in C(*)^0| \partial^0_j (\alpha) = 0 \ {\rm for } \ j \ge i+2\}.
\] 
Let $\tau_j \in {\rm \mathbb{P}erm}_n$ be the permutation such that
\[
\tau_j (i) =
\left\{\begin{array}{ll}
i \ \ \ \ \ \ \ \mbox{if $i < j$} \\
i-1 \ \ \mbox{if $i >j$} \\
n \ \ \ \ \ \ \ \mbox{if $i = j$}.
\end{array}
\right.
\]
For any $0 \le i \le n$ and admissible cycle $\alpha \in 
{\TZ}^q(X, n+1 ;m)$, let 
\[
H^n_i (\alpha) = (-1)^{n+1-i} {\tau}_{n+1-i} \cdot
W^X_n (\alpha). 
\]
By Proposition~\ref{prop:addm}, $W^X_n(\alpha) \in
{\TZ}^q(X, n+2 ;m)$. Since $\tau$ clearly preserves the admissibility, we 
get a homomorphism 
$H^n_i : {\TZ}^q(X, n+1 ;m) \to {\TZ}^q(X, n+2 ;m)$.
Now we use the computations of the boundaries of $W^X_n (\alpha)$ in
Lemma~\ref{lem:intersection} to see that $H^n_i$ restricts to a map
\begin{equation}\label{eqn:Homtp}
H^n_i : C(*)_n ^0\to C(*)_{n+1}^0.
\end{equation}
Define $\psi_0 : C(*)_n ^0\to C(*)_n^0$ by
$\psi_0 = {\rm Id} - (\partial \circ H_0 + H_0 \circ \partial)$ and we 
inductively
define $\psi_{i+1} =  
\left({\rm Id} - (\partial \circ H_n + H_n \circ 
\partial)\right) \circ \psi_i$.
In particular, we have  
\[
\psi: = \psi_n = \left({\rm Id} - (\partial \circ H_n + H_n 
\circ \partial)\right)
\circ \cdots \circ \left({\rm Id} - 
(\partial \circ H_0 + H_0 \circ \partial)\right),
\]   
where $\partial = \stackrel{\ }{\underset {\ } \sum} (-1)^i 
\partial_i ^{0}$.
Thus $\psi$ defines an endomorphism of $C(*)^0$ which is homotopic to identity.
Furthermore, Lemma~\ref{lem:intersection} implies that the restriction of
$\psi$ on $C(*)_{0}^0 = {\TZ}^q_N(X, n+1 ;m)$ is identity.
The proof of the theorem will now be complete from the following.\\
{\bf Claim : } \emph{
$\psi_i \left(C(*)_{n-i-1}^0\right) \subset C(*)_{n-i-2}^0$ for 
$0 \le i \le n-2$.}

We prove it for $\psi_0$ and other cases are exactly similar and can be proved
inductively. We have
\begin{eqnarray*}
H_0 \circ \partial & = & H^{n-1}_0 \circ \partial \\
& = & (-1)^n W^X_{n-1} \circ \partial \\
& = & (-1)^n  \sum_{i=1} ^{n-1} (-1)^i [W^X_{n-1} \circ
\partial^0_i] + (-1)^{2n} [W^X_{n-1} \circ \partial^0_n].
\end{eqnarray*}
On the other hand, we have
\[
\begin{array}{lll}
\partial \circ H_0 & = & \partial \circ H^n_0 \\
& = & (-1)^{n+1} \partial \circ W^X_n \\
& = & (-1)^{n+1} \stackrel{n+1}{\underset {i =1} \sum} (-1)^i \partial^0_i 
\circ W^X_n \\
& = & (-1)^{n+1}\stackrel{n-1}{\underset {i =1} \sum} (-1)^i [\partial^0_i 
\circ W^X_n] \\
& & + (-1)^{2n+1} \partial^0_i \circ W^X_n + (-1)^{2n+2}  
\partial^0_i \circ W^X_n \\ 
& = & (-1)^{n+1}\stackrel{n-1}{\underset {i =1} \sum} (-1)^i [W^X_{n-1} \circ
\partial^0_i] \\
& & + (-1)^{2n+1} id + (-1)^{2n+2} {\rm Id}\\
& = & (-1)^{n+1}\stackrel{n-1}{\underset {i =1} \sum} (-1)^i [W^X_{n-1} \circ
\partial^0_i], \\
\end{array}
\] 
where the fifth equality follows from Lemma~\ref{lem:intersection} and
~\eqref{eqn:WX0}.
Thus we get $H_0 \circ \partial + \partial \circ H_0 = 
W^X_{n-1} \circ \partial^0_n$. However, we see from 
Lemma~\ref{lem:intersection} again that $\partial^0_n \circ 
W^X_{n-1} \circ \partial^0_n = \partial^0_n$. This shows that
$\psi_0 \left(C(*)^0_{n-1}\right) \subset C(*)_{n-2}^0$. This proves the claim
and the theorem.
\end{proof}

Although we are unable to prove Theorem~\ref{thm:QIsum} for the modulus
conditions $M_{sup}$ and $M_{ssup}$ in this paper, the following result gives
a partial answer to Question~\ref{ques:QI} in these cases.
\begin{thm}\label{thm:QIP}
Let $X$ be a smooth projective variety over $k$. Then for any modulus condition
and for any $n, m \ge 1$, the natural map ${\TH}^q_N(X, n ;m) 
\to {\TH}^q(X, n;m)$ is injective. In particular, the map
${\TH}^n_N(k, n ;m) \to {\TH}^n(k, n;m)$ is an isomorphism.
\end{thm}
\begin{proof}
Let $X^{0} = X \times \G_m$. We have seen
before that forgetting the modulus condition, one has a natural inclusion
of cubical objects 
\[
\left(\un n \mapsto (\un{\TZ}^q(X, n; m), {\partial}^{\epsilon}_{i})
\right) \to 
\left(\un n \mapsto (\un{z}^q(X^{0}, n), {\partial}^{\epsilon}_{i})
\right),
\]
where the right side is the cubical object for the higher Chow group.
This induces a natural inclusion of chain complexes
$i_X : \un{\TZ}^q(X, \bullet ; m) \hookrightarrow \un{z}^q(X^{0}, \bullet)$
such that a cycle $z \in \un{\TZ}^q(X, \bullet ; m)$ is degenerate if and 
only if $i_X(z) \in \un{z}^q(X^{0}, \bullet)$ is so. 
Considering the normalized version of these complexes, we get a commutative
diagram
\begin{equation}\label{eqn:NonAd}
\xymatrix{
{\TZ}^q_N(X, n; m) \ar[r] \ar[d] & {z}^q_N(X^{0}, n-1) \ar[d] \\
{\TZ}^q(X, n; m) \ar[r] & {z}^q(X^{0}, n-1)}
\end{equation}
which is clearly Cartesian and all arrows are injective.
Let 
\[
{\ov \TZ}^q(X, n;m) = 
{\rm Image}\left({z}^q_N(X^{0}, n-1) \oplus {\TZ}^q(X, n; m) \to
{z}^q(X^{0}, n-1)\right).
\]
Then we get the exact sequence of complexes
\[
\xymatrix{ 
0 \to {\TZ}^q_N(X, \bullet; m) \ar[r] & {z}^q_N(X^{0}, \bullet) \oplus 
{\TZ}^q(X, \bullet; m) \ar[r] & {\ov \TZ}^q(X, \bullet; m) \to 0.} 
\]
The theorem would follow if we show that the map
$$H_n\left({z}^q_N(X^{0}, \bullet)\right) \to 
H_n\left({\ov \TZ}^q(X, \bullet; m)\right)$$ is injective for all $n \ge 1$.
For this, we consider the inclusions
\[
{z}^q_N(X^{0}, \bullet) \inj {\ov \TZ}^q(X, \bullet; m) \inj
{z}^q(X^{0}, \bullet).
\]
By \cite[Theorem~4.4.2]{Bl3}, the composite map is a quasi-isomorphism.
Hence the map $H_n\left({z}^q_N(X^{0}, \bullet)\right) \to  
H_n\left({\ov \TZ}^q(X, \bullet; m)\right)$ is in fact split injective. This 
finishes the proof of the theorem.

To prove the isomorphism of the additive higher Chow groups of zero cycles,
we simply note that the inclusion map ${\TZ}^n_N(k, n; m) \inj 
{\TZ}^n(k, n; m)$ is in fact an isomorphism and hence the map 
${\TH}^n_N(k, n; m) \to {\TH}^n(k, n; m)$ is surjective too.
\end{proof}
For our purposes, the following is the main application of
Theorems~\ref{thm:QIsum} and ~\ref{thm:QIP}.  
\begin{cor}\label{cor:alt0}
For the modulus condition $M_{sum}$, the map $\delta_{alt}$ defines a
homomorphism 
\[
\delta_{alt} : {\TH}^q(X, n; m) \to {\TH}^{q+1}(X, n+1; m)
\]
satisfying $\delta^2_{alt} = 0$.
\\
For the modulus condition $M_{ssup}$, $\delta_{alt}$ defines a 
homomorphism 
\[
\delta_{alt} : {\TH}^n(k, n; m) \to {\TH}^{n+1}(k, n+1; m)
\]
satisfying $\delta^2_{alt} = 0$.
\end{cor}
\begin{proof} This follows immediately by combining 
Lemmas~\ref{lem:del2}, ~\ref{lem:pdelta}, 
Lemma \ref{reason for normalization}, and Theorem~\ref{thm:QIsum}.
The second statement about the additive zero cycles follows in the same way,
where we now use Theorem~\ref{thm:QIP} in place of Theorem~\ref{thm:QIsum}.
\end{proof}
We complete our study of CDGA structures on the additive higher Chow groups
with the following main result of this section.
\begin{thm}\label{thm:DGAsum}
Let $X$ be a smooth projective variety over a field $k$.
Then for the modulus condition $M_{sum}$,
the additive higher Chow groups $\left(\TH(X), \wedge, \delta_{alt}\right)$
form a graded-commutative differential graded algebra, where $\delta_{alt}$ is
the graded derivation for the wedge product $\wedge$. This derivation
commutes with the pull-back and push-forward maps of additive higher Chow
groups. 
\end{thm}
\begin{proof}
This follows directly from Corollary~\ref{cor:wedge-product**},
Proposition~\ref{prop:derivation} and Corollary~\ref{cor:alt0}.
The commutativity property of $\delta_{alt}$ with the pull-back and 
push-forward maps follows from the similar property of $\delta$ in
Theorem~\ref{thm:preDGA}.
\end{proof}
\begin{remk}\label{remk:delalt}
It follows from the above results that for the modulus condition
$M_{sum}$, $\TH(X)$ has two differentials, namely, $\delta$ of 
Theorem \ref{thm:preDGA}  and $\delta_{alt}$ of Theorem \ref{thm:DGAsum}.
But one can in fact show that the latter is just a finitely many copies
of the former. To see this, one needs to know that each 
$\sigma \in {\rm \mathbb{P}erm}_n$ 
acts on $\TH^q(X, n+1, m)$ as $\sgn(\sigma) \cdot {\rm Id}$.
Using this, one can also see that while $\delta$ is analogous to the
exterior derivation of K{\"a}hler differentials, $\delta_{alt}$ 
corresponds to the exterior derivation of the Hochschild homology. These
differential operators are related by a kind of anti-symmetrizer maps defined
via permutations ({\sl cf.} \cite{Loday}).
\end{remk}
\begin{remk}\label{remk:delaltS}
We see from Corollary~\ref{cor:wedge-product**} and 
Proposition~\ref{prop:derivation} that 
$\left(\TH(X), \wedge, \delta_{alt}\right)$ is a 
differential graded algebra also for the modulus condition $M_{ssup}$
if the answer to Question~\ref{ques:QI} is affirmative. As we have already
remarked before, this is very much expected and a proof of this should be
available in a near future. For now, it does follow from 
Corollary~\ref{cor:alt0} that the groups 
$\left({\{{\TH}^n(k,n;m)\}}_{n \ge 1}, \wedge, \delta_{alt}\right)$ form
a CDGA.
\end{remk}

\subsection{Motivic cyclic homology}\label{subsection:MCH}
We end this section by showing how one can use the normalized additive
Chow groups and Theorem~\ref{thm:QIsum} to define a motivic version of
the cyclic homology of A. Connes. 
We see from Lemmas~\ref{lem:del2} and \ref{reason for normalization}
that there is a bicomplex
\begin{equation}\label{eqn:bicomplex}
\xymatrix{ 
\vdots \ar[d]_{\partial^N} & \vdots \ar[d]_{\partial^N} & \vdots 
\ar[d] _{\partial^N}  & \\
{\TZ}^{q+2}_N (n+2) \ar[d] _{\partial^N} & 
{\TZ}^{q+1}_N (n+1) \ar[l]_{\delta_{alt}} \ar[d]_{\partial^N} & 
{\TZ}^q_N (n) \ar[l] _{- \delta_{alt}} 
\ar[d]_{\partial^N} &\cdots \ar[l]_{\delta_{alt}} \\
{\TZ}^{q+2}_N (n+1) \ar[d] _{\partial^N} & 
{\TZ}^{q+1}_N (n) \ar[l] _{- \delta_{alt}} \ar[d] _{\partial^N}
& {\TZ}^q_N (n-1) \ar[l] _{\delta_{alt}} 
\ar[d] _{\partial^N} & \cdots \ar[l] _{- \delta_{alt}} \\
\vdots & \vdots & \vdots &
}
\end{equation}
where ${\TZ}^q_N (n) := {\TZ}^q_N (X, n;m)$. 
This allows us to propose a cyclic analogue of additive cycle complex, 
regarding the additive cycle complex as the motivic analogue of the 
Hochschild complex. Let ${\TZ}(n):=\bigoplus^q {\TZ}^q (n)$. 
Note that $\partial^N$ decreases only $n$ by $1$, while $\delta_{alt}$ 
increases both $q$ and $n$ by $+1$. The above bicomplex then reads,

\begin{equation}\label{eqn:firstmixedcomplex}
\xymatrix{
\vdots \ar[d] _{\partial^N} & \vdots \ar[d] _{\partial^N} & 
\vdots \ar[d]_{\partial^N} & \vdots \ar[d]_{\partial^N}  \\
{\TZ} (4) \ar[d] _{\partial^N} & {\TZ}(3) 
\ar[l] _{\delta_{alt}} \ar[d] _{\partial^N} & {\TZ}(2) 
\ar[l]_{- \delta_{alt}} \ar[d] _{\partial^N} & {\TZ} (1) 
\ar[l]_{\delta_{alt}} \\
{\TZ}(3) \ar[d] _{\partial^N} & {\TZ} (2) \ar[l]_{- \delta_{alt}} 
\ar[d]_{\partial^N} & {\TZ} (1) \ar[l]_{\delta_{alt}} & \\
{\TZ}(2) \ar[d] _{\partial^N} & {\TZ} (1)\ar[l]_{\delta_{alt}} & & \\
{\TZ}(1). & & &}
\end{equation}

Let $\sB \sZ$ be this bicomplex. This is a mixed complex in the sense of 
A. Connes ({\sl cf.} \cite[p. 79]{Loday}). We apply the usual formalism of 
mixed complexes to $\sB \sZ$. By definition, its homology $H_n(\sB \sZ)$
is the homology of the first column and this is just the additive higher 
Chow groups ${\TH}^*_N(X, n-1; m)$. 
Its {\sl cyclic homology} $HC_n(\sB \sZ)$ is the homology 
$H_n(Tot (\sB \sZ))$ of the total complex. Notice that the bicomplex 
\eqref{eqn:bicomplex} itself is not a mixed complex, but since 
$\sB \sZ$ is the direct sum of these, the groups $H_n(\sB \sZ)$ and 
$HC_n(\sB  \sZ)$ have natural decompositions.

We define the {\sl motivic cyclic homology} ${\CH}^* (X, n;m)$ as the cyclic 
homology $HC_n (\sB \sZ)$ of the bicomplex $\sB \sZ$. 
In this analogy, we could as well call our additive higher Chow groups as
{\sl motivic Hochschild homology}. The group 
${\CH}^q (X, n;m)$ is the direct summand of ${\CH}^* (X,n;m)$ that comes 
from the diagonal of \eqref{eqn:bicomplex} that contains ${\TZ}^q_N(n)$ in 
the first column. Note that, despite the double index $(q, n)$, the group 
${\CH}^q (X, n;m)$ contains cycles not just from ${\TZ}^q_N(n)$, but also from
\[
\bigoplus _{i \geq 0} ^{\min\left\{ q, \left\lfloor 
\frac{n-1}{2}\right\rfloor \right\}}{\TZ}^{q-i}_N (n-2i).
\]
Following the formalism of mixed complexes ({\sl cf.} \cite[2.5.3]{Loday}), 
we have the long exact sequence of complexes
\[
0 \to ({\TZ}(*), \partial^N) \overset{I}{\to} Tot (\sB \sZ ) 
\overset{S}{\to} Tot \left( \sB \sZ [1,1] \right) \to 0.
\]
Notice that $Tot \left( \sB \sZ [1,1] \right) = 
\left(Tot (\sB \sZ )\right) [2]$. Thus, we obtain its long exact sequence, 
which is the Connes' periodicity exact sequence, that decomposes as follows:
\begin{cor}\label{cor:Connesperiodicity}
Suppose the modulus condition is $M_{sum}$. Then there is a 
Connes' periodicity exact sequence involving ${\TH}$ and ${\CH}$:
\[
\cdots \overset{B}{\to} {\TH}^q (X, n;m) \overset{I}{\to} 
{\CH}^q (X,n;m) \overset{S}{\to} {\CH}^{q-1}(X,n-2;m) \overset{B}{\to} 
\]
\enlargethispage*{100pt}
\[
\hspace*{9cm} {\TH}^{q} (X,n-1;m) \overset{I}{\to} \cdots,
\]
where the maps $I, S, B$ have bidegrees $(0,0), (-1, -2), (+1, +1)$ in 
$(q,n)$ respectively.
\end{cor}

As a consequence, we have the following motivic interpretation of the top
Hodge piece $HC^{n-1}_{n-1}(k)$ of the cyclic homology $HC_{n-1}(k)$ of the 
ground field.
\begin{cor}\label{cor:CCH_0}
Assume that char$(k) \neq 2$. Then, for any modulus condition and for 
$n,m \ge 1$, there is an isomorphism 
\[
{\CH}^n (k,n;m) \xrightarrow{\cong}
{\W}_m\Omega_{k/\Z}^{n-1} / {d {\W}_m\Omega_{k/\Z} ^{n-2}}.
\]
\end{cor}
\begin{proof} By definition, 
\[
{\CH}^n (k, n;m) = \frac{{\TZ}^n_N(k, n)}{ \partial^N {\TZ}^n_N (k, n+1) + 
\delta_{alt} {\TZ}^{n-1}_N (k, n-1)}.
\]
By Theorems~\ref{thm:0-cycle} and ~\ref{thm:QIP},
we have ${\TZ}^{n}_N (k,n) / \partial^N {\TZ}^n_N (k, n+1) \simeq 
{\W}_m\Omega_{k/\Z} ^{n-1}$. On the other hand, combining these two theorems 
with Theorem~\ref{thm:QIP} and Remark~\ref{remk:delaltS}, we see that
the elements of the group 
$\delta_{alt} {\TZ}^{n-1}_N(k, n-1)$ are exact de Rham-Witt forms. This 
finishes the proof.
\end{proof}
Further study of the above motivic cyclic homology using algebraic cycles
and its applications to additive higher Chow groups will be taken up in
a sequel.

\section{Remarks and computations}\label{section:remarks}
\subsection{Moving modulus conditions}
We saw that $M_{sum}$ and $M_{ssup}$ apparently have much better
structural behavior than the modulus condition $M_{sup}$
studied in \cite{KL, P1}, and this makes the former better suited for
being a motivic cohomology. 
On the other hand, in the main theorem of \cite{P1}, the regulators on 
$1$-cycles were defined with the modulus condition $M_{sup}$. 
Although we have seen that this regulator map does exist and has good
properties with the modulus condition $M_{ssup}$, its construction
doesn't automatically generalize to the groups with $M_{sum}$. 
So, one may ask the following :
\begin{ques}\label{ques:DML}
Given an $M_{sum}$-admissible cycle $\xi$ with $\partial \xi = 0$, can one 
find $M_{sup}$-admissible cycles $\eta$ and $\Gamma$ such that 
$\xi = \eta + \partial \Gamma$?
\end{ques}
A positive answer to this question will immediately solve one part of
Conjecture~\ref{conj:QI}.
This is a kind of {\sl deeper moving lemma} than we have proved in this
paper. This moving lemma allows one to move the modulus as well as the proper
intersection property when we
move a cycle. On the other hand, the moving lemma of this paper does not
allow changing the modulus conditions. We expect the answer to the above
question to be much harder.

\subsection{Examples}
\begin{exm}\label{exm:first example}
We give an example where the homotopy used in \cite{Bl1,Le} doesn't preserve 
the modulus conditions for additive Chow groups of quasi-projective 
varieties, even for the simplest possible cases. 

Take $X= \mathbb{A} ^1 _k$ and $n=1$, so, we are interested in admissible 
cycles in $X \times {\wt B}_1 = X \times \mathbb{A}^1 _k$. Admissible 
closed subvarieties $Z \subset X \times \mathbb{A}^1 _k$ are given by the 
condition $Z \cap (X \times \{0 \}) = \emptyset$. Let $\mathbb{G}_{a,k}= 
\mathbb{A}^1 _k$ act on $X$ by translation, and take its function field 
$K = k(s)$, $s$ transcendental over $k$. Take the line 
$\phi: \square^1 _K \to \mathbb{G}_{a,K}$ defined by 
$y \mapsto s y/(y-1)$ that sends $0$ to $0$ and $\infty$ to the $k$-generic 
point $s$ of $\mathbb{G}_{a,k}$, which is $K$-rational in 
$\mathbb{G}_{a, K}$.  

Take $Z$ given by the ideal $(xt + 1) \subset k[x,t]$, which is in 
${\TZ}^1 (\mathbb{A}^1, 1; m)$. Then, $Z_K$ is given by 
$(xt +1) \subset K[x,t]$ and ${{\rm pr}'}^* Z_K$ is given by 
$(xt +1) \subset K[x,t,y/({y-1})]$. Pulling back through $\mu_{\phi}$, we get 
$\left(x+s{y}/({y-1})\right)t + 1 = 0$. This is the equation for our 
homotopy variety $Z'$. Rewriting it as $1-y = t ((y-1)x + sy)$, we see that 
it doesn't satisfy any of the given modulus conditions 
$M_{sum}, M_{sup}, M_{ssup}$. 
For instance, for a given $m\geq 1$, we need $1-y$ to be divisible by at 
least $t^{1+m}$ where $m \geq 1$, which is obviously false in this case. 
Hence $Z' \not \in {\TZ} ^1 (\mathbb{A}^1 _K, 2; m)$.
\end{exm}

\begin{exm}
Recall from Remark~\ref{remk:n=1} that if $X$ is projective, then admissible 
cycles in $X \times {\wt B}_1 = X \times \mathbb{A}^1$ have a very simple 
description : an admissible irreducible closed subvariety $Z$ should be of 
the form $Y \times \{ * \} \subset X \times \mathbb{A}^1$ for some closed 
subvariety $Y \subset X$, and a closed point $\{ * \} \not = \{ 0 \}$ of 
$\mathbb{A}^1$. This variety obviously satisfies all of the modulus 
conditions. 

Note that the admissible variety $Z$ in Example~\ref{exm:first example} is 
not of the form $Y \times \{ * \}$: this happens because $X=\mathbb{A}^1 _k$ 
is not complete.

These two examples show that the additive higher Chow groups of 
quasi-projective varieties may have a bit more complicated structures than 
those of projective varieties. The authors don't know yet what 
extra-structures one can expect in general for this quasi-projective case.
\end{exm}

\subsection{A computation}
We finish the paper with a calculation of some additive higher Chow groups,
which the authors had worked out while working on this paper. The following 
extends \cite[Theorem 6.4, p. 153]{BE2} to affine spaces. 

\begin{thm}
Let $M$ be the modulus condition $M_{sum}$, $M_{sup}$, or $M_{ssup}$. Let 
$X = \mathbb{A} ^r _k$, and let $m=1$. Then, the additive higher Chow groups 
of zero-dimensional cycles of $X$ are the absolute K\"ahler differentials of 
$k$: 
\[
\TH^{r+n+1} (X, n; 1) \simeq \Omega_{k/\mathbb{Z}} ^{n-1}.
\]
\end{thm}
\begin{remk}
Note that, although it looks similar, this theorem does not imply that 
additive higher Chow groups have $\mathbb{A}^1$-homotopy invariance. 
For the structure morphism $\mathbb{A}^r _k \to {\rm Spec} (k)$, the 
pull-backs of 0-cycles on ${\rm Spec} (k) \times {\wt B}_n$ to 
$X \times \widetilde{B}_n$ are $r$-cycles,
not $0$-cycles.
\end{remk}
\begin{proof}
The proof is very similar to that of \cite[Theorem 6.4, p. 153]{BE2}. For a 
closed point $p \in X \times {\wt B}_n$ that does not intersect the faces and 
the divisor $\{ t= 0 \}$, we define a homomorphism by setting
\[
\psi (p) := {\rm Tr}_{k(p)/k} \left( \frac{1}{t} \frac{dy_1}{y_1} \wedge 
\cdots \wedge \frac{dy_{n-1}}{y_{n-1}}\right)(p) \in 
\Omega_{k/\mathbb{Z}} ^{n-1}.
\]
In other words, we ignore the coordinate of $X$. This defines a homomorphism 
$\psi : {\TZ}^{r+n+1} (X, n ; 1) \to \Omega_{k/\mathbb{Z}} ^{n-1}$.

{\bf Claim (1):} \emph{The composition 
\[
\psi \circ \partial : {\TZ}^{r+n+1} (X, n+1; 1) \overset{\partial}{\to} 
{\TZ}^{r+n+1} (X, n; 1) \overset{\psi}{\to} \Omega_{k/\mathbb{Z}} ^{n-1}
\]
is zero.}
\\
It just follows from \cite[Proposition 6.2, p. 150]{BE2}. 

{\bf Claim (2):} \emph{ Any two closed admissible points 
$p, p' \in X \times {\wt B_n}$ for which 
only the coordinates of $X$ differ are equivalent as additive higher Chow 
cycles.}
\\
Abusing notations, write 
$p = (a, b, s_1, \cdots, s_{n-1}), p' = (a', b, s_1, \cdots, s_{n-1})$, where 
$a, a'$ are closed points of $X$, where $b \not = 0$, $s_i \not = 0, \infty$. 
Consider a parametrized line 
\[
C=\left\{ \left( a\frac{y}{y-1} + 
a' \left( 1 - \frac{y}{y-1} \right), b, y, s_1, \cdots, 
s_{n-1} \right) \in X \times {\wt B}_{n+1} | \ \ y \in \square^1\right\}.
\]
This $1$-cycle satisfies all the modulus conditions 
$M_{sum}, M_{sup}, M_{ssup}$ having $b \not = 0$, and it intersects all faces 
properly having constant $y_i$-coordinate values $s_i$.

By direct calculations, $\partial_1 ^0 (C) = p', 
\partial _1 ^{\infty} (C) = p$, and $\partial_i ^\epsilon(C) = 0$ for 
$i \geq 2$ 
and $\epsilon \in \{ 0, \infty \}$. Hence, $\partial (C) = p' - p$ proving 
Claim (2).

Given Claim (2), by \cite[Proposition 6.3]{BE2} and the rest of the arguments 
of  \cite[Theorem 6.4]{BE2}, this theorem follows.
\end{proof}

We remark that the same arguments work for any variety $X$ as long as we can 
prove Claim (2). In particular, for any connected union of affine spaces, 
irreducible or not, we can conclude the same results.

\begin{ack}
The authors would like to thank Spencer Bloch and Marc Levine for some
valuable comments. JP would like to thank TIFR and KAIST for their 
hospitality and reduced teaching loads, and Juya for supports throughout 
this work.

For this research, JP was partially supported by Basic Science Research 
Program through the National Research Foundation of Korea (NRF) funded by the 
Ministry of Education, Science and Technology (2009-0063180).
\end{ack}


\begin{thebibliography}{99}

\bibitem{Bl1} Bloch, S., {\sl Algebraic cycles and higher $K$-theory\/},
Adv. Math., \textbf{61}, (1986), 267-304.

\bibitem{Bl2} Bloch, S., {\sl The moving lemma for higher Chow groups\/}, 
J. Algebraic Geom., \textbf{3}, (1994), no. 3, 537-568.

\bibitem{Bl3} Bloch, S., {\sl Lecture notes on higher Chow groups\/},
available at Bloch's Homepage, University of Chicago.  

\bibitem{BE1} Bloch, S., Esnault, H., {\sl An additive version of higher Chow 
groups\/}, Ann. Sci. \'Ec. Norm. Sup\'er., (4) \textbf{36}, (3), (2003),
463-477.

\bibitem{BE2} Bloch, S., Esnault, H., {\sl The additive dilogarithm\/},
Doc. Math., J., DMV \textbf{Extra Vol.}, (2003), 131-155.

\bibitem{Chow} Chow, W.-L., {\sl On equivalence classes of cycles in an 
algebraic variety\/}, Ann. of Math., (2) \textbf{64}, (1956), 450-479.

\bibitem{FGA} Fantechi, B., G\"ottsche, L., Illusie, L., Kleiman, S. L., 
Nitsure, N., Vistoli, A.,
{\sl Fundamental algebraic geometry. Grothendieck's FGA explained\/},
Mathematical Surveys and Monographs, {\bf 123},
American Mathematical Society, Providence, RI, 2005.

\bibitem {F} Fulton, W., {\sl Intersection theory\/},
Second Edition, Ergebnisse der Mathematik und ihrer Grenzgebiete 3,
Folge. A Series of Modern Surveys in Mathematics, {\bf 2},
Springer-Verlag, Berlin, 1998.

\bibitem {EGA4} Grothendieck, A., {\sl \'El\'ements de g\'eom\'etrie 
alg\'ebrique. IV. \'Etude locale des sch\'emas et des morphismes de 
sch\'emas. IV\/}., Inst. Hautes \'Etudes Sci. Publ. Math., {\bf 32}, (1967).

\bibitem{GN} Guill\'en, F., Navarro Aznar, V., 
{\sl Un crit\`ere d'extension des foncteurs 
d\'efinis sur les sch\'emas lisses\/},
Publ. Math. Inst. Hautes \'Etudes Sci., {\bf 95}, (2002), 1-91.

\bibitem {Hart} Hartshorne, R., {\sl Algebraic geometry\/}, 
Graduate Texts in Mathematics, No. {\bf 52},
Springer-Verlag, New York-Heidelberg, 1977.

\bibitem{Hesselholt} Hesselholt L., {\sl $K$-theory of truncated polynomial 
algebras\/}, Handbook of K-Theory, {\bf 1},
Springer-Verlag, Berlin, 2005,  71-110.

\bibitem{HeMa} Hesselholt, L., Madsen, I., {\sl The $K$-theory of nilpotent 
endomorphisms\/},  
Homotopy methods in algebraic topology (Boulder, CO, 1999),  
127-140, Contemp. Math., {\bf 271}, Amer. Math. Soc., Providence, RI, 2001. 

\bibitem{KL} Krishna, A., Levine, M., {\sl Additive higher Chow groups of 
schemes\/}, J. Reine Angew. Math., \textbf{619,} (2008) 75-140.

\bibitem{Le0} Levine, M., {\sl Bloch's higher Chow groups revisited\/},
$K$-theory (Strasbourg, 1992). Ast\'erisque, {\bf 226}, (1994), 10, 235-320.

\bibitem{Le} Levine, M., {\sl Mixed Motives\/},
Mathematical Surveys and Monographs, {\bf 57},
American Mathematical Society, Providence, RI, 1998.
 
\bibitem{Loday} Loday, J.-L. {\sl Cyclic homology\/}, 
Second edition. Grundlehren der Mathematischen Wissenschaften, {\bf 301},
Springer-Verlag, Berlin, 1998. 

\bibitem{P1} Park, J., {\sl Regulators on additive higher Chow groups\/},
Amer. J. Math., \textbf{131,} (2009) no. 1, 257-276.

\bibitem{P2} Park, J., {\sl Algebraic cycles and additive dilogarithm\/},
Int. Math. Res. Not., {\bf 18}, (2007), Article ID rnm067.

\bibitem{Ro} Roberts, J., {\sl Chow's moving lemma\/},
 Appendix 2 to: ``Motives'', Algebraic geometry, Oslo, 1970, 
(Proc. Fifth Nordic Summer School in Math.), pp. 53-82, 
Wolters-Noordhoff, Groningen, 1972.

\bibitem{R} R\"ulling, K., {\sl The generalized de Rham-Witt complex over a 
field is a complex of zero-cycles\/}, 
J. Algebraic Geom., \textbf{16}, (2007), no. 1, 109-169.

\bibitem{T} Totaro, B., {\sl Milnor $K$-theory is the simplest part of 
algebraic $K$-theory\/}, K-theory, \textbf{6}, (2), (1992), 177-189.
\end{thebibliography}
\end{document}